\documentclass[12pt]{article}
\usepackage{amsmath, amsthm, amssymb}
\usepackage[mathscr]{eucal}
\usepackage[all]{xy}

\font \boldfrak eufb10 at 12 pt

\def\bfr#1{\hbox{\boldfrak #1}}

\newtheorem{proposition}{Proposition}
\newtheorem{lemma}{Lemma}

\newtheorem{definition}{Definition}

\newtheorem*{HW1-2}{HW's Proposition 1.2}
\newtheorem*{HW1-3}{HW's Corollary 1.3}
\newtheorem*{HW1-9}{HW's Lemma 1.9}
\newtheorem*{HW4-3}{HW's Proposition 4.3}
\newtheorem*{M10-1}{Murnaghan's Proposition 10.1}

\newtheorem{theorem}{Theorem}

\newtheorem{problem}{Problem}
\newtheorem{example}{Example}

\DeclareMathOperator{\Hom}{Hom}
\DeclareMathOperator{\Ad}{Ad}
\DeclareMathOperator{\Int}{Int}
\DeclareMathOperator{\tr}{tr}
\DeclareMathOperator{\disc}{disc}
\DeclareMathOperator{\Hasse}{Hasse}

\begin{document}

\def\cJ{{\cal J}}
\def\bA{{\bf A}}

\def\sG{{\mathsf{G}}}
\def\sH{{\mathsf{H}}}
\def\sT{{\mathsf{T}}}
\def\boxit#1{\vbox{\hrule
\hbox{\vrule\kern1pt
\vbox{\kern1pt#1\kern1pt}
\kern-3pt\vrule}\hrule}}
\def\On{{\rm O}}
\def\SOn{{\rm SO}}
\def\gH{{\frak H}}
\def\J{{\cal J}}
\def\fJ{\mathfrak {J}}
\def\fW{\mathfrak{W}}

\def\ladic{\overline{\Q}_\ell}
\def\cg{{\frak c}}
\def\s{{\frak s}}
\def\mg{{\frak m}}
\def\h{{\frak h}}
\def\sg{{\frak s}}
\def\tg{{\frak t}}
\def\kbb{\mbox{k \kern-.72em k}}
\def\kbbl{{\rm k\kern-.5em k}}
\def\kbbll{{\rm k\kern-.4em k}}
\def\ad{{\rm ad}}
\def\Ad{{\rm Ad}}
\def\ni{\noindent}
\def\bB{{\bf B}}
\def\bG{{\bf G}}
\def\bP{{\bf P}}
\def\bM{{\bf M}}
\def\bS{{\bf S}}
\def\bW{{\bf W}}
\def\bN{{\bf N}}
\def\bU{{\bf U}}
\def\bX{{\bf X}}
\def\cS{{\cal S}}
\def\bH{{\bf H}}
\def\bZ{{\bf Z}}
\def\F{{\Bbb F}}
\def\C{{\Bbb C}}
\def\R{{\Bbb R}}
\def\bs{\backslash}
\def\cA{{\cal A}}
\def\cF{{\cal F}}
\def\cG{{\cal G}}
\def\bT{{\bf T}} 
\def\cB{{\cal B}}
\def\cO{{\cal O}}
\def\ind{{\rm ind}}
\def\bGU{{\bf U}}
\def\bGL{{\bf GL}} 
\def\GL{{\rm GL}}
\def\M{{\rm M}}
\def\O{{\rm O}}
\def\SL{{\rm SL}}
\def\U{{\rm U}}
\def\gO{{\frak O}}
\def\gP{{\frak P}}
\def\gal{{\rm Gal}}
\def\g{{\frak g}}
\def\half{\frac{1}{2}}
\def\Fq{{\Bbb F}_q}
\def\Fqq{{\Bbb F}_{q^2}}
\def\Fqb{\overline{{\Bbb F}}_q}
\def\t{\kern.1em{}^t\kern-.1em}
\def\2by2#1#2#3#4{\hbox{$\left( 
\begin{array}{cc}
{#1}&{#2}\\ {#3}&  {#4}\end{array}\right)$}}
\def\A{{\Bbb A}}
\def\Q{{\Bbb Q}}
\def\R{{\Bbb R}}
\def\cN{{\cal N}}
\def\Z{{\Bbb Z}}
\def\gB{\cB}
\def\ord{{\rm ord}}
\def\gO{{\frak O}}

\def\val{{\rm val}}
\def\cH{{\cal H}}
\def\cX{{\cal X}}
\def\cZ{{\cal Z}}
\def\cM{{\cal M}}
\def\cP{{\cal P}}
\def\vectr#1#2{\binom{#1}{#2}}

\def\b{{\frak b}}
\def\f{{\frak f}}
\def\g{{\frak g}}
\def\k{{\frak k}}
\def\gl{{\frak{gl}}}
\def\ind{\hbox{ind}}
\def\End{\hbox{End}}
\def\v#1{\underline{#1}}
\def\m#1{\underline{\underline{#1}}}

\title{Tame Supercuspidal Representations of $\GL_n$ Distinguished by Orthogonal Involutions}
\author{Jeffrey Hakim
\thanks{Department of Mathematics and Statistics, 
American University\hfill\break jhakim@american.edu\hfill\break Supported by NSF grant DMS-0854844.} }



\date{July 26, 2011}
\maketitle

\tableofcontents

\section{Introduction}

\subsection{Overview and acknowledgments}\label{sec:overview}

Let $G = \GL_n (F)$, where $n>1$ and $F$ is a finite extension of $\Q_p$, for some odd prime $p$.  The space $C^\infty (\cS)$ of locally constant functions on the set $\cS$ of symmetric matrices in $G$ is a  (non-smooth) $G$-module with respect to the action
$$(g\cdot f)(\nu) = f ({}^tg\, \nu\,  g),$$ with $g\in G$, $f\in C^\infty (\cS)$ and $\nu\in \cS$.  Given  a complex representation $(\pi ,V_\pi)$ of $G$, we say that a $G$-equivariant linear embedding of $V_\pi$ in $C^\infty (G)$ is a {\it symmetric matrix model} for $\pi$.
One of the main results of this paper specifies the dimension of  the space $\Hom_G (\pi ,C^\infty (\cS))$ of such models when $\pi$ is a  tame supercuspidal representation of $G$ (as constructed in \cite{rH}) with central character $\omega_\pi$:

\begin{theorem}\label{symmodelthm}
If $\pi$ is an irreducible  tame supercuspidal representation of $G$ then $\pi$ has a symmetric matrix model precisely when $\omega_\pi (-1)=1$.  In this case,  the dimension of $\Hom_G(\pi,C^\infty (\cS))$ is 4.
\end{theorem}

It is easy to see that determining the symmetric matrix models for $\pi$ is equivalent to determining for all $\nu\in \cS$ the space $\Hom_{G_\nu}(\pi,1)$ of linear forms $\lambda$ on $V_\pi$ that are invariant under the action of the orthogonal group
$$G_\nu = \{ h\in G\, : {}^t h \, \nu\, h = \nu \},$$ in the sense that 
$$\lambda (\pi (h)v) = \lambda(v),$$ for all $h\in G_\nu$ and $v\in V_\pi$.

For us,  ``orthogonal group''  will  mean such a group $G_\nu$  (or the associated algebraic group over $F$).  If $H$ is an orthogonal group in $G$ 
and $\Hom_{H}(\pi,1)$ is nonzero, we say that $\pi$ is {\it $H$-distinguished}.
We show:

\begin{theorem}\label{mainone}
If $\pi$ is an irreducible  tame supercuspidal representation of $G$ and if $H$ is a split orthogonal group then $\pi$ is $H$-distinguished precisely when  $\omega_\pi (-1)=1$.
\end{theorem}

Theorems \ref{maintwo} and \ref{mainthree}  supplement Theorem \ref{mainone} by precisely specifying the distinguished representations and  the dimension of $\Hom_{H}(\pi,1)$ for all tame supercuspidal representations $\pi$ of $G$ and all orthogonal groups $H$.  Theorem~\ref{symmodelthm} then follows easily from a case-by-case analysis, but we also provide a more conceptual proof.

To further illustrate Theorem \ref{symmodelthm} in the context of general symmetric spaces over $F$, we state it as a ``dichotomy''  result in the language of pure inner forms.  In other words, we give a formula for the total dimension of the spaces $\Hom_H (\pi,1)$ as $H$ varies over a pure inner class.  This restatement in general language is not meant to suggest what should happen in all other examples, but rather to provide a baseline for comparing this example with others.  As we will discuss, the examples considered in this paper are exceptional in several ways.

For simplicity, we only refer to the dimensions of  $\Hom_G (\pi ,C^\infty (\cS))$ and $\Hom_{G_\nu}(\pi,1)$ in our results, however, it is a routine matter to use our calculations and the methods of \cite{HM} and \cite{HL1} to explicitly describe the elements of the latter spaces.

In the theory of distinguished representations, it has become expected that for a given $F$-symmetric space $H\bs G$, the set of $H$-distinguished representations of $G$ is the image of a lifting map from some other group $H'$ to $G$.  Ideally, the connection between $H$-distinction and the lifting from $H'$ fits harmoniously within the Langlands program framework and its recent symmetric space formulation in the work of Sakellaridis and Venkatesh \cite{SV}.  Such harmony is lacking for the symmetric spaces we study.  On the one hand, the relevant lifting is the Flicker-Kazhdan metaplectic correspondence \cite{FK} from the double cover $\widetilde G$ of $G$ to  $G$.  Since $\widetilde G$ is not a linear group, it is not described by the Langlands correspondence.  In addition, the notion of the $L$-group of a symmetric space in \cite{SV} is undefined for our symmetric spaces.  (Roughly speaking, one hopes to have a symmetric space $L$-group ${}^L (H\bs G)$  such that a representation of $G$ is $H$-distinguished exactly when its Langlands parameter factors through ${}^L (H\bs G)$.)

Representations of $G$ with symmetric matrix models arise in the local theory of automorphic representations of $\GL_n$ with orthogonal periods.  Jacquet \cite{Ja} has suggested a comparison of (global) relative trace formulas associated to the two sides of the metaplectic correspondence mentioned above.  This idea has been pursued by Mao \cite{M}, but much work remains.   A local implication of this setup is that the linear forms in $\Hom_{H} (\pi, 1)$, for suitable orthogonal groups $H$, should be correlated with Whittaker functionals for the corresponding representation of $\widetilde G$.  We refer the reader to the work of Chinta and Offen  (\cite{CO1} and \cite{CO2}) for more information.

This paper is by no means self-contained.  Its heavy dependence on  \cite{rH}, \cite{Y}, \cite{HM} and \cite{HL1} places a burden on the reader that we have attempted to lessen by including extended sketches of our main proofs and recapitulations of some of the relevant material from the mathematical literature.  Readers of this paper are advised to first read the introductory sections and then skip to \S\ref{sec:Torbmultsec} and \S\ref{sec:evalsec} for a rigorous development of the proofs.  The latter two sections  refer to  \S\ref{sec:tamesec}--\ref{sec:etaprimetheta}  for various technical details, from general facts about tame supercuspidal representations and orthogonal involutions to  technical tools needed specifically in this paper.

The present work uses general techniques  for studying distinguished tame supercuspidal representations developed by the author and Murnaghan \cite{HM}, together with refinements appearing in \cite{HL1}.  This paper also builds on previous  work of the author with Mao \cite{HzM} (the depth zero case) and  Lansky \cite{HL1} (the case in which $n$ is odd) involving the specific examples considered in this paper.

The author wishes to thank Jeffrey Adams, Stephen DeBacker, Zhengyu Mao, Omer Offen, Dipendra Prasad, Yiannis Sakellaridis and, especially, Jeffrey Adler and Joshua Lansky for conversations that affected the progress and outcome of this paper.

\subsection{Dimensions of spaces of invariant linear forms}\label{sec:statements}

There is an obvious necessary condition for $H$-distinction, for any orthogonal group $H$:

\begin{quote}
If an irreducible representation $\pi$ of $G$ is $H$-distinguished then $\omega_\pi (-1) =1$.
\end{quote}
Indeed, this follows from the fact that if $\lambda\in \Hom_{H} (\pi, 1)$  then $\lambda (v) = \lambda (\pi (-1)v) = \omega_\pi (-1)\lambda (v)$, for all $v\in V_\pi$.

Given a degree $n$ tamely ramified extension $E$ of $F$ and an $F$-admissible (in the sense of \cite{rH}) quasicharacter $\varphi : E^\times \to \C^\times$ then Howe's construction \cite{rH} gives an equivalence class $\pi (\varphi)$ of tame supercuspidal representations of $G$.
The central character $\omega_\pi$ of $\pi\in \pi(\varphi)$ agrees with the restriction of $\varphi$ to $F^\times$.  
So the necessary condition above can also be expressed as $\varphi (-1) =1$.
The analogue of Theorem \ref{mainone} for non-split orthogonal groups is:

\begin{theorem}\label{maintwo}
Suppose $H$ is a non-split orthogonal group and $\pi$ is a tame supercuspidal representation asociated to an $F$-admissible quasicharacter $\varphi :E^\times \to \C^\times$ such that $\varphi (-1)=1$.
Then $\pi$ is 
$H$-distinguished precisely when one of the following conditions holds:
\begin{itemize}
\item $E$ contains a unique quadratic extension $L$ of $F$ and $H= G_\nu$,  where $(-1)^{n(n-1)/2}\det  (\nu )$ lies in  $N_{L/F} (L^\times) - (F^\times)^2$.  In this case, $H$ is necessarily quasi-split.
\item $E$ contains three quadratic extensions of $F$ and $H$ is not quasi-split.
\end{itemize}
\end{theorem}

\bigskip \noindent  For orthogonal groups $H$ and for $H$-distinguished $\pi$, the dimension of the space $\Hom_{H} (\pi,1)$ is given by:

\begin{theorem}\label{mainthree}
Suppose $H$ is an orthogonal group and  $\pi$ is  an irreducible $H$-distinguished tame supercuspidal representation of $G$ associated to an $F$-admissible quasicharacter $\varphi :E^\times \to \C^\times$.
\begin{enumerate}
\item
If $H$ is split then the dimension of $\Hom_{H} (\pi,1)$ is
\begin{itemize}
\item 1, if $n$ is odd or, equivalently, if $E$ contains no quadratic extensions of $F$,
\item 2, if $E$ contains a unique quadratic extension of $F$,
\item 3, if $E$ contains three quadratic extensions of $F$.
\end{itemize}
\item
If $H$ is not split  then the dimension of $\Hom_{H} (\pi,1)$ is 1.
\end{enumerate}
\end{theorem}

\subsection{Pure inner forms of orthogonal groups}\label{sec:pureinner}

Let $G$ act on $\cS$ via the right action:
$$\nu\cdot g = {}^t g \, \nu\, g.$$
The stabilizer of $\nu \in \cS$ is the orthogonal group $G_\nu$.  We let $\cS_\nu$ denote the orbit of $\nu$.  
Each orbit $\cS_\nu$ is open and closed in $\cS$ and thus
$$C^\infty (\cS)\cong \bigoplus_\nu C^\infty (\cS_\nu),$$
where we are summing over a set of representatives for the $G$-orbits (a.k.a., similarity classes) in $\cS$.  The number of $G$-orbits is 8, unless $n=2$ in which case there are 7 orbits.  The $G$-orbit of $\nu\in \cS$ is determined by the discriminant and the Hasse invariant of $\nu$.

For a representation $\pi$ of $G$,  we have
$$\Hom_G (\pi, C^\infty (\cS))\cong \bigoplus_\nu \Hom_G (\pi,C^\infty (\cS_\nu)).$$  
We may identify $\cS_\nu$ with $G_\nu\bs G$.  Then the space of smooth vectors in $C^\infty (\cS_\nu)$ is identified with the smooth representation ${\rm Ind}_{G_\nu}^G(1)$ induced from the trivial character of $G_\nu$.  
Therefore, when $\pi$ is smooth we have
$$\Hom_G (\pi,C^\infty (\cS_\nu))\cong \Hom_G (\pi,{\rm Ind}_{G_\nu}^G  (1)).$$  Applying Frobenius reciprocity, we obtain
\begin{equation}\Hom_G (\pi, C^\infty (\cS))\cong \bigoplus_\nu \Hom_{G_\nu} (\pi,1).\label{eqnsym}
\end{equation}

Theorem \ref{symmodelthm} can be deduced from the latter formula and Theorems \ref{mainone} -- \ref{mainthree} once we recall some basic facts about orthogonal groups.  For convenience, we now fix a tame supercuspidal representation $\pi$ such that $\omega_\pi (-1) =1$.  Assume that $\pi$ comes from an $F$-admissible quasicharacter $\varphi : E^\times \to \C^\times$.

For $\nu\in \cS$, define 
$$\theta_\nu (g) = \nu^{-1}\cdot {}^t g^{-1}\cdot \nu,$$ for all $g\in G$.  We call such involutions $\theta_\nu$ of $G$ {\it orthogonal involutions} because the group $G^{\theta_\nu}$ of fixed points in $G$ of $\theta_\nu$ is the orthogonal group $G_\nu$.  (See \S\ref{sec:quadraticspaces} for basic facts about orthogonal involutions.)    
The group $G$ acts on the set of its orthogonal involutions by $$g\cdot \theta = \Int (g) \circ \theta\circ \Int(g)^{-1},$$ where $(\Int (g))(g') = gg'g^{-1}$.  

We observe that if $\nu_1,\nu_2\in \cS$ then
$$G_{\nu_1} = G_{\nu_2}\Leftrightarrow \theta_{\nu_1} = \theta_{\nu_2} \Leftrightarrow \nu_1 Z = \nu_2 Z,$$ where $Z$ is the center of $G$.
Furthermore, 
\begin{eqnarray*}
G_{\nu_1}\cong G_{\nu_2}
&\Leftrightarrow&
G_{\nu_1}\text{ and }G_{\nu_2}\text{ are $G$-conjugate}\\
&\Leftrightarrow&\text{$\theta_{\nu_1}$ and $\theta_{\nu_2}$ are in the same $G$-orbit}\\
&\Leftrightarrow&\text{$\nu_1 Z$ and $\nu_2 Z$ are in the same $G$-orbit.}
\end{eqnarray*}
In fact, for all $g\in G$ and $\nu\in \cS$ we have
$$G_{{}^t g \nu g} = g^{-1}G_\nu g.$$

The $G$-orbits of orthogonal involutions are given as follows.  Let $\Theta_J$ be the $G$-orbit of $\theta_J$, where  $$J= J_n=
\begin{pmatrix}
& &1\\
&\raisebox{-.1ex}{.} \cdot \raisebox{1.2ex}{.}&\\
1&&
\end{pmatrix}$$
Then the involutions in $\Theta_J$ are precisely the involutions that give rise to split orthogonal groups.

If $n$ is odd  there is only one other $G$-orbit of orthogonal involutions.  We denote it by $\Theta_{\rm nqs}$.
When $n$ is even and greater than two, we let $\Theta_{\rm nqs}$ denote the $G$-orbit consisting of all orthogonal involutions $\theta_\nu$ associated to symmetric matrices $\nu$ not similar to $J$ but having the same discriminant as $J$.  Whether $n$ is odd or even, the elements of $\Theta_{\rm nqs}$ are the involutions that yield orthogonal groups that are not quasi-split.

When $n$ is even there are three additional $G$-orbits of orthogonal involutions corresponding to the three possible discriminants other than the discriminant of $J$.  In other words, if $\theta_\nu$ is not in $\Theta_J$ or $\Theta_{\rm nqs}$ then its $G$-orbit is determined by the discriminant of $\nu$.  The associated orthogonal groups are quasi-split but not split.

Another elementary fact is that
we have an isomorphism
$$\Hom_{G_\nu}(\pi,1) \cong \Hom_{g^{-1}G_\nu g}(\pi,1)$$ given by mapping
$\lambda\in \Hom_{G_\nu}(\pi ,1)$ to the linear form $w\mapsto \lambda (\pi(g)w)$.  This shows that, for a given $\pi$, the dimension of $\Hom_{G^\theta}(\pi,1)$ is constant as $\theta$ varies over a $G$-orbit of involutions.

Let $\bG$ be the $F$-group $\GL_n$.  Given $\sigma \in \Gamma$ and $g\in \bG$, we use the notation $\sigma (g)$ for the standard Galois action.
  Fix $\nu\in \cS$ such that the orthogonal group
$$\bH = \{ g\in \bG \ : {}^t g\, \nu\, g = \nu \}$$ is split over $F$.  (For example, take $\nu =J$.)  
Let $H = \bH(F) = G_\nu$. Suppose $z\in Z^1(F,\bH)$ is a Galois 1-cocycle.  In other words, if $\overline{F}$ is an algebraic closure of $F$ then $z$ is a map $\sigma \mapsto z_\sigma$ from $\Gamma = {\rm Gal}(\overline{F}/F)$ to $\bH$ such that $z_{\sigma\tau} = z_\sigma\ \sigma (z_\tau)$, for all $\sigma,\tau \in \Gamma$.

We may also view $z$ as an element of $Z^1(F,\bG)$.  Since $H^1(F,\bG)$ is trivial, we can choose $\zeta\in \bG$ such that $z_\sigma = \zeta\, \sigma (\zeta)^{-1}$, for all $\sigma\in \Gamma$.
The coset $\zeta G$ is canonically associated to $\omega$ and it consists of all possible choices of $\zeta$.
Note that for any element $\zeta$ in $\bG$, if we define $z_\sigma = \zeta \, \sigma (\zeta)^{-1}$ then we have 
\begin{eqnarray*}
{}^t\zeta\, \nu\, \zeta\in \cS&\Leftrightarrow&\sigma ({}^t\zeta\, \nu\, \zeta) = {}^t\zeta\, \nu\, \zeta,\ \forall\sigma\\
&\Leftrightarrow&\nu ={}^t z_\sigma \nu \, z_\sigma,\ \forall\sigma\\
&\Leftrightarrow&z_\sigma\in \bH,\ \forall\sigma\\
&\Leftrightarrow&z\in Z^1(F,\bH).
\end{eqnarray*}
This gives a canonical bijection between $H^1(F,\bH)$ and the set of $G$-orbits in $\cS$ that sends the cohomology class of $z$ to the $G$-orbit of ${}^t \zeta \, \nu\, \zeta$.  (This also follows from  Corollary 1 of Proposition I.5.36 and Lemma III.1.1 in \cite{Se}.) 

Associated to $\bH$ and $z$ is the group $\bH_z^*$ which is the same as $\bH$ with the modified Galois action
$$\sigma_* (g) = z_\sigma\, \sigma (g)\, z_\sigma^{-1}.$$
Consider the group $\zeta^{-1} \bH^*_z\zeta$ with the standard Galois action inherited from $\bG$.
 Then $h\mapsto \zeta h \zeta^{-1}$ determines an $F$-isomorphism between $\zeta^{-1}\bH_z^*\zeta$ and $\bH_z^*$.  
On the other hand, if $\nu_* = {}^t\zeta\, \nu\, \zeta$ then
$$\zeta^{-1}\bH_z^*\zeta  = \{ g\in \bG\ : {}^t g\, \nu_*\, g = \nu_*\}.$$
If $\omega\in H^1(F,\bH)$ is the cohomology class of $z$ then we let $H_\omega$ denote the $G$-conjugacy class of $G_{\nu_*}$.

 For our purposes, a {\it pure inner form of $H$} is a pair $(\omega , H_\omega)$, with $\omega\in H^1(F,\bH)$.  By abuse of notation, we  also use the notation $H_\omega$ for a specific group in the conjugacy class of $H_\omega$.
The connection between $\Hom_G (\pi,C^\infty (\cS))$ and pure inner forms (as we have defined them) is now given by  $$\Hom_G (\pi,C^\infty (\cS)) \cong \bigoplus_{\omega\in H^1(F,\bH)}  \Hom_{H_\omega}(\pi,1).$$
Therefore Theorem \ref{symmodelthm} may be stated as:
\begin{equation}
\sum_{\omega\in H^1(F,\bH)}  \dim \Hom_{H_\omega}(\pi,1)=4.
\end{equation}

One can also consider the sum 
$$ \bigoplus_{\omega\in H^1(F,\bH^\circ )}  \Hom_{H_\omega}(\pi,1),$$
where $\bH^\circ$ is the identity component of $\bH$.  So $\bH^\circ$ is a split special orthogonal group.  In the above discussion, to require $z\in Z^1(F,\bH^\circ )$ means that $\nu_*= {}^t \zeta \, \nu \, \zeta \in \cS$ and $\det (\zeta\, \sigma(\zeta)^{-1})=1$ for all $\sigma\in \Gamma$.  But the latter condition is equivalent to $\det \zeta\in F^\times$.  It is also equivalent to $(\det \zeta)^2 = (F^\times)^2$.  So $z\in Z^1(F,\bH^\circ)$ if and only if the element $\nu_*$ lies in $\cS$ and has the same discriminant as $\nu$.
When $n=2$, $H^1(F,\bH^\circ)$ is trivial, but otherwise $H^1(F,\bH^\circ)$ has order two.  When $n\ne 2$ and $\omega$ is the nontrivial cohomology class then $H_\omega$ is a non-quasi-split orthogonal group.
 Theorems \ref{mainone}, \ref{maintwo} and \ref{mainthree} imply:
\begin{equation}
\sum_{\omega\in H^1(F,\bH^\circ)}  \dim \Hom_{H_\omega}(\pi,1)= [E^\times : (E^\times)^2 F^\times].
\end{equation}

\subsection{Deducing Theorem \ref{symmodelthm} from Theorems \ref{mainone}, \ref{maintwo}, and \ref{mainthree}}

\subsubsection{Computational approach}

Theorem \ref{symmodelthm} follows directly from Theorems \ref{mainone}, \ref{maintwo}, and \ref{mainthree}.  We now describe in detail how this is verified, in so doing, we exhibit the contributions of the various $G$-orbits in $\cS$ (or, in other words, the various pure inner forms) to $\Hom_G(\pi, C^\infty (\cS))$.

Assume $n$ is odd.  Then there are eight $G$-orbits in $\cS$.  Four of these orbits have representatives of the form $zJ$, with $z\in Z$.  These $G$-orbits project to the single $G$-orbit $\Theta_J$ of orthogonal involutions.  This, in turn, corresponds to the unique conjugacy class of split orthogonal groups in $G$.  For such an orthogonal group $H$, the dimension of $\Hom_H(\pi,1)$ is one, according to the results of \cite{HL1}.  
The other four $G$-orbits in $\cS$ yield a single conjugacy class of non-quasi-split orthogonal groups.  the representation $\pi$ is not $H$-distinguished for such orthogonal groups $H$, according to \cite{HL1}.  Therefore, Equation \ref{eqnsym} implies
$$\dim \Hom_G(\pi , C^\infty (\cS)) = 4\cdot 1 + 4\cdot 0 = 4,$$
which is consistent with Theorem \ref{symmodelthm}.

If $n$ is even and greater than 2, then there are again eight $G$-orbits in $\cS$.  The $G$-orbit of $J$ projects to a unique $G$-orbit of orthogonal involutions, namely, $\Theta_J$.  The orthogonal groups in this case are split.  There is also a unique $G$-orbit in $\cS$ consisting of the symmetric matrices $\nu$ that  have the same discriminant as $J$ but do not lie in $\cS_J$.  This $G$-orbit projects to $\Theta_{\rm nqs}$.  The remaining six $G$-orbits in $\cS$ collapse in pairs to form three $G$-orbits of involutions.  Two of these $G$-orbits in $\cS$ merge precisely when they have the same discriminant.  All of the associated orthogonal groups are quasi-split but not split.  When $n=2$, the situation is similar, except that there are only seven $G$-orbits in $\cS$.  The orbit that projects to $\Theta_{\rm nqs}$ is missing.

According to Lemma \ref{fieldsA}, when $n$ is even the field $E$ must contain at least one quadratic extension of $F$.  
Let us consider first the case in which $E$ contains a unique quadratic extension $L$ of $F$.  According to Theorems \ref{mainone} and \ref{mainthree}, if $H$ is a split orthogonal group then $\Hom_H(\pi,1)$ has dimension two.  According to Theorems \ref{maintwo} and \ref{mainthree}, there is a single conjugacy class of quasi-split, non-split orthogonal groups $H$ such that $\pi$ is $H$-distinguished.  For this $H$, the dimension of $\Hom_H(\pi,1)$ is 1.  However, since the conjugacy class of $H$ comes from two $G$-orbits in $\cS$, we count it twice in the tally of dimensions coming from Equation \ref{eqnsym}.  For all other orthogonal groups $H$, the representation $\pi$ is not $H$-distinguished.
We obtain
$$\dim \Hom_G(\pi , C^\infty (\cS)) = 2 + 2 = 4,$$
which is again consistent with Theorem \ref{symmodelthm}.

Finally, we assume that $E$ contains more than one quadratic extension of $F$.  Note that if $E$ contains two quadratic extensions of $F$ then it must in fact contain all three possible quadratic extensions of $F$.  Theorems \ref{mainone}, \ref{maintwo}, and  \ref{mainthree} imply that $\Hom_H (\pi,1)$ has dimension three,  when $H$ is split, and has dimension one when $H$ is not quasi-split.  Otherwise, $\pi$ is not $H$-distinguished.  We have
$$\dim \Hom_G(\pi , C^\infty (\cS)) = 3+1 = 4.$$

\subsubsection{Conceptual approach}

In this section, we provide another sketch of the proof of Theorem \ref{symmodelthm}  that highlights the aspects of the argument that might generalize.

Let $\bG$ be the $F$-group $\GL_n$.  Let  $\theta$ be the involution $\theta_J$.  We have a tame supercuspidal representation $\pi$ of $G = \bG (F)$ that is associated to an $F$-admissible quasicharacter of the multiplicative group of a degree $n$ tamely ramified extension $E$ of $F$.  We choose a maximal $F$-torus $\bT$ such that $T = \bT (F)\cong E^\times$ and $\bT$ is $\theta$-split in the sense that $\theta (t) = t^{-1}$ for all $t\in \bT$.  The fact that such a $\bT$ is a very special feature of our example.
  (See \cite{S} for moredetails.)
  
  Let $\bH = \bG^\theta$, $H = \bH(F)$,  $\underline{\cS} = \bH\bs \bG$ and $\cS = \underline{\cS}^\Gamma$.
We let $\bG$ act by right translations on $\underline{\cS}$. 
Suppose $\nu\in \cS$.  Then $\nu = \bH g$, where $g\in \bG$ and $\sigma\mapsto g\sigma (g)^{-1}$ defines a cocycle in $Z^1(F,\bH)$.
This yields a canonical bijection between $H^1(F,\bH)$ and the set of $G$-orbits in $\cS$.  We define $\theta_\nu = g^{-1}\cdot\theta$ and $G_\nu = G^{\theta_\nu} = g^{-1}Hg$.

Theorem \ref{symmodelthm} asserts that the space
$$\Hom_G (\pi,C^\infty (\cS)) = \bigoplus_{\nu \in H^1 (F, \bH)} \Hom_{G_\nu} (\pi,1)$$ has dimension four.
This can be proved by showing that the set  ${\cal X}/T$ of $T$-orbits in the set $$\cX = \{ \nu\in \cS\ : \text{$\bT$ is $\theta_\nu$-split}\}$$  has cardinality four and parametrizes a basis of 
$\Hom_G (\pi,C^\infty (\cS))$.  This is essentially what we do.

It is elementary to verify that there is  a bijection between $T/T^2$ and $\cX /T$ given by mapping the coset $tT^2$ to $\bH s$ where $s$ is any square root of $t$ in $\bT$.
Thus $\cX /T$ has cardinality four since it is in bijection with $E^\times /(E^\times)^2$.
(This argument was suggested to us by Jeffrey Adams.  Another argument is given later.)

Given $\nu\in \cS$, let $\cX_\nu = \{ \eta\in \cS_\nu\, :  \text{$\bT$ is $\theta_\eta$-split}\}$.  It suffices to show that
$$ \dim \Hom_{G_\nu}(\pi,1)=\# (\cX_\nu /T) .$$
Suppose $\eta\in \cS_\nu$.  Let $\zeta$ be the $T$-orbit of $\theta_\eta$.  Let $m_T(\zeta)$ be the number of $T$-orbits in $\cS_\nu$ that project to $\zeta$.  Then we have
$$\# (\cX_\nu/T) = \sum_\zeta m_T(\zeta),$$ where we are summing over the $T$-orbits $\zeta$ in the $G$-orbit of $\theta_\nu$ such that $\bT$ is $\theta$-split for some (hence all) $\theta'\in \zeta$.

The fact that $$\dim\Hom_{G_\nu}(\pi,1) = \sum_\zeta m_T(\zeta)$$ follows from  Proposition \ref{newmultformula} and some results of Lusztig on distinguished representations of finite groups of Lie type.  We
sketch the details in the next section.

%

\subsection{Proof outlines for Theorems \ref{mainone}, \ref{maintwo}, and \ref{mainthree}}\label{sec:outline}

\subsubsection{General theory}
\label{sec:general}

For $F$-groups, like $\bG = \GL_n$, we use boldface letters.  Nonbold letters, like $G$ for $\GL_n (F)$, are used for the corresponding groups of $F$-rational points.  

The tame supercuspidal representations of $G$ were constructed by Howe \cite{rH} and then Howe's construction was generalized to other connected, reductive $F$-groups by Yu \cite{Y}.  The data used by Howe and Yu to parametrize the supercuspidal representations of $G = \GL_n(F)$ looks slightly different.   Since we draw on other papers that refer to both \cite{rH} and \cite{Y}, it is convenient for us to introduce the notion of a {\it $G$-datum}, an amalgamation of Howe and Yu's inducing data.  The precise definition is stated  in \S\ref{sec:inducing}, but, roughly, a $G$-datum $\Psi$ consists of
\begin{itemize}
\item a choice of $E$,
\item an $F$-admissible quasicharacter $\varphi$ of $E^\times$,
\item a Howe factorization $\{ \varphi_i : E_i^\times \to \C^\times \}$, 
\item and an $F$-embedding of $E$ in $\M_n (F)$,
\item $\vec\bG = (\bG^0,\dots ,\bG^d = \bG)$, where $\bG^i$ is the centralizer of $E_i^\times$ in $\bG$,
\item the elliptic maximal $F$-torus $\bT$ in $\bG$ such that $T  = E^\times$ and the corresponding point $[y]$ in the reduced building of $G^0$,
\item a certain representation $\rho$ of the isotropy group $K^0 = G^0_{[y]}$ of $[y]$ in $G^0$,
\item $\vec\phi = (\phi_0,\dots , \phi_d)$, where $\phi_i$ is the quasicharacter $\varphi_i \circ \det^{G^i}_{E_i^\times}$ of $G^i$.
\end{itemize}
The quasicharacter $\varphi_i$ in the Howe factorization is defined on the multiplicative group of a field $E_i$ and we have
$$F = E_d \subsetneq \cdots \subsetneq E_0 \subseteq E\subsetneq \M_n (F).$$

A $G$-datum $\Psi$ is {\it toral} if  $E=E_0$.  In this case, $\rho$ is the trivial representation of $E^\times$.
In the non-toral case, $\rho$ is defined as follows. 
 Let $G^0_{y,0}$ be the parahoric subgroup of $G^0$ associated to $[y]$ and $G^0_{y,0^+}$  its pro-unipotent radical.   The quotient 
 $G^0_{y,0:0^+} = G^0_{y,0}/G^0_{y,0^+}$ is isomorphic to the group of $\f$-rational points of a connected reductive group 
 $\sG^0_y$
  defined over the residue field $\f$ of $F$.  The group $\sG_y^0(\f)$ is isomorphic to $\GL_{n_0}(\f_{E_0})$, where $\f_{E_0}$ is the residue field of $E_0$ and $n_0 = [E:E_0]$.
 The restriction of $\varphi_{-1}$ to the multiplicative group of the ring of integers of $E$ yields a nonsingular character $\lambda$ of an elliptic torus $\mathsf{T}(\f)$ of $\sG^0_y(\f)$.  The character $\lambda$ then gives an irreducible cuspidal representation $\pm R^\lambda_{\mathsf{T}(\f)}$ of $\sG^0_y(\f)$ via the construction of Green/Deligne/Lusztig \cite{DL}.  The inflation of  this representation to $G^0_{y,0}$ is the restriction of $\rho$ to $G^0_{y,0}$.  To complete the definition of $\rho$, we use the fact that $K^0$ is generated by $G^0_{y,0}$ and any prime element $\varpi_{E_0}$ of $E_0$ and we declare that $\rho (\varpi_{E_0})$ is the scalar operator corresponding to the scalar $\varphi_{-1}(\varpi_{E_0})$.

Fix $\Psi$ and a $G$-orbit $\Theta$ of orthogonal involutions of $G$ and let $\pi$ denote the associated tame supercuspidal representation of $G$.
Let $\xi$ be the set of refactorizations of $\Psi$ (in the sense of \cite{HM}).
Define $$\langle \Theta,\xi \rangle_G = \dim \Hom_{G^\theta}(\pi,1),$$ for any $\theta\in \Theta$.
(The fact that this is well-defined is explained in \cite{HM}.)

Since $\pi$ is induced from a certain open, compact-mod-center subgroup $K$ of $G$, one can use Mackey's theory to express $\Hom_{G^\theta}(\pi,1)$ as a direct sum of smaller $\Hom$-spaces
parametrized by the double cosets in $K\bs G/G^\theta$.

Each double coset $K gG^\theta$ corresponds to a $K$-orbit of involutions in $\Theta$, namely, the $K$-orbit of $g\cdot \theta$.  Using this correspondence, one can rewrite the Mackey sum as a sum parametrized by $K$-orbits of involutions.  This is done in \cite{HM} (though there is an error that is corrected in \cite{HL1}).  One advantage of the latter reformulation is that certain repeated terms in the Mackey sum are identified and collected together.

Unfortunately, the group $K$ has a rather complicated structure that would seem to make it impractical to completely describe the structure of the double coset space $K\bs G/G^\theta$ or the corresponding space of $K$-orbits of involutions.  Theorem 3.10 (1) \cite{HL1} offers a significant simplification.  It replaces the sum over $K$-orbits from \cite{HM} with a certain sum over $K^0$-orbits,
\begin{equation}\label{HLonemult}
\langle \Theta ,\xi \rangle_G =\sum_{\vartheta\sim \xi }  m_{K^0} (\vartheta )\  \langle \vartheta ,\xi\rangle_{K^0},\end{equation}
that we now explain.

Propositions 5.7 and 5.20 in \cite{HM} imply that if a given $K$-orbit param-etrizes a nonzero term in the sum (in the aforementioned formula in \cite{HM}) then there exists $\dot\Psi\in \xi$ that is $\theta$-symmetric for some $\theta$ in the given $K$-orbit.  (The notion of $\theta$-symmetry is defined in \cite{HM} and recalled below in Definition \ref{thetasymmdef}.)
If $\dot\Psi$ is $\theta$-symmetric then it is $\theta'$-symmetric for all $\theta'$ in the $K^0$-orbit $\vartheta$ of $\theta$.  In this case, we write $\vartheta\sim \xi$.  The key to the transition from $K$-orbits to $K^0$-orbits is Proposition 3.8 \cite{HL1} which states that  a given $K$-orbit can contain at most one $K^0$-orbit $\vartheta$ such that $\vartheta \sim \xi$.

If $\theta\in \vartheta$, $\dot\Psi\in \xi$ and $\dot\Psi$ is $\theta$-symmetric then, according to Proposition 3.9 \cite{HL1}, $\vartheta\sim \xi$.  In this case, by definition,
\begin{equation}\label{termformula}
\langle \vartheta ,\xi\rangle_{K^0} = \dim \Hom_{K^{0,\theta}}(\rho (\dot\Psi),\eta_\theta (\dot\Psi)),
\end{equation}  where $K^{0,\theta} = K^0\cap G^\theta$ and $\eta_\theta$ is a character of $K^{0,\theta}$ whose definition (from \cite{HM}) is recalled in \S\ref{sec:etaprimetheta}.
A more invariant formula can be obtained  by first defining a quasicharacter $$\phi (g) = \prod_{i=0}^d \phi_i (g)$$ of $G^0$ and letting $\rho' = \rho \otimes (\phi |K^0)$ and $\eta'_\theta = \eta_\theta (\phi |K^{0,\theta})$.  Then $\rho'$ and $\eta'_\theta$ do not vary within the refactorization class $\xi$ and we have
$$\langle \vartheta ,\xi\rangle_{K^0} = \dim \Hom_{K^{0,\theta}}(\rho'  ,\eta'_\theta ).$$

It will turn out that, for the examples in this paper,  one always has $\langle \vartheta ,\xi\rangle_{K^0}\le 1$.  Roughly speaking, this is a consequence of the fact that one has multiplicity-free decompositions over the residue fields.

The constant $m_{K^0}(\vartheta)$ is defined by
$$m_{K^0}(\vartheta ) = [G_\theta : (K^0\cap G_\theta)G^\theta],$$ where $\theta$ is an arbitrary element of $\vartheta$ and $G_\theta$ is the similitude group associated to $G^\theta$.  It represents the number of double cosets in $K\bs G/G^\theta$ that correspond to the $K$-orbit of $\theta$, as indicated above.

\subsubsection{The theory for our examples}\label{sec:ourexamples}

In the previous section, we described the transition from $K$-orbits to $K^0$-orbits.  For the examples considered in this paper, we can go two steps further.  We show that if $\vartheta$ is a $K^0$-orbit such that $\langle \vartheta ,\xi\rangle_{K^0}\ne 0$ then $\vartheta$ contains a unique $T$-orbit $\zeta$ of involutions $\theta$ such that $\bT$ is $\theta$-stable.  Moreover, in the latter situation $\bT$ must be $\theta$-split, that is, $\theta (t) = t^{-1}$ for all $t\in \bT$.
This allows us to replace Equation \ref{HLonemult} with an equation 
\begin{equation}\label{Tmultformula}
\langle \Theta ,\xi\rangle_G =\sum_{\zeta}  m_{T} (\zeta )\  \langle \zeta ,\xi\rangle_{T},\end{equation}
where we are summing over the $T$-orbits $\zeta$ in $\Theta$ that contain involutions $\theta$ such that $\bT$ is $\theta$-split.  (This is Proposition \ref{newmultformula}.)  
The constant $m_T(\zeta)$ is defined by
$$m_T(\zeta) = [ G_\theta : (T\cap G_\theta) G^\theta].$$  (If $\zeta$ is contained in the $K^0$-orbit $\vartheta$ then $m_T(\zeta) = m_{K^0}(\vartheta)$.)

We now sketch the derivation of Equation \ref{Tmultformula} and describe how its terms are evaluated.

Let us consider a term $\langle \vartheta,\xi\rangle_{K^0}$ in Equation \ref{HLonemult}, where $\vartheta\sim \xi$.
We assume that we have $\Psi\in \xi$ and $\theta\in \vartheta$ such that $\Psi$ is $\theta$-symmetric, and note that there is no essential loss of generality in making this assumption.  Then Equation \ref{termformula} says that $$\langle \vartheta ,\xi\rangle_{K^0} = \dim \Hom_{K^{0,\theta}}(\rho, \eta_\theta),$$
where $\eta_\theta$ is a character of $K^{0,\theta}$ that is the product of $\phi |K^{0,\theta}$ and another character $\eta'_\theta$ of $K^{0,\theta}$.

The precise definition of $\eta'_\theta$ is rather technical and it appears to be a rather difficult, but necessary, task to compute $\eta'_\theta$.  We proceed as follows.  It is easily seen   that $\eta'_\theta$ may be regarded as character on $\sG^{0,\theta}_y(\f)$, the group of $\theta$-fixed points in $\sG_y^0(\f) = G^0_{y,0:0^+}$, but determining the structure of $\sG^{0,\theta}_{y,0}(\f)$ requires some effort.
Indeed, it appears that this group could have one of several structures.  Ultimately, we show in Lemma \ref{restsareorthogonal} that $\langle \vartheta ,\xi\rangle_{K^0}\ne 0$ implies that $\sG^{0,\theta}_y(\f)$ is a finite orthogonal group.

Without actually knowing the precise structure of $\sG_y^0(\f)$, we  show in Lemma \ref{etaprimeidentity} that $\eta'_\theta$ has trivial restriction to the group $(\sG^{0,\theta}_y)^\circ(\f)$ of $\f$-rational points of the identity component of $\sG_y^{0,\theta}$.  (The techniques used in the proof Lemma \ref{etaprimeidentity} appear to be applicable to other examples.  This may be the most useful technical tool introduced in this paper that applies to a general class of examples.)

In Lemma \ref{thetastableT}, we show that if $\langle \vartheta ,\xi\rangle_{K^0}$ is nonzero then there exists $\Psi\in \xi$ and $\theta\in\vartheta$ such that $\Psi$ is $\theta$-symmetric and $\bT$ is $\theta$-stable.  Then, in Lemma \ref{thetasplitreduction}, we establish  that, in the latter case, $\bT$ must actually be $\theta$-split.  The proof of Lemma \ref{thetasplitreduction} uses Lemma \ref{tenfour} and the fact that $\vartheta$ is $F$-admissible.  It also uses the theory of ``$J$-symmetric embeddings'' from \S\ref{sec:Jsymmetric} to precisely describe the restrictions of orthogonal involutions of $G$  to the subgroups $G^i$.
This leads us to 
Lemma \ref{restsareorthogonal} which establishes that the relevant restrictions of $\theta$ to the groups $G^i$ and  $\sG_y^{0}(\f)$ must always be orthogonal involutions.  (Some of the restrictions of $\theta$ are not orthogonal involutions, however,  they can be ignored because they make no contribution to the dimension we are computing.)

At this point, we can easily compute $\langle \vartheta,\xi\rangle_{K^0}$.  Suppose $\langle \vartheta ,\xi\rangle_{K^0}$ is nonzero and $\xi$ is nontoral.   Choose $\Psi\in \xi$ and $\theta\in\vartheta$ such that $\Psi$ is $\theta$-symmetric and $\bT$ is $\theta$-split. 
%
%
Then 
$$\langle \vartheta ,\xi\rangle_{K^0} = \dim \Hom_{\sG_y^{0,\theta}(\f)}((-1)^{n_0+1}R^\lambda_{\sT(\f)},\eta_\theta).$$
Theorem 3.11 \cite{HL1}  gives a formula for the latter dimension in which $\eta_\theta$ may be regarded as an unspecified character.   This formula is a routine generalization of Theorem 3.3 \cite{L} that is needed since Lusztig's result only addresses the case in which $\eta_\theta$ is trivial.  To evaluate the terms in the formula, we  use Lemma \ref{tenfour}, which generalizes Lemma 10.4 \cite{L}, again to account for nontrivial $\eta_\theta$.  
If one examines Lusztig's lemma, one sees that the proof carries through with no essential modification in our case due to the fact that $\eta'_\theta$ is trivial on $(\sG_y^{0,\theta})^\circ (\f)$.  
Lemma \ref{tenfour} says that the tori that enter into Theorem 3.11 \cite{HL1} in our case are $\theta$-split and this makes it trivial to evaluate the terms in the formula from Theorem 3.11 \cite{HL1}.
Remarkably, the only value of $\eta'_\theta$ that is needed to evaluate the latter formula is $\eta'_\theta (-1)$, which, from the definition of $\eta'_\theta$ is obviously trivial since $-1$ is a scalar.  So, in fact, we never actually need to compute $\eta'_\theta$.  (We only needed to know that it is trivial on $(\sG_y^{0,\theta})^\circ(\f)$ and at $-1$).) When $\xi$ is nontoral, we  obtain
\begin{eqnarray*}
\langle \vartheta ,\xi\rangle_{K^0}
&=&
\dim \Hom_{\sG_y^{0,\theta}(\f)}((-1)^{n_0+1}R^\lambda_{\mathsf{T}(\f)} ,\eta_\theta)\\
&=&
\begin{cases}
1,&\text{if $\varphi (-1)=1$},\\
0,&\text{otherwise},
\end{cases}
\end{eqnarray*}
but this formula clearly also holds in the toral case.

Lemma \ref{thetasplitreduction} shows that 
if $\langle \vartheta ,\xi\rangle_{K^0}$ is nonzero then  $\vartheta$ must contain a {\it split $T$-orbit} $\zeta$ in the sense that $\zeta$ is a $T$-orbit of involutions $\theta$ such that $\bT$ is $\theta$-split.  According to
Proposition \ref{uniqueTinKzero}, it must be the case that $\vartheta$ contains a unique split $T$-orbit $\zeta$.
At this point, we have nearly demonstrated Equation \ref{Tmultformula} (which is Proposition \ref{newmultformula}), where if $\zeta$ is the unique split $T$-orbit in $\xi$, we define $$\langle \zeta ,\xi\rangle_T = \langle \vartheta , \xi\rangle_{K^0}.$$  The only thing that remains to be shown is that $m_T(\zeta) = m_{K^0}(\vartheta)$.
This is proved with an algebraic theory that we develop in \S\ref{sec:Jsymmetric}.  This theory involves the notion of a $J$-symmetric embedding of $E$ in $\M_n (F)$.  We also use this theory to obtain a very convenient description of the split $T$-orbits of orthogonal involutions, and to determine the restrictions of orthogonal involutions to the groups $G^i$ and $\sG^0_{y,0}(\f)$.

A {\it $J$-symmetric embedding} of $E$ in $\M_n (F)$ is an $F$-linear embedding $x\mapsto \m{x}$ such that $J\, {}^t \m{x}\, J = x$ for all $x\in E$.  Replacing $J$ with another symmetric matrix $\nu\in \cS$, we obtain the more general notion of a {\it $\nu$-symmetric embedding}.  The first question regarding such embeddings is the question of existence.   When $\nu$ is the identity matrix, the existence issue is the classical question:
\begin{quote}
Given a degree $n$ (finite) extension $E/F$ of fields, when does there exist an $F$-linear embedding of $E$ in the vector space of symmetric matrices in $\M_n (F)$?
\end{quote}
The existence question has been successfully attacked in many (or perhaps all) cases.  (See \cite{B}, for example.)  
Our approach is to downplay the use of trace forms and to place $J$-symmetric embeddings at the center of the theory of all $\nu$-symmetric embeddings.  More precisely, once one understands $J$-symmetric embeddings, it is easy to use them to describe $\nu$-symmetric embeddings for all $\nu$.  Using results in \cite{HL1}, we see that, for our tamely ramified extensions of $p$-adic fields, $J$-symmetric embeddings always exist, but $\nu$-symmetric embeddings do not exist for all $\nu\in \cS$.

Suppose we fix a $J$-symmetric embedding of $E$ in $\M_n (F)$ and choose $\bT$, as we are free to do, so that $T$ is the image of $E^\times$ in $G$.  Then it is easy to see that if $\nu\in \cS$ then our $J$-symmetric embedding is also $\nu$-symmetric precisely when $\nu = Jt$, for some $t\in T$.  It is also easy to deduce from this that we have a bijection between $T/(T^2Z)$ and the set $\cO^T$ of split $T$-orbits of orthogonal involutions given by mapping the coset of $t\in T$ to the $T$-orbit of $\theta_{Jt}$.  (See Proposition~\ref{symmetricembedding} and Lemma \ref{muprime}.)  Given the latter description of $\cO^T$, it is a relatively straightforward matter to determine how the various split $T$-orbits in $\cO^T$ are distributed among the various $G$-orbits of orthogonal involutions.  These calculations overlap with the computations of the constants $m_{K^0}(\vartheta)$ and $m_T(\zeta)$.

 \section{Tame supercuspidal representations}\label{sec:tamesec}
 
 \subsection{Inducing data}\label{sec:inducing}
 
This section discusses basic terminology and notations related to the inducing data for the construction of tame supercuspidal representations.  For the most part, we follow the presentation in \cite{HM} and \cite{HL1} and the reader should consult the latter references for a more detailed exposition.

Howe's construction \cite{rH} associates an equivalence class $\pi (\varphi)$ of irreducible tame supercuspidal representations of $G= \GL_n (F)$ to an $F$-admissible quasicharacter $\varphi : E^\times \to \C^\times$ of the multiplicative group of a  tamely ramified degree $n$ extension $E$ of $F$.  Assume that $E$ and $\varphi$ have been fixed.

Howe shows that the choice of $\varphi$ canonically determines a tower of intermediate fields of $E/F$.  Following the conventions in  \cite{HM}, we denote this tower as follows:
$$F= E_d\subsetneq \cdots \subsetneq E_1\subsetneq E_0\subseteq E_{-1}= E.$$

To construct an actual representation in the equivalence class $\pi (\varphi)$, one needs to choose an $F$-embedding of $E$ in $\g = \gl_n (F)$ and, in the following sense, a Howe factorization of $\varphi$:

\begin{definition}
 A \textbf{Howe factorization} of $\varphi$
consists of 
\begin{itemize}
\item a tower of fields
$F=E_d\subsetneq E_{d-1}\subsetneq\cdots\subsetneq E_0\subseteq E_{-1}= E$,
with $d\ge 0$, 
\item
a collection of quasicharacters $\varphi_i$, $i=-1,
\dots, d$, 
\end{itemize}
with the following properties:
\begin{itemize}
\item Let $N_{E/E_i}$ denote the norm map from $E^\times$ to
$E_i^\times$, for $i\in \{\, 0,\dots,d\,\}$.
  Then for each $i\in \{\, 0,\dots,d\,\}$,
 $\varphi_i$
is a quasicharacter of $E_i^\times$ such that the conductoral
exponent $f_i=f(\varphi_i\circ N_{E/E_i})$ of $\varphi_i\circ N_{E/E_i}$
 is greater than $1$, and such that
$\varphi_i$ is generic over $E_{i+1}$
if $i\not=d$. 
\item  $f_0<f_1<\cdots <f_{d-1}$ 
\item  If $\varphi_d$ is nontrivial,
then $f_d>f_{d-1}$.
\item  (\textbf{The toral case}) If $E_0=E$, then $\varphi_{-1}$ is the trivial
character of $E^\times$.
\item   (\textbf{The nontoral case})  If $E_0\subsetneq E$, then 
$\varphi_{-1}$ is a quasicharacter
of $E^\times$ such that $f(\varphi_{-1})=1$
 and $\varphi_{-1}$ is generic over $E_0$.
\item
 $\varphi=\varphi_{-1}\prod_{i=0}^d\varphi_i\circ N_{E/E_i}$.
 \end{itemize}
\end{definition}

\bigskip
Given $\varphi$, one has latitude in choosing the embedding of $E$ and the Howe factorization.  Later, we describe how to make these choices in ways that greatly facilitate the study of distinguished representations.
The next definition is taken from  \cite{HL1} but derived from \cite{rH}:

\begin{definition}\label{Howedatum} A  \textbf{Howe datum (for $G$)} consists of:
\begin{itemize}
\item a degree $n$ tamely ramified extension $E$ of $F$,
\item an $F$-admissible quasicharacter $\varphi : E^\times \to\C^\times$,
\item a Howe factorization of $\varphi$,
\item an $F$-linear embedding of $E$ in $\M_n( F)$.
\end{itemize}
\end{definition}

\bigskip
For our purposes, it is convenient to choose our embedding of $E$ in a way that particularly well-suited to studying tame supercuspidal representations that are distinguished with respect to orthogonal groups.   {\it To be more precise, we generally assume that our embeddings are chosen in accordance with 
Lemma \ref{Jsymtameexistence} below}.

In \cite{Y}, Yu generalizes Howe's construction to general connected, reductive $F$-groups.  Yu's construction provides the foundation for the general theory of distinguished tame supercuspidal representations in \cite{HM} and so we tend adopt Yu's notations and point of view even though we are only interested here in representations of $\GL_n (F)$.  The  next definition  is  from \cite{HM}, but based on a similar notion in \cite{Y}:

\begin{definition}\label{cuspidalGdatum}
A $5$-tuple $(\vec\bG, y, \vec r,\rho,\vec\phi)$ satisfying
the following conditions is called a \textbf{Yu datum (a.k.a., cuspidal
$G$-datum)}:

\begin{itemize}
\item[\textbf{D1.}] $\vec{\bG}$ is a tamely ramified twisted Levi sequence $\vec{\bG}
= (\bG^0, \ldots, \bG^d)$ in $\bG$  and $\bZ^0/\bZ$
is $F$-anisotropic, where $\bZ^0$ and $\bZ$ are the centers of 
$\bG^0$ and $\bG=\bG^d$, respectively.

\item[\textbf{D2.}] $y$ is a point in $A(\bG,\bT,F)$, where $\bT$ is
a tame maximal $F$-torus of $\bG^0$ and $E'$ is a Galois tamely
ramified extension of $F$ over which $\bT$ (hence $\vec\bG$)
splits.
(Here, $A(\bG,\bT,E')$ denotes the apartment in $\cB(\bG,E')$
corresponding to $\bT$ and $A(\bG,\bT,F)=A(\bG,\bT,E')\cap \cB(\bG,F)$.)

\item[\textbf{D3.}] $\vec{r} = (r_0, \ldots, r_d)$ is a sequence
of real numbers satisfying
$0 < r_0 < r_1 < \ldots < r_{d-1} \leq r_d$, if $d > 0$, and $0 \leq r_0$
if $d = 0$.

\item[\textbf{D4.}] $\rho$ is an irreducible representation of the stabilizer
$K^0 = G^0_{[y]}$ of $[y]$ in $G^0$ such that
$\rho\,|\, {G^0_{y,0^+}}$ is $1$-isotypic and
the compactly induced representation $\pi_{-1} = \ind_{K^0}^{G^0} \rho$
is irreducible (hence supercuspidal). Here, $[y]$ denotes the
image of $y$ in the reduced building of $G$.

\item[\textbf{D5.}] $\vec\phi = (\phi_0,\dots ,\phi_d)$ is a 
sequence of quasicharacters, where $\phi_i$ is a quasicharacter of $G^i$.  
We assume that $\phi_d =1$ if $r_d=r_{d-1}$ (with $r_{-1}$ defined to be 0),
 and in all other cases if $i\in \{\, 0,\dots,d\,\}$ 
then $\phi_i$ is trivial on $G^i_{y,r_i^+}$ 
but nontrivial on $G^i_{y,r_i}$.
\end{itemize}
\end{definition}

Note that the component $\vec r$ of $(\vec\bG ,y,\vec r,\rho,\vec\phi)$ is determined by the other components.  Furthermore, Yu's construction only depends on the image $[y]$ of $y$ in the reduced building of $G^0$, not on $y$.  The $4$-tuple $(\vec\bG , [y],\rho,\vec\phi )$ is canonically associated  to a Howe datum, but $(\vec\bG,y,\vec r,\rho,\vec\phi)$ is not.  In fact, for the contruction of tame supercuspidal representations of $\GL_n(F)$ the distinction between $y$ and $[y]$ is irrelevant.

The notation $[y]$ seems to express that we have in mind ``the point in the reduced building of $G^0$ corresponding to the point $y$ in the extended building.''  In other words, it seems to convey that we have made a choice of $y$.  Taking this one step further, if we are given $[y]$ then we will free to use the notation $G^0_{y,0}$ for the associated parahoric subgroup, since this object only depends on $[y]$ and since the notation $G^0_{y,0}$ is standard.  We trust that the reader will not be inconvenienced by such abuses of notation.  Note that, as in \cite{HM}, we use condition D1 to embed the building of $G^i$ in the building of $G^{i+1}$, for $i\in \{ 0,\dots ,d-1\}$.

Note that we are not assuming that the extension $E/F$ is Galois.  The extension $E'/F$ occurring in Condition D2 in Definition \ref{cuspidalGdatum} may be taken to be the Galois closure of $E/F$.

In much of this paper, we would like assume that we have a Yu datum $\Psi$ that comes from a fixed a Howe datum $\Phi$ for $G$ and we would like  to use all of the above notations, $\varphi$, $E$, $\vec\phi$, $\vec\bG$, $[y]$, $\bT$, etc., without explicitly recalling them.  With this in mind, we make the following definition for the sake of convenience:

\begin{definition}
A \textbf{$G$-datum} $\Psi$ consists of a Howe datum together with an associated Yu datum.
\end{definition}

We will also often say ``let $\xi$ be a {\bf refactorization class} of $G$-data.''  This means we are considering the set of all refactorizations (in the sense of Definition 4.19 \cite{HM}) of a given $G$-datum $\Psi$.

Suppose $\theta$ is an orthogonal involution of $G$.  The following definition is essentially in  \cite{HM}:

\begin{definition}\label{thetasymmdef}
A $G$-datum $\Psi = (\vec\bG , [y],\rho,\vec\phi)$ is \textbf{$\theta$-symmetric} if:
\begin{itemize}
\item $\theta (G^i) = G^i$, for all $i$,
\item $\theta ([y]) = [y]$,
\item $\phi_i\circ\theta = \phi_i^{-1}$, for all $i$.
\end{itemize}
\end{definition}

If $\Theta$ is a $G$-orbit of orthogonal involutions of $G$ then the methods of \cite{HM} reduce the computation of the dimension of $\Hom_{G^\theta}(\pi ,1)$, for $\theta\in \Theta$ and $\pi\in \pi(\varphi)$, to the computation of the dimension of certain $\Hom$-spaces associated to certain $\theta'$-symmetric Yu data $\dot\Psi$, where $\theta'$ is an involution in $\Theta$ and $\dot\Psi$ is a refactoriztion of a Yu datum associated to $\varphi$.  More simply put, we ultimately only need to study those Yu data $\dot\Psi$ and involutions $\theta'$ such that $\dot\Psi$ is $\theta'$-symmetric.

In this paper, we take things somewhat further.  We essentially consider $\bT$ to be part of the  datum, as indicated above, and then we show that we may reduce to the case of studying $(\Psi ,\bT , \theta)$ such that $\Psi$ is $\theta$-symmetric and $\bT$ is $\theta$-split.

\subsection{Buildings}

\subsubsection{Buildings of general linear groups}

There are many accounts of the theory of explicit models of the  Bruhat-Tits building of a general linear group over a 
nonarchimedean local field $F$.
We will give a streamlined exposition based  on \cite{GY}.

Fix a vector space $V$ of finite dimension $n$ over $F$ and, in this section, let $G= GL(V)$.  To develop the desired explicit model for the building of $G$, we need to fix a  a valuation $\omega : F\to \R\cup\{ +\infty\}$.  We will let $\gamma$ denote $\omega (\varpi)$, where $\varpi$ is any prime element of $F$.

\begin{definition}  An \textbf{additive norm on $V$ (with respect to $\omega$)} is a map
 $x :
V\to \R\cup
\{+\infty\}$ that satisfies:
\begin{itemize}
\item $x (v+w)\ge \inf \{ x
(v),x (w)\},\hbox{ for all }v,w\in V,$
\item $x (\alpha v)= \omega (\alpha
)+
 x (v),\hbox{ for all }\alpha\in F,v\in
V,$
\item
$x (v)=+\infty\hbox{ if and only if }
v= 0.$
\end{itemize}
\end{definition}

\bigskip
The set $\cB = \cB (G)$ of all valuations on $V$ provides an explicit model for the (extended) Bruhat-Tits building of $G$.  Since we are really only interested in the reduced building, we will not discuss the polysimplicial structure of $\cB$.  (See \cite{GY} for this.) 
The group $G$ acts on $\cB$ according  to 
$$(g\cdot x)(v) = x(g^{-1}v).$$

To simplify our notation, we introduce some slightly nonstandard conventions.  When we write $c\in \R^n$, we mean $c$ is the element of $\R^n$ which in component form is $(c_1,\dots ,c_n)$.  Similarly, when we say ``let $e$ be a basis of $V$,'' we mean that $e$ is ordered $F$-basis $e_1,\dots ,e_n$ of $V$.
Given  $c\in \R^n$, define  $x_e^c\in \cB$ by
$$x_e^c \left( \sum_{i=1}^n\alpha_ie_i\right) = \inf_i \{  \omega (\alpha_i)+c_i\},$$  where the $\alpha_i$'s lie in $F$.
It turns out that every element $x\in \cB$  has the form $x^c_e$, but the choice of $(c,e)$ is not uniquely determined by $x$.

Apartments in the building of $G$ are  parametrized by the maximal split tori of $G$.  Such a torus has $n$ eigenspaces of dimension one and they completely characterize the torus.  Suppose $e$ is a basis of $V$.  Then there is an associated maximal split torus $$T_e =\{ t\in G:\ t e_i \in F^\times e_i\text{ for all  }i\}.$$  Permuting the basis vectors or multiplying them by scalars does not affect the resulting torus.

Having fixed a basis $e$, the apartment $\cA (T_e)$ associated to $T_e$ is given by $$\cA (T_e) = \{ x_e^c : \  c\in \R^n\}.$$
Implicit in this definition is the fact that $T_e = T_{e'}$ implies $\cA (T_e) = \cA (T_{e'})$.

An alternate model for the building of $G$ is given in \cite{GY}.  The points in the building correspond to graded lattice chains in the following sense:

\begin{definition} A \textbf{lattice chain} on $V$ is a totally ordered (non-empty) chain $L_\bullet$ of lattices in $V$ that is stable under homotheties.  Such a lattice chain is determined by a segment:  $$L_0\supsetneq L_1\supsetneq \cdots \supsetneq L_{m-1}\supsetneq \gP_F L_0.$$ The number $m$ is called the \textbf{rank of the lattice chain}.  A \textbf{graded lattice chain} is a pair $(L_\bullet , c)$ where $L_\bullet$ is a lattice chain and $c$ is a strictly decreasing function from $L_\bullet$ to $\R$ such that
$$c(\alpha\cdot L) = \omega (\alpha) + c(L),$$ for all $\alpha\in F$ and $L\in L_\bullet$.
\end{definition}

Regarding terminology, we mention that lattice chains are also referred to as periodic lattice chains or lattice flags.   In addition, we  say we have a ``chain'' of lattices means we have lattices $L_i$ for each integer $i$ such that $L_i\supsetneq L_{i+1}$, however, the choice of how to index the lattices (that is, which lattice to call $L_0$) is irrelevant.  In some cases, it is convenient to index the lattices by elements of the image of $F^\times$ under $\omega$, as with Moy-Prasad filtrations.
The term ``period'' is sometimes used instead of ``rank.''
``Strictly decreasing'' means that as $L$ varies over the lattices in $L_\bullet$, then as $L$ gets larger $c(L)$ decreases.

There is a canonical bijection between graded lattice chains and points in $\cB$.
More precisely, if  $(L_\bullet, c)$ is a graded lattice chain then we define an additive norm 
$$x_{(L_\bullet ,c)} (v)= \max_{L\in L_\bullet ,\, v\in L} c(L).$$

Conversely, if $x\in \cB$ then a graded lattice chain $(L_x,c_x)$ is defined by declaring that the lattices in $L_x$ are the lattices
$$L_{x,r} = \{ v\in V\, :\, x(v) \ge r\},$$ for $r\in \R$, and the map $c_x$ is given by
$$c_x (L_{x,r}) = \inf_{v\in L_{x,r}} x(v).$$

When $x= x^c_e$, the corresponding lattice chain is
$$L_{x,r} = F_{r-c_1}e_1+\cdots + F_{r-c_n}e_n,$$ where $F_t = \omega^{-1}( [t,+\infty] )$ and the corresponding grading is
$$c_x(L_{x,r}) = \inf_i  \left(\gamma\left\lceil \frac{ r-c_i}{\gamma}\right\rceil +c_i\right).$$
Note that $$r\le c_x(L_{x,r})< r+\gamma .$$

Transferring the action of $G$ on $\cB$ to an action on graded lattice chains, we have
$$g\cdot (L_\bullet , c) = (g\cdot L_\bullet , g\cdot c),$$ where $gL_\bullet = (gL)_{L\in L_\bullet}$ and $(g\cdot c)(gL) = c(L)$.

Suppose $x\in \cB$ and $x = x_{(L_\bullet ,c)}$.  Then the parahoric subgroup $G_{x,0}$ associated to $x$ is the stabilizer in $G$ of $x$ or, equivalently, $(L_\bullet ,c)$.  But $g$ stabilizes $(L_\bullet , c)$ precisely when for all $i$ there exists $j$ such that $gL_i = L_j$ and $(g\cdot c)(gL_i) = L_j$.  But $(g\cdot c)(gL_i) = c(L_i)$ and, furthermore, $c(L_i) = c(L_j)$ implies $i=j$.  Therefore, $$G_{x,0}= \{ g\in G\, :\,  gL_i = L_i,\text{ for all }i\}.$$

We define an equivalence relation on $\cB$ by declaring that $x\sim y$ precisely when $x-y$ is a constant.  The set $\cB_{\rm red}= \cB_{\rm red}(G)$ of equivalence classes is serves as our  model for the reduced Bruhat-Tits building of $G$.  If $x\in \cB$, we will let $[x]$ denote the corresponding point in $\cB_{\rm red}$.

The choice of a nonzero vector $v\in V$ determines a bijection $$\cB \cong \R \times \cB_{\rm red}: x\mapsto (x(v),[x]).$$
It is common to identify $\cB$ and $\R\times \cB_{\rm red}$ even though the choice of $v$, and hence the associated bijection, is not canonical.   On the other hand, if we fix a basis for $V$, as we shall throughout most of this paper,  then we do in fact obtain a canonical choice of $v$.

In the lattice chain model, the facets in the reduced building are readily apparent.  If $x\in  \cB$ corresponds to the graded lattice chain $(L_\bullet ,c)$ then the 
the lattice chain $L_\bullet$ is what determines the facet of the image of $x$ in the reduced building $\cB_{\rm red}$.  If $m$ is the rank of $L_\bullet$  then $m-1$ is the dimension of the facet.

To give a vertex in $\cB_{\rm red}$ is equivalent to giving
 a lattice chain of the form $L_\bullet = \{ \gP_F^i L:\ i\in \Z\}$, where $L$ is some fixed lattice in $V$.
Equivalently, vertices in $\cB_{\rm red}$ correspond to homethety classes of lattices.

\subsubsection{What is needed for tame supercuspidal representations}\label{sec:whatisneeded}

For the construction of tame supercuspidal representations, we are given a tower of fields
$$F= E_d\subsetneq \cdots \subsetneq E_1\subsetneq E_0\subseteq E_{-1}= E,$$
as described in \S\ref{sec:inducing}.    Let $n_0 = [E:E_0]$.
We begin by describing how an $E_0$-linear embedding of $E$ in $\M_{n_0} (E_0)$ determines a vertex $[y]$ in $\cB_{\rm red}(G^0)$, where $G^0 = \GL_{n_0}(E_0)$.

We apply the discussion of the previous section with $F$ replaced by $E_0$ and with $V$ replaced by $E_0^{n_0}$.
Choose an $\gO_{E_0}$-basis $e_1,\dots ,e_{n_0}$ of $\gO_E$.  Given $x\in E$, define a vector
$$\v{x} =
\left( \begin{array}{c}
x_1\cr 
\vdots\cr
x_{n_0}
\end{array}\right)\in V,$$ where $x = x_1e_1+\cdots +
x_{n_0}e_{n_0}$ and $x_1,\dots ,x_{n_0}\in E_0$. 
Multiplication by $x$ is an $E_0$-linear
transformation of $E$ and hence defines a
matrix $\m{x}\in \M_{n_0}(E_0)$.  So $x\mapsto \m{x}$
is the regular representation associated
to our basis.  Given $x,y\in E$, we have
the relations $\m{x}\ \v{y} = \v{xy}$ and
$\m{x}\ \m{y} = \m{xy}$. 

The basis $e_1,\dots , e_{n_0}$ generates a lattice
$$\gO_{E_0} e_1+\cdots + \gO_{E_0} e_{n_0}$$ in $E$ whose image under $x\mapsto \v{x}$ is the lattice 
$$L = \v{\gO_E}= \gO_{E_0}^{n_0}$$ in $V= E_0^{n_0}$.  The associated lattice chain is
$$L_\bullet = \{ \gP_{E_0}^i L:\ i\in \Z\}
= \{ \v{\gP_E^i}\ :\ i\in \Z\}.$$
(We are using the fact that $E/E_0$ is unramified.)  
In this way, the choice of $E$ and an $\gO_{E_0}$-basis of $\gO_E$ determines a vertex $[y]$ in $\cB_{\rm red}(G^0)$.

It is not actually necessary to use an $\gO_{E_0}$-basis of $\gO_E$; any $E_0$-basis of $E$ can be used.  However, using integral bases greatly simplifies the form of our parahoric subgroups and algebras and their Moy-Prasad filtrations.  So we do this as a matter of convenience.

Let $\g^0=\M_{n_0}(E_0)$ be the Lie algebra of $G^0$.    The Moy-Prasad filtration of $\g^0$ is given by
$$\g^0_{y,r} = \left\{ X\in \g^0\ : X\, \v{\gP_E^k}\subset \v{\gP_E^{k+\lceil er\rceil}}\right\},$$ where $e = e(E/F)$.
These sets are clearly $\gO_E$-modules, with $\gO_E$ acting by $x\cdot X = \m{x}\, X$.  We have  $$\g^0_{y,r} = \m{\varpi_E^{\lceil er\rceil}}\, \g^0_{y,0}$$  and $$\g^0_{y,0} = \bigcap_{k\in \Z} \m{\varpi_E^k}\ \gl_{n_0}(\gO_{E_0})\ \m{\varpi_E^{-k}}.$$

We define Moy-Prasad filtration groups $G^0_{y,r}$ for $r\ge 0$ by $G^0_{y,0} = \g^0_{y,0} -\{ 0\}$ and $G^0_{y,r} = 1+\g^0_{y,r}$, for $r>0$.
It is easy to verify that
\begin{eqnarray*}
G^0_{y,0} = \GL_{n_0} (\gO_{E_0}),&\quad&  G^0_{y,0^+} = 1+ \M_{n_0}(\gP_{E_0}),\\
\sG^0_y(\f_F) := G^0_{y,0:0^+} = \GL_{n_0} (\f_{E_0}),&\quad&K^0 := G^0_{[y]} = E_0^\times \GL_{n_0} (\gO_{E_0}),\\
\end{eqnarray*}
 where $\gO_L$, $\gP_L$ and $\f_L$ denote the ring of integers, maximal ideal, and residue field, respectively, of a $p$-adic field $L$.

We have just discussed embedding $E$ in $\g^0 = \M_{n_0}(E_0)$ via the choice of an $E_0$-basis of $E$ and noted that choosing an integral basis simplifies matters.  Similarly, when $i\in \{ 0,\dots ,d-1\}$, the choice an $E_{i+1}$-basis of $E_i$ (preferably an integral basis) yields an embedding of $E_i$ in $M_{[E_i:E_{i+1}]}(E_i)$ and thus an embedding
$$\g^i = \M_{n_i}(E_i) \subset \M_{n_i}(\M_{[E_i:E_{i+1}]}(E_{i+1})) = \M_{n_{i+1}} (E_{i+1}) =\g^{i+1},$$
where $n_j = [E:E_j]$.
This yields a chain of inclusions
$$E\subseteq \g^0 \subsetneq \cdots \subsetneq\g^d = \g.$$
The composite embedding of $E$ in $\g$ is, in fact, associated to an $F$-basis of $E$, namely, the basis obtained by tensoring together the various bases associated to each link in the chain.

Taking multiplicative groups, we obtain a chain
$$T\subseteq G^0 \subsetneq \cdots \subsetneq G^d =  G.$$
Note that we have identities $G^i = \GL_{n_i}(E_i)$, not merely isomorphisms.  {\it Besides using integral bases, we will generally assume in our proofs that our bases are chosen as in Lemma \ref{Jsymtameexistence} below.}  This greatly facilitates the study of tame supercuspidal representations that are distinguished with respect to orthogonal groups.

Given a graded lattice chain on $V$, viewed as an $E_{i}$-vector space, we may regard the various lattices as $\gO_{E_{i+1}}$-modules, rather than $\gO_{E_{i}}$-modules.  This gives an embedding of $\cB (G^i)$ in $\cB (G^{i+1})$, so long as the valuation we choose on $E_i$ restricts to the valuation we choose on $E_{i+1}$.  In fact, for convenience and compatibility with the literature, we will always choose valuations that extend the standard valuation on $F$.

It is easy to see that our embeddings $\cB (G^i)\subset \cB (G^{i+1})$ are compatible with the projections from the extended buildings to the reduced buildings.  Therefore, we obtain embeddings 
$\cB_{\rm red} (G^i)\subset \cB_{\rm red} (G^{i+1})$.  It follows that our point $$[y]  
= \{ \v{\gP_E^i}\ :\ i\in \Z\}$$  may be viewed as an element of each $\cB_{\rm red} (G^i)$, however, it is not necessarily a vertex when $i\ne 0$.

If $i\in \{ 0,\dots ,d\}$  then the Moy-Prasad filtration of $\g^i$ is given by
$$\g^i_{y,r} = \left\{ X\in \g^i\ : X\, \v{\gP_E^k}\subset \v{\gP_E^{k+\lceil er\rceil}}\right\},$$ where $e = e(E/F)$.
These sets are clearly $\gO_E$-modules, with $\gO_E$ acting by $x\cdot X = \m{x}\, X$.  Note that $$\g^i_{y,r} = \m{\varpi_E^{\lceil er\rceil}}\, \g^i_{y,0}.$$  and, assuming we are using inegral bases, $$\g^i_{y,0} = \bigcap_{k\in \Z} \m{\varpi_E^k}\ \gl_{n_i}(\gO_{E_i})\ \m{\varpi_E^{-k}}.$$

We define Moy-Prasad filtration groups $G^i_{y,r}$ for $r\ge 0$ by $G^i_{y,0} = \g^i_{y,0} -\{ 0\}$ and $G^i_{y,r} = 1+\g^i_{y,r}$, for $r>0$.

With our chosen bases, it is plausible, though tedious, to give a precise block matrix description of the $\g^i_{y,r}$'s and the $G^i_{y,r}$'s.

\section{Lemmas involving tamely ramified extensions}

This section contains a collection of basic, and presumably well known, facts about tamely ramified extensions.  These facts will be needed in the proofs of our main results.  The reader should assume that all fields discussed are finite extensions of $\Q_p$ where $p$ is an odd prime.  All field extensions will be finite and tamely ramified.  If $E/F$ is such an extension and if $L$ is the maximal unramified extension of $F$ contained in $E$ then we use the standard notations $e(E/F) = [E:L]$ and $f(E/F) = [L:F]$.  Tameness means that $e(E/F)$ is not divisible by $p$.  The ring of integers of a field $F$ is denoted $\gO_F$.  The maximal ideal in $\gO_F$ is denoted $\gP_F$.

\begin{lemma}\label{tametotally}
If $E/F$ be a totally ramified, tamely ramified extension of finite degree $n$ then:
\begin{itemize}
\item $1+\gP_F = (1+\gP_F)^n = (1+\gP_E)^n\cap \gO^\times_F$.
\item $(\gO_E^\times)^n\cap \gO_F^\times = (\gO_F^\times)^n$.
\item There exist prime elements $\varpi_E$ and $\varpi_F$ of $E$ and $F$, respectively, such that $\varpi_E^n = \varpi_F$.  Given one such $\varpi_F$ the coset $\varpi_F (\gO_F^\times)^n$ is the set of all such prime elements of $F$.
\end{itemize}
\end{lemma}

\begin{proof}
Let $p$ be the characteristic of the residue field of $F$.  Then since $E/F$ is totally ramified and tamely ramified, its degree $n$ is not divisible by $p$.  Let $\varpi$ be a prime element of $F$.  If $a\in \gO_F$  then applying Hensel's Lemma to the polynomial $$f(x) = \varpi^{-1} ((1+\varpi x)^n - (1+\varpi a))$$
allows one to find $b\in \gO_F$ such that $(1+\varpi b)^n = 1+\varpi a$.  It follows that $(1+\gP_F)^n = 1+\gP_F$ and 
$$1+\gP_F = (1+\gP_F)^n \subset (1+\gP_E)^n\cap \gO_F^\times \subset 1+\gP_F.$$  This proves the first claim.

To prove the second claim, it suffices to show that $$(\gO_E^\times)^n\cap  \gO_F^\times \subset (\gO_F^\times)^n.$$
Let $\mu$ be the set of roots of unity in $E$ whose order is relatively prime to $p$.  Then $\mu$ is also the set of roots of unity in $F$ whose order is relatively prime to $p$.  We have $\gO_E^\times  = \mu \cdot (1+\gP_E)$ and $\gO_F^\times = \mu\cdot (1+\gP_F)$.
Suppose $u\in (\gO_E^\times)^n\cap \gO_F^\times$.  
Then there exists $\zeta\in \mu$ and $v\in 1+\gP_E$ such that $u = (\zeta v)^n$.  We have $\zeta^{-n}u = v^n \in (1+\gP_E)\cap \gO_F^\times = 1+\gP_F$.  Since $1+\gP_F = (1+\gP_F)^n$, we can assume $v\in 1+\gP_F$.  Thus $u = (\zeta v)^n\in (\gO_F^\times)^n$.  The second claim follows.

The fact that there exist prime elements $\varpi_E$ and $\varpi_F$ of $E$ and $F$ such that $\varpi_E^n = \varpi_F$ is well known.  
Suppose we have two such pairs $(\varpi_E,\varpi_F)$ and $(\varpi'_E,\varpi'_F)$.  Let $w = \varpi'_E \varpi_E^{-1}$.  Then $w^n = \varpi'_F\varpi_F^{-1} \in (\gO_E^\times)^n\cap \gO_F^\times = (\gO_F^\times )^n$.  So $\varpi'_F \in \varpi_F(\gO_F^\times)^n$.
Conversely, given $\varpi''_F = \varpi_F t^n$, with $t\in \gO_F^\times$, then $(t\varpi_E , \varpi''_F)$ is another pair of the same type as $(\varpi_E,\varpi_F)$.  This completes the proof.
\end{proof}

\begin{lemma}\label{EEprimeF}
Let $L$ be a finite totally ramified, tamely ramified extension of $F$ of degree $m$ and let $E$  be a ramified quadratic extension of $L$. 
Let $n=2m$ and let $\varpi_F$ be a representative of the coset $\varpi_F (\gO_F^\times)^n$ of all prime elements of $F$ that have an $n$-th root $\varpi_E$ in $E$.  Then $\varpi_F$ has an $m$-th root $\varpi_L$ in $L$.
\end{lemma}

\begin{proof}
First, we show that there exists a triple $(\varpi_F,\varpi_L,\varpi_E)$ of prime elements for $F$, $L$ and $E$, respectively, such that $\varpi_L^m =\varpi_F$ and $\varpi_E^2= \varpi_L$.

Choose a prime element $\varpi_F$ in $F$ that has an $m$-th root $\varpi_L$ in $E$.  If $\varpi_L$ has a square root $\varpi_E$ in $E$ then $(\varpi_F , \varpi_L,\varpi_E)$ is a triple of the desired type.

On the other hand, if $\varpi_L$ does not have a square root in $E$ then we may choose a root of unity $u$ in $L$ of order prime $p$ such that $u\varpi_L$ has a square root in $E$.  But since $L/F$ is totally ramified, it must be the case that $u$ lies in $F$.  We now rename $u\varpi_L$ and $u^m\varpi_F$ as $\varpi_L$ and $\varpi_F$, respectively, and take $\varpi_E$ to be a square root of the new element $\varpi_L$.  This proves the existence of the desired triple.

Given such a triple $(\varpi_F , \varpi_L,\varpi_E)$, Lemma \ref{tametotally} implies that the coset\break $\varpi_F (\gO_F^\times)^n$ is the set of all prime elements of $F$ that have an $n$-th root in $E$.  But since we have chosen $\varpi_F$ so that it has an $m$-th root in $L$, it follows that every element of the latter coset has an $m$-th root in $L$.
\end{proof}

For general field extensions of even degree, it is not necessarily the case that the extension field must contain a quadratic extension of the base field.  For example, there are quartic extensions of $\Q$, called primitive quartic fields, that do not contain any quadratic extensions of $\Q$.

On the other hand, for finite fields it is obviously true that an even degree extension $\F_{q^n}/\F_q$ contains the quadratic extension $\F_{q^2}$ of $\F_q$.
For the field extensions that we are considering in this section, we also have a positive result:

\begin{lemma}\label{fieldsA} If $E$ is a finite, tamely ramified extension of $F$ of even degree then $E$ contains a quadratic extension of $F$.
\end{lemma}

\begin{proof} 
Consider first the special case in which $E/F$ is a totally ramified, tamely ramified extension of even degree $n=2m$.  Then $E = F[\varpi_E]$, where $\varpi_E$ is a prime element in $E$ that is an $n$-th root of a prime element in $F$.  Then $E$ contains the ramified quadratic extension $F[\varpi_E^m]$ of $F$.

Now assume $E/F$ is tamely ramified and $f(E/F)$ is even.  Let $L$ be the maximal unramified extension of $F$ contained in $E$.  This is the unique unramified extension of $F$ of degree $f(E/F)$ and it contains the unique unramified quadratic extension of $F$.  So our claim holds in this case.

Now suppose $E/F$ is tamely ramified, $f(E/F)$ is odd, and $e(E/F)$ is even.   Let $L$ be the maximal unramified extension of $F$ contained in $E$.  We know that $E$ contains a totally ramified quadratic extension $K$ of $L$.  It suffices to show that $K$ contains a ramified quadratic extension of $F$.  We may as well replace $E$ by $K$ or, equivalently, assume $e(E/F) =2$.  Then $E= L[\varpi_E]$, where $\varpi_E$ is the square root of a prime element $\varpi_L$ in $L$.  Let $\varepsilon$ be a nonsquare root of unity in $F$.  In other words, $\varepsilon^{(q_F-1)/2}=-1$, where $q_F$ be the order of the residue field of $F$.   
Since $f(E/F)$ is odd, $\varepsilon^{(q_F^{f(E/F)}-1)/2}=-1$ and thus $\varepsilon$ is also a nonsquare root of unity in $L$.
Therefore, given a prime element $\varpi_F$ in $F$, the element $\varpi_L$ must either lie in $\varpi_F(\gO_{L}^\times)^2$ or $\varepsilon\varpi_F(\gO_{L}^\times)^2$.  It follows that either $F[\sqrt{\varpi_F}]$ or $F[\sqrt{\varepsilon\varpi_F}]$ is  a ramified quadratic extension of $F$ contained in $E$.
\end{proof}

Given a finite extension $E/F$, let $y_{E/F}$ be defined by letting $y_{E/F}-1$ be the number of quadratic extensions of $F$ contained in $E$.  Thus, in the present setting, $y_{E/F}$ must be 1, 2 or 4.  The relevance of $y_{E/F}$ to the study of distinguished tame supercuspidal representations is discussed in 
\S\ref{sec:generalities}.

\begin{lemma}\label{fieldsB}  Let $E$ be a tamely ramified degree $n$ extension of $F$ such that $e = e(E/F)$ and $f = f(E/F)$ are even.  Let $L$ be the maximal unramified extension of $F$ contained in $E$.  Then $y_{E/F}=4$ if and only if 
 $(\gO_L^\times)^2 F^\times$ contains a prime element $\varpi_L$ in $L$ that has an $e$-th root in $E$.
\end{lemma}

\begin{proof} When $\varpi_E\in E$ is an $e$-th root of a prime element $\varpi_L$ in $L$, we let $\sqrt{\varpi_L} = \varpi_E^{e/2}$.  The  element $\varpi_L$ lies in $(\gO_L^\times)^2 F^\times$ precisely when it has the form $\varpi_L = u^2 \varpi_F$, where $u\in \gO_L^\times$ and $\varpi_F$ is a prime element of $F$.  In the latter situation,  $u^{-1} \sqrt{\varpi_L}$ is a square root $\sqrt{\varpi_F}$ of $\varpi_F$ that lies in $E$.  Thus $F[\sqrt{\varpi_F}]$ is a ramified quadratic extension of $F$ that is contained in $E$.  Since $f$ is even, $E$ must also contain an unramified quadratic extension of $F$ and thus $y_{E/F}=4$.

Now suppose $y_{E/F}=4$.  Then, since $E$ contains ramified quadratic extensions of $F$,  there exists a prime element $\varpi_F$ in $F$ with  a square root $\sqrt{\varpi_F}$ in $E$.  There also exists a prime element $\varpi_L$ in $L$ with an $e$-th root $\varpi_E$ in $E$.   Let $\sqrt{\varpi_L} = \varpi_E^{e/2}$ and let    $u  = \sqrt{\varpi_L}(\sqrt{\varpi_F})^{-1}\in \gO_E^\times$.   The element $\varpi_{L} \varpi_F^{-1} = u^2$ is a element of $\gO_{L}^\times\cap (\gO_E^\times)^2$.  The element $u$ must lie in $\gO_{L}^\times$, since otherwise $L[u]$ would be an unramified quadratic extension of $L$ contained in $E$.  It follows that $\varpi_L\in (\gO_L^\times)^2 F^\times$, which completes the proof.
\end{proof}

\begin{lemma}\label{fieldsC}
Let $L$ be a tamely ramified extension of $F$ with even ramification degree $e(L/F)$.  Let $E$ be an unramified quadratic extension of $L$.  Then $y_{E/F} =4$.
\end{lemma}

\begin{proof}
Let $e = e(E/F) = e(L/F)$.
Let $K$ be the maximal unramified extension of $F$ contained in $L$.  Choose a prime element $\varpi_K$ in $K$ that has  an $e$-th root $\varpi_L \in L$.  Let $\sqrt{\varpi_{K}} = \varpi_L^{e/2}$.   Then $K [\sqrt{\varpi_{K}}]$ is a ramified quadratic extension of $K$ contained in $L$.  The unique unramified quadratic extension of $K [\sqrt{\varpi_{K}}]$ is contained in $E$.  It suffices to solve our problem with $E$ replaced by the latter quadratic extension of $K [\sqrt{\varpi_{K}}]$.

In other words, we  now assume $e=2$ and $L=K [\sqrt{\varpi_{K}}]$.  
Since $f(E/F)$ is even, it must be the case that $E$ contains an unramified quadratic extension of $F$.  Therefore it suffices to show that $E$ contains a ramified quadratic extension of $F$.
  
Fix a prime element $\varpi_F$ in $F$.  
If $\varpi_F$ has a square root $\sqrt{\varpi_F}$ in $E$ then $F[\sqrt{\varpi_F}]$ is a ramified quadratic extension of $F$ in $E$ and we are done.
Therefore, we assume $\varpi_F$ does not have a square root in $E$.

Since $\varpi_F$ does not have a square root in $E$, the field $L$ cannot be obtained by adjoining a square root of $\varpi_F$ to $K$.  
Therefore, we can choose a nonsquare unit $u$ in $K$ such that the prime element $\varpi_K = u\varpi_F$ has a square root $\varpi_L$ in $L$.  Since $E/L$ is unramified quadratic, there exists a square root $v$ of $u$ in $E$.  We have $\varpi_L^2 = \varpi_K = u\varpi_F = v^2\varpi_F$.   Thus $\varpi_Lv^{-1}$ is a square root of $\varpi_F$ in $E$.  We deduce that $E$ contains a ramified quadratic extension of $F$ and thus the proof is complete.
\end{proof}

\begin{lemma}\label{quadraticinduction}
Suppose $Q_1$ is a quadratic extension of $F$ and $Q_2$ is a quadratic extension of $Q_1$.  Then one of the following holds:
\begin{enumerate}
\item[{\rm (1)}]\quad $y_{Q_2/F} =2$ and $(N_{Q_2/F}(Q_2^\times) )(F^\times)^2= N_{Q_1/F}(Q_1^\times)$.
\item[{\rm (2)}]\quad  $y_{Q_2/F}=4$ and $(N_{Q_2/F}(Q_2^\times) )(F^\times)^2= (F^\times)^2$.
\end{enumerate}
\end{lemma}

\begin{proof}
Since $[Q_2:F]$ is even, Lemma \ref{fieldsA} implies $y_{Q_2/F}=2$ or $4$.
Suppose $y_{Q_2/F}=4$.  Then there must exist a quadratic extension $Q_0$ of $F$ that is contained $Q_2$ and is distinct from $Q_1$.  By transitivity of norms, $N_{Q_2/F}(Q_2^\times)$ must be contained in both $N_{Q_0/F}(Q_0^\times)$ and $N_{Q_1/F}(Q_1^\times)$.  Since
$$N_{Q_0/F}(Q_0^\times)\cap N_{Q_1/F}(Q_1^\times)= (F^\times)^2,$$ we see that $y_{Q_2/F} =4$ implies that $(N_{Q_2/F}(Q_2^\times) )(F^\times)^2= (F^\times)^2$.

It now suffices to show that if $y_{Q_2/F}=2$ then $(N_{Q_2/F}(Q_2^\times) )(F^\times)^2= N_{Q_1/F}(Q_1^\times)$.
So let us assume $y_{Q_2/F}=2$ and thus $Q_1/F$ is the unique quadratic extension of $F$ contained in $Q_2$.  We observe that the image of $N_{Q_1/F}(Q_1^\times)$ in $F^\times /(F^\times)^2$ is a subgroup of order two that contains the image of $N_{Q_2/F}(Q_2^\times)$ in $F^\times /(F^\times)^2$.  It suffices therefore to show that the image of $N_{Q_2/F}(Q_2^\times)$ in $F^\times /(F^\times)^2$ is nontrivial.

There are several cases to consider.  First, suppose both $Q_1/F$ and $Q_2/Q_1$ are ramified.  If $\varpi_{Q_2}$ is any prime element in $Q_2$ then $N_{Q_2/F}(\varpi_{Q_2})$ is a prime element in $F$.  This gives an element of  $N_{Q_2/F}(Q_2^\times)$ whose image in $F^\times /(F^\times)^2$ is nontrivial.

Now, suppose both $Q_1/F$ and $Q_2/Q_1$ are unramified.  Then $N_{Q_2/F}(\gO_{Q_2}^\times ) = \gO_F^\times$.  Thus we can choose an element of  $N_{Q_2/F}(\gO_{Q_2}^\times)$ that is a nonsquare unit in $F$.

Next, suppose $Q_1/F$ is unramified and $Q_2/Q_1$ is ramified.  
There exists a prime element $\varpi_{Q_1}$ in $Q_1$ that has a square root $\varpi_{Q_2}$ in $Q_2$.  It must be the case that $\varpi_{Q_1}\not\in (\gO_{Q_1}^\times)^2 F^\times$, since otherwise  we would have $y_{Q_2/F}=4$.  
We can choose a prime element $\varpi_F$ in $F$ and a nonsquare unit $u$ in $Q_1$ such that $\varpi_{Q_1} = u\varpi_F$.    
Now, $N_{Q_2/F}(\varpi_{Q_2}) =N_{Q_1/F}( -\varpi_{Q_1}) = N_{Q_1/F}(u)\cdot \varpi_F^2$ .
But the latter element has nontrivial image in $F^\times / (F^\times)^2$ since $N_{Q_1/F}$ determines an isomorphism
$$\gO_{Q_1}^\times /(\gO_{Q_1}^\times)^2 \to \gO_F^\times/ (\gO_F^\times)^2$$ of groups of order two.

Finally, we are left with the case in which $Q_1/F$ is ramified and $Q_2/Q_1$ is unramified.  But this case is impossible since $y_{Q_2/Q_1}=2$.
\end{proof}

\begin{lemma}\label{ylemma}
Suppose $L$ is a finite tamely ramified extension of $F$ and suppose $E$ is a finite tamely ramified extension of $L$ of even degree.  If $y_{E/F}=2$ then $y_{E/L}=2$.
\end{lemma}

\begin{proof}
Assume $y_{E/F}=2$.  Since $[E:L]$ is even, Lemma \ref{fieldsA} implies that $y_{E/L}=2$ or $4$.  Suppose $y_{E/L}=4$.  Then $E$ contains square roots of all elements of $L$.  In particular, it contains square roots of all elements of $F$.  This implies $y_{E/F}=4$ which is a contradiction.
\end{proof}

\begin{lemma}\label{normimage}  If $E$ is a finite, tamely ramified extension of $F$ of degree $n$ then the image of $N_{E/F}(E^\times)$ in $F^\times /(F^\times)^2$ is:
\begin{itemize}
 \item $F^\times / (F^\times)^2$, if $y_{E/F}=1$,
 \item $N_{L/F}(L^\times)/(F^\times)^2$, if $y_{E/F}=2$ and $L$ is the  unique quadratic extension of $F$ contained in $E$,
\item trivial, if $y_{E/F}=4$.
 \end{itemize}
Consequently, the index of  $N_{E/F}(E^\times)\cdot (F^\times)^2$  in $F^\times$ is  $y_{E/F}$.
\end{lemma}

\begin{proof}
Suppose $L$ is a quadratic extension of $F$ that is contained in $E$.  It follows from Lemma 5.6 {\bf [HL1]} that $L^\times /((L^\times)^2F^\times)$ has order two.  Therefore, the norm $N_{L/F}$ must induce an isomorphism
$$L^\times/((L^\times)^2 F^\times)\cong N_{L/F}(L^\times)/(F^\times)^2.$$

If $L_1$ and $L_2$ are distinct quadratic extensions of $F$ contained in $E$ then the images of $N_{L_1/F}(L_1^\times)$ and $N_{L_2/F}(L_2^\times)$ in $F^\times/ (F^\times)^2$ have trivial intersection.  Consequently, the image of $N_{E/F}(E^\times)$ in $F^\times /(F^\times)^2$ must be trivial.
So if $y_{E/F}=4$ then $N_{E/F}(E^\times)$ must have trivial image in $F^\times /(F^\times)^2$.

Now suppose $y_{E/F}=2$ and $L$ is the  unique quadratic extension of $F$ contained in $E$.   Then, according to Lemma \ref{fieldsA}, there exists a tower
$$F=Q_0\subset Q_1\subset \dots \subset Q_r\subset E,$$ where $Q_1= L$ and each extension of successive terms is  quadratic, except for for $E/Q_r$ which is an extension of odd degree.
The homomorphism $$E^\times/(E^\times)^2\to Q_r^\times /(Q_r^\times)^2$$ induced by $N_{E/Q_r}$ is an isomorphism (since $N_{E/Q_r}(Q_r^\times ) (Q_r^\times)^2 = Q_r^\times$). 
It follows that the image of $N_{E/Q_{r-1}}(E^\times)$ in $Q_{r-1}^\times /(Q_{r-1}^\times)^2$ has order two.  If $r=1$, we are done.  Otherwise, we apply Lemmas \ref{quadraticinduction} and  \ref{ylemma} to deduce that 
the image of $N_{E/Q_{r-2}}(E^\times)$ in $Q_{r-2}^\times /(Q_{r-2}^\times)^2$ has order two.   Repeatedly applying Lemmas \ref{quadraticinduction} and \ref{ylemma} in this way, we deduce that
the image of $N_{E/F}(E^\times)$ in $F^\times /(F^\times)^2$ has order two.   But  by transitivity of norms, the latter image is contained in $N_{L/F}(L^\times)$.
Therefore,  the image of $N_{E/F}(E^\times)$ in $F^\times /(F^\times)^2$ is $N_{L/F}(L^\times)$.

Finally, if $E/F$ contains no quadratic extensions of $F$ then $n$ must be odd.  In this case, as we have just argued for $E/L$, we can show that $N_{E/F}$ maps $E^\times$ onto $F^\times /(F^\times)^2$.
\end{proof}

\section{$J$-symmetric embeddings}
\label{sec:Jsymmetric}

In this section, we develop a geometric algebraic theory to study  variants of the question:
\begin{quote}
When does a degree $n$ extension $E$ of a field $F$ have an $F$-linear embedding in the set of symmetric matrices in $\M_n (F)$?
\end{quote}
The latter question has been studied especially for its applications to the theory of symmetric bilinear forms, but, since we are interested in rather different applications, we need to refine various aspects of the theory.  The basic results regarding existence of embeddings are mostly known and we refer to \cite{B} for more details and references to the literature.

The trace form $(x,y)\mapsto \tr_{E/F} (xy)$ on $E$ has tended to play a central role in the theory of embeddings of the type just mentioned.  However, as discussed in 
\S\ref{sec:ourexamples}, we deemphasize the use of trace forms and  pursue a different approach than that which was used in \cite{HL1} in the case of odd $n$.  The two approaches are contrasted somewhat in \S\ref{sec:splitTgeneralities}.  For our purposes, the approach we use is simpler, more illuminating, and leads to more canonical constructions.

Though the results in this section are rather simple, the reader may be surprised to see how frequently these results are used throughout the paper in seemingly different contexts.

We assume throughout this section  that our fields do not have characteristic two.  Though we often consider general fields (whose characteristic is not two), we really are only interested in finite extensions of a field $\F_p$ or $\Q_p$, where $p$ is an odd prime.

\subsection{Preliminaries}

Let $E/F$ be a field extension of (finite) degree $n$.   Let $G = \GL_n (F)$ and let $\g = \M_n(F)$ (viewed as an associative algebra, not a Lie algebra).  Fix an $F$-basis $e_1,\dots ,e_n$  of $E$.  Then to each $y\in E$ there is a column vector
$$\v{y}= \begin{pmatrix}y_1\\ \vdots\\ y_n\end{pmatrix}\in F^n$$
such that $$y=\sum_{i=1}^n y_i e_i.$$
  We use $y\mapsto \v{y}$ to identify $E$ with $F^n$ and then, given $x\in E$,  we let $\m{x}$ denote the matrix of the $F$-linear transformation $y\mapsto xy$ on $E$.  We then have
 $$\m{x}\, \v{y} = \v{xy},$$  for all $x,y\in E$, and
$$xe_j = \sum_{i=1}^n x_{ij} e_i$$ for all $j$.

Every $F$-embedding of $E$ in $\g$ has the form $x\mapsto \m{x}$ for a suitable choice of $F$-basis of $E$.

Let $$J =J_n =\begin{pmatrix}
& &1\\
&\raisebox{-.1ex}{.} \cdot \raisebox{1.2ex}{.}&\\
1&&
\end{pmatrix}\in G .$$

\begin{definition}
A matrix $X\in \g$ is \textbf{$J$-symmetric} if $J\cdot {}^t X\cdot J = X$ or, equivalently, if it is symmetric about its anti-diagonal (consisting of the matrix entries $X_{ij}$ with $i+j = n+1$).  An $F$-embedding of $E$ in $\g$ is \textbf{$J$-symmetric} if its image consists only of $J$-symmetric matrices.
\end{definition}

The notion of a $J$-symmetric embedding may be generalized by replacing transpose by the order two anti-automorphism $\sigma_\nu$ of $\g$  given by
$$\sigma_\nu (X) = \nu^{-1}\ {}^t X\ \nu,$$
where $\nu$ is any fixed symmetric matrix in $G$.

\begin{definition}
If $\nu\in G$ is symmetric then a matrix $X\in \g$ is  \textbf{$\nu$-symmetric} if $\sigma_\nu (X) = X$.  
An $F$-embedding  of $E$ in $\g$  is \textbf{$\nu$-symmetric} if its image consists of $\nu$-symmetric matrices.
\end{definition}

We will see that for the field extensions $E/F$ of interest to us, $J$-symmetric embeddings of $E$ always exist, whereas $\nu$-symmetric embeddings do not exist for arbitrary symmetric matrices $\nu$.  Moreover, once one fixes a $J$-symmetric embedding of $E$, it is easy to describe the $\nu$-symmetric embeddings for arbitrary $\nu$ according to the following:

\begin{proposition}\label{symmetricembedding}
Suppose $x\mapsto \m{x}$ is a $J$-symmetric embedding of $E$ in $\g$.  Then the set of symmetric matrices $\nu$ in $G$ such that $x\mapsto \m{x}$ is $\nu$-symmetric is identical to $J\m{E^\times}$.  If $x\in E^\times$ and $\nu = J\m{x}$ then $\det (\nu) = (-1)^{n(n-1)/2}N_{E/F}(x)$.
\end{proposition}

\begin{proof}
Assume $\nu\in G$ is symmetric and  $x\mapsto\m{x}$ is a $J$-symmetric embedding of $E$ in $\g$.  If $x\in E$ then ${}^t (J\m{x}) = {}^t\m{x} J = J\m{x}$.    So the elements of $J\m{E}$ are all symmetric matrices.  The given embedding is $\nu$-symmetric precisely when every $x\in E$ satisfies $\m{x}= \nu^{-1}({}^t\m{x})\nu = \nu^{-1}J\m{x}J\nu$.  In other words, the embedding is $\nu$-symmetric exactly when $J\nu$ centralizes $\m{E}$.  Since $\m{E}$ is its own centralizer, our first assertion follows.  The second assertion reduces to the well-known fact that if $E$ is embedded as an $F$-subalgebra of $\g$ then the restriction of the determinant to the image of $E$ agrees with the norm map $N_{E/F}$.
\end{proof}

The previous result can be used to recover the results about the existence of $1$-symmetric embeddings in \cite{B}, once we establish the existence of $J$-symmetric embeddings for our extensions $E/F$.

\bigskip
For us, one of the most important properties of $J$-symmetric embeddings is their transitivity with respect to towers of fields:

\begin{lemma}\label{composite}
Let $L$ be an intermediate field of  $E/F$ with $r = [L:F]$ and $s = n/r = [E:L]$.
Suppose $L$ is embedded in $\M_r(F)$ via a $J_r$-symmetric embedding and suppose $E$ is embedded in $\M_s(L)$ via a $J_s$-symmetric embedding.  Then the composite embedding of $E$ in $\g= \M_n(F) = \M_s(\M_r(F))$ is $J_n$-symmetric.
\end{lemma}

\begin{proof}
Embed $L$ in $\M_r(F)$ in a $J_r$-symmetric way.  Thus,
$$J_r\ {}^tx\ J_r = x,$$ for all $x\in L$.  (For simplicity, we are suppressing  the double underline notation.)

Embed $E$ in $\M_s (L)$ in $J_s$-symmetric way.  Then 
$$J_s \ {}^T x\ J_s = x,$$ for all $x\in E$, where $x\mapsto {}^T x$ is the transpose in $\M_s(L)$.

Since $$\M_s (L)\subset \M_s(\M_r(F))= \g,$$ if $x\in \M_s(L)$ then one needs to distinguish between the transpose ${}^T x$ in $\M_s(L)$ and the transpose ${}^t x$ in $\g$.  

The relation between these two transposes can be precisely expressed as follows.  Let $x\in \M_s(L)$.  Then $x = (x_{ij})$, where $x_{ij}\in L$ and $1\le i,j\le s$. We have
$${}^T x = {}^T (x_{ij}) = (x_{ji}).$$
and
$${}^t x = {}^t (x_{ij}) = ({}^t x_{ji}) = (J_r x_{ji}J_r) = J_r^s \ {}^T x \ J_r^s,$$ where $J^s_r$ is the block diagonal matrix $$J_r^s= J_r\oplus\cdots \oplus J_r$$ in $G$ whose diagonal blocks are all $J_r$.
We observe that  $$J_n = J_sJ_r^s= J_r^s J_s.$$
Therefore, if $x\in E$ then we have 
$$J_n \ {}^t x\ J_n = J_n\ J_r^s \ {}^Tx\ J_r^s\ J_n = J_s \ {}^T x \ J_s = x.$$   This proves our claim.
\end{proof}

The computations in the latter proof (and the notations $J_r^s$ and ${}^T x$) are used implicitly and explicitly in what follows.

\subsection{Existence and construction of $J$-symmetric integral embeddings}

For a given symmetric matrix $\nu\in G$, we observe that if $X_1,\dots , X_s\in \g$ are $\nu$-symmetric matrices that commute with each other then their product $X_1\cdots X_s$ is also $\nu$-symmetric.  In particular, powers of $\nu$-symmetric matrices are $\nu$-symmetric.  When $\nu = J$, this yields the following basic construction of $J$-symmetric embeddings:

\begin{lemma}\label{constructJsym}
Suppose $n>1$ is an integer and suppose $\tau$ is an element of $F$ such that $x^n -\tau$ is irreducible over $F$.  Let $E = F[\omega ]$, where $\omega$ is an $n$-th root of $\tau$.  Then the embedding of $E$ in $\g$ associated to the basis $1,\omega,\omega^2,\dots ,\omega^{n-1}$ is $J$-symmetric.
\end{lemma}

\begin{proof}
The matrix that corresponds to the image of $\omega$ in $\g$ is
$$\omega =\begin{pmatrix}
0&0&\cdots&0&\tau\\
1&0&\cdots&0&0\\
0&1&\ddots&&\vdots\\
\vdots&\ddots&\ddots&\ddots&\vdots\\
0&\cdots&0&1&0
\end{pmatrix}.$$
The latter matrix is clearly symmetric about its anti-diagonal.  Thus it is $J$-symmetric and so are its powers, as are all $F$-linear combinations of these powers.  Therefore, the given embedding is $J$-symmetric.
\end{proof}

\begin{example}
If $E/F$ is quadratic and $E = F[\sqrt{\tau}]$ then
$$x+y\, \sqrt{\tau} = \begin{pmatrix}x&y\,\tau\\ y&x\end{pmatrix}$$ is a $J$-symmetric embedding of $E$ in $\M_2(F)$.
\end{example}

\begin{example}\label{totallyex}
If $E/F$ is  a totally ramified, tamely ramified degree $n$ extension of fields that are finite extensions of $\Q_p$ with $p$ odd then $E$ has a $J$-symmetric embedding in $\g$.  In this case, we take $\tau$ to be a suitable prime element of $F$ and use Eisenstein's criterion to show that $x^n-\tau$ is irreducible.
\end{example}

If $E/F$ is unramified then we cannot necessarily use Lemma \ref{constructJsym}. In fact, the unramified case essentially reduces to the case of finite fields, where, again, Lemma \ref{constructJsym} does not construct all $J$-symmetric embeddings.   For example, if $E = \F_{27}$ and $F = \F_3$, there is no polynomial $x^3-\tau$ that is irreducible over $\F_3$ with $\tau\in \F_3$.  
Fortunately, for finite fields  we can appeal to:

\begin{lemma} Let $F = \F_q$ and $E= \F_{q^n}$, where $q$ is a power of an odd prime.  Then there exists a $J_n$-symmetric embedding of $E$ in $\g$.
\end{lemma}

\begin{proof}
We adapt the proof of Lemma 5.13 of [HL1] to finite fields.  Let $\beta$ be a primitive $(q^n-1)$-root of $1$ in $E$.  Let $f$ be the minimal polynomial of $\beta $ over $F$.
Let $e_i = \beta^{i-1}$ for $i\in \{ 1,\dots ,n\}$.   This gives an $F$-basis of $E$.  Let $a= f'(\beta)^{-1}$ and define a symmetric matrix $\nu^a$  in $\g$ by 
$$(\nu^a)_{ij} = \tr_{E/F}(ae_ie_j) = \sum_{k=1}^n \sigma_k(a)\, \sigma_k(e_i)\, \sigma_k (e_j),$$ where $\sigma_1,\dots ,\sigma_n$ are the distinct $F$-embeddings of $E$ in an algebraic closure $\overline{F}$ of $F$ containing $E$.  The matrix identity
$$\nu^a = (\sigma_j(e_i))\ (\sigma_i (a)\delta_{ij})\ (\sigma_i (e_j))$$ in $\M_n (\overline{F})$ implies $$\det (\nu^a) = N_{E/F}(a\, f' (\beta))\ (-1)^{n(n-1)/2} = (-1)^{n(n-1)/2} = \det J_n.$$
So $\nu^a$ and $J_n$ must be similar symmetric matrices in $\GL_n (F)$.  Therefore, we can choose $g\in \GL_n (F)$ such that $g\, \nu^a\, {}^t g = J_n$.  We use $g$ to change our basis:  $e'_i = \sum_j g_{ij}e_j$.  The embedding associated to the latter basis is then $J$-symmetric.
\end{proof}

The fact that $J$-symmetric embeddings always exist for tamely ramified  extensions (of fields that are finite extensions of $\Q_p$ with $p$ odd)  follows  from Lemma 5.14 \cite{HL1}.  The proof in \cite{HL1} uses the transitivity of $J$-symmetric embeddings (Lemma \ref{composite} in this paper) to reduce to the totally ramified case (as discussed above in Example \ref{totallyex}) and the unramified case (treated in Lemma 5.13 \cite{HL1}).  

\begin{definition}
If $E/F$ is  tamely ramified degree $n$ extension of fields that are finite extensions of $\Q_p$ then a \textbf{$J$-symmetric integral embedding} of $E$ in $\g$ 
is a $J$-symmetric $F$-embedding of $E$ in $\g$ associated to an $\gO_F$-basis of $\gO_E$.   In this case, the basis is called a \textbf{$J$-symmetric integral basis} of $E/F$.
\end{definition}

The following result is essentially implicit in \cite{HL1}:

\begin{lemma}\label{Jsymtameexistence}
If $E/F$ is  tamely ramified degree $n$ extension of fields that are finite extensions of $\Q_p$ with $p$ odd then 
there exists a $J$-symmetric integral embedding of $E$ in $\g$.  Given a tower of intermediate fields
$$F=E_d \subset E_{d-1}\subset \cdots \subset E_0\subset E_{-1}=E,$$ such that each $E_i$ with $i\in \{ -1,\dots ,d-1\}$ is embedded in $\M_{[E_i:E_{i+1}]} (E_{i+1})$ via a $J$-symmetric integral embedding,  the composite embedding of $E$ in $\g$ is a $J$-symmetric integral embedding.
\end{lemma}

\begin{proof}
The proof of Lemma 5.14 \cite{HL1} shows that $E$ has a $J$-symmetric integral embedding in $\g$.  The second assertion follows from Lemma \ref{composite} and the fact that
 if $L$ is an intermediate field of $E/F$, then the tensor product of an integral basis of $L/F$ with an integral basis of $E/L$ is an integral basis of $E/F$.
\end{proof}

\subsection{Parahoric subalgebras and Moy-Prasad filtrations}\label{sec:parahorics}

Suppose we are given and $F$-admissible quasicharacter $\varphi : E\to \C^\times$ of the multiplicative group of a degree $n$ tamely ramified extension $E$ of $F$.  To execute Howe's construction of a representation in the associated isomorphism class of tame supercuspidal representations, one must choose how to embed $E$ in $\g$.  
In Lemma \ref{Jsymtameexistence}, we have described one system of embedding $E$ that is especially convenient for the purposes of studying distinguished representations.  Unfortunately, with this approach the relevant parahoric groups and the associated Moy-Prasad filtrations do not have the simplest possible block matrix form.

On the other hand, if one desires filtration groups with simple block matrix descriptions, there is a standard way to proceed.   If $L$ is the maximal unramified extension of $F$ contained in $E$ then one embeds $E$ using the basis $\{ e_i\}$ that is the tensor product basis $\{ \alpha_j\beta_k\}$, where 
\begin{itemize}
\item $\{\alpha_j\}$ is a $J$-symmetric integral basis of $E/L$ with $\alpha_j\in \gP_E^{j-1}-\gP_E^j$,
\item  $\{ \beta_k\}$ is a $J$-symmetric integral basis for $L/F$ with $\beta_k\in \gO_L^\times$.
\end{itemize}
Note that $e_1,\dots , e_n$ are ordered as follows: $\alpha_1\beta_1,\alpha_1\beta_2,\dots ,\alpha_2\beta_1,\alpha_2\beta_2,\dots$.
We also remark that normally the basis $\{ \alpha_j\}$ is constructed by using consecutive powers of a prime element of $E$.

Fortunately, it is easy to check that if one permutes the elements of one of our bases in a suitable way then one obtains a basis of the standard form just described.  Thus, after a suitable conjugation by a permutation matrix, our filtration groups simplify.

Let us now quickly recall the block matrix form of the filtration algebras associated to a basis of the form just described.  To get the associated filtration groups and for more details, the reader can refer to \S\ref{sec:whatisneeded}.

Let $y\in \cB (G)$ be associated to our given embedding.
The parahoric algebra at $y$ is
$$\g_{y,0} =
\begin{pmatrix}
\M_{f} (\gO_F)&\M_f (\gP_F)&\cdots&\M_f(\gP_F)\\
\vdots&\kern-3em\ddots&\kern-1em\ddots&\vdots\\
\vdots&&\kern-2em\ddots&\M_f(\gP_F)\\
\M_f(\gO_F)&\cdots&\kern-2em\cdots&\M_{f}(\gO_F)
\end{pmatrix},$$
where $f = f(E/F)$.

For simplicity, let us assume $\alpha_j = \varpi_E^{j-1}$ for some prime element $\varpi_E$ of $E$ such that the element $\varpi_L = \varpi_E^e$ lies in $L$, where $e = e(E/F)$.
Then 
$$\m{\varpi_E} =\begin{pmatrix}
0&0&\cdots&0&\m{\varpi_L}\\
1_f&0&\cdots&0&0\\
0&1_f&\ddots&&\vdots\\
\vdots&\ddots&\ddots&\ddots&\vdots\\
0&\cdots&0&1_f&0
\end{pmatrix}.$$  Then the filtration of $\g_{y,0}$ is given by
$$\g_{y,r} = \m{\varpi_E^{\lceil er\rceil}}\, \g_{y,0}.$$   Thus a filtration jump occurs at $r$ precisely when $r$ lies in $\frac{1}{e}\Z$.  Consequently, if $\f$ is the residue field of $F$ then 
$$\g_{y,0:0^+} =
\begin{pmatrix}
\M_{f} (\f)&&\\
&\ddots&\\
&&\M_{f}(\f)
\end{pmatrix}$$ and $\g_{y,r:r^+} = \m{\varpi_E^{\lceil er\rceil}}\, \g_{y,0:0^+}$ if $r\in \frac{1}{e}\Z$.  If $r\not\in \frac{1}{e}\Z$ then $\g_{y,r:r^+} =0$.
In particular, with a slight abuse of notation, we have
$$\g_{y,e^{-1}:(e^{-1})^+} = 
\begin{pmatrix}
&&&\kern-1.5em \gP_F\M_{f} (\f)\\
\M_f(\f)&&\\
&\ddots&&\\
&&\M_{f}(\f)&
\end{pmatrix},$$
$$\g_{y,2e^{-1}:(2e^{-1})^+} = 
\begin{pmatrix}
&&&\kern-1.5em \gP_F\M_{f} (\f)&\\
&&&&\gP_F\M_f(\f)\\
\M_f(\f)&&&\\
&\ddots&&&\\
&&\M_{f}(\f)&&
\end{pmatrix},$$
and so forth.

Note that with the above conventions, the family of objects $\g_{y,r}$ only depends on the triple $(F, e ,f)$.

\subsection{Embeddings of degree $n$ extensions of $F$ stable under an outer automorphism}

Let $F$ be a field and let $x\mapsto \bar x$ be an automorphism of $F$ whose square is trivial.  Let $F'$ be the fixed field of $x\mapsto \bar x$.   There are two cases:
\begin{enumerate}
\item $F = F'$ and $\bar x = x$ for all $x\in F$.
\item $F/F'$ is quadratic and ${\rm Gal} (F/F') = \{ 1,x\mapsto \bar x\}$.
\end{enumerate}

If $X$ is a matrix with entries in $F$, we denote by $\overline{X}$ the matrix whose $ij$-th entry is $\overline{X}_{ij}$.  We also let $X^* = {}^t\overline{X}$.  If $X$ is a square matrix and $X^* = X$ then we say $X$ is {\bf hermitian}.  As before, we let $G= \GL_n (F)$ and $\g = \M_n(F)$ and we fix a degree $n$ extension $E$ of $F$.

Fix a matrix $\nu\in G$ that is a scalar multiple of a hermitian matrix.  Define an anti-automorphism $\sigma^*_\nu$ of $\g$ of order two by
$$\sigma^*_\nu (X) = \nu^{-1}\ X^*\ \nu.$$  

\begin{definition}
A \textbf{$\nu$-embedding}  of $E$ in $\g$  is an $F'$-embedding whose image is stable under $\sigma^*_\nu$.
\end{definition}

The relevance of $\nu$-embeddings to the main problems addressed in this paper is as follows.  We are considering tame supercuspidal representations of a group $G = \GL_n (F)$.  These representations are associated to $G$-data that involve various subgroups $G^i = \GL_{n_i}(E_i)$ of $G$.  If $\theta$ is an orthogonal involution of $G$ that stabilizes $G^i$ then the restriction of $\theta$ to $G^i$ may not be an orthogonal involution but, rather, a unitary involution.  (See Proposition \ref{symrestrict}.)  So when we consider $G = \GL_n (F)$ in this section, we really have a twisted Levi subgroup $G^i$ in mind.

Fix a $\nu$-embedding of $E$ in $\g$ and identify $E$ with its image in $\g$.  For every intermediate field $L$ of $E/F$ the restriction of $\sigma_\nu^*$ to $L$ is an $F'$-automorphism whose square is trivial.  We let $L^\nu$ denote the fixed field of $\sigma_\nu^*|L$.  
If $F\ne F'$ then $\sigma^*_\nu |L$ must have order two and $L/L^\nu$ must be quadratic.  But if $F= F'$, we either have $L = L^\nu$ or $L/L^\nu$ is quadratic.

In the present setting, the notion of a $\nu$-symmetric embedding of $E$ in $\g$ is identical to the notion of a $\nu$-embedding of $E$ in $\g$ such that $E = E^\nu$.   Since we have dealt with such embeddings in the previous section, we will now assume that our given embedding of $E$ in $\g$ is a ``strict $\nu$-embedding'' in the following sense:

\begin{definition}
A \textbf{strict $\nu$-embedding}  of $E$ in $\g$  is a $\nu$-embedding such that $E/E^\nu$ is quadratic.
\end{definition}

It turns out that $J$-symmetric embeddings are just as useful in studying strict $\nu$-embeddings as they are in studying $\nu$-symmetric embeddings.
Note that ${\rm Gal}(E/E^\nu) = \{ 1,\sigma^*_\nu |E\}$.

We first treat the case in which $F = F'$:

\begin{proposition}\label{symembedB}
Let $L$ be an intermediate field of $E/F$ such that $E/L$ is quadratic and let $r = [L:F] = n/2$.
Choose $\tau\in L$ so that $E= L[\sqrt{\tau}]$ and embed $E$ in $\M_2 (L)$ via 
$$x+y\, \sqrt{\tau} = \begin{pmatrix}x&y\tau\\ y&x\end{pmatrix}$$ and 
embed $L$ in $\M_r(F)$ via a $J_r$-symmetric embedding.   Then the symmetric matrices $\nu$ in $G = \GL_n(F)$ such that $\sigma_\nu |E$ is the nontrivial Galois automorphism of $E/L$  are precisely the matrices of the form
$$\nu =   \begin{pmatrix}1_{r}&0\\ 0&-1_{r}\end{pmatrix}J_n\ x ,$$ where $x \in E^\times$ and ${\rm tr}_{E/L}(x)=0$.  The determinant of $\nu$ is then given by
$\det\nolimits_G(\nu) = N_{E/F}(x) = \det (J_n) N_{L/F}(x)^2$.
\end{proposition}

\begin{proof}
Suppose $x$ lies in $E$.  Then $x = \begin{pmatrix}a&b\tau\\ b&a\end{pmatrix}$ for some $a,b\in L$.  The Galois conjugate of $x$ is identical to $\mu x \mu$, where $\mu = 1_r\oplus (-1_r)$.  We need to compute the set of symmetric matrices $\nu$ in $G$ such that
$$x = \sigma_\nu (\mu x\mu) = \nu^{-1}\  \mu\  {}^t x\ \mu\ \nu = (J_n \mu\nu)^{-1} x (J_n \mu \nu)$$ for all $x\in E$.  The latter condition simply says that  $J_n \mu\nu$ centralizes $E$ or, equivalently, $J_n\mu \nu\in E^\times$.   Thus the symmetric matrices in question are those symmetric matrices that lie in  $\mu J_n E^\times$.

Now suppose $x\in E^\times$ and let $\nu = \mu \, J_n\,  x$.  Then $\nu$ is symmetric exactly when $\mu J_n\, x = \nu = {}^t\nu = {}^tx\, J_n\, \mu = J_n (J_n \ {}^t x\, J_n)\mu = J_n x\mu$.  But, since $\mu J_n = -J_n\, \mu$, we deduce that $\nu$ is symmetric precisely when  $-\mu x= x\mu$ or, equivalently, $-x= \mu x\mu$.  The latter condition is the same as ${\rm tr}_{E/L} (x)=0$.

The determinant identity reduces to showing that $\det (\mu J_n)=1$, which is easily verified.
\end{proof}

The analogous result when $F/F'$ is quadratic is:

\begin{proposition}\label{symembedC}
Assume $F/F'$ is quadratic and choose $\tau\in F'$ so that $F= F' [\sqrt{\tau}]$.   Assume $L$ is a degree $n$ extension of $F'$ that is embedded in  
 $\M_n(F')$ via a $J_n$-symmetric embedding.  Assume $L\cap F = F'$ and let $E$ be the degree $n$ extension of $F$ given by $E= LF = L[\sqrt{\tau}]$.
Embed $E$ in $\g = \M_n(F)$ 
via 
$$x+y\, \sqrt{\tau} = \begin{pmatrix}x&y\tau\\ y&x\end{pmatrix},$$ 
where $x,y\in L \subset \M_n (F')$.

Then the set $J_n L^\times F^\times$  is identical to the set of matrices in $J_nE^\times$ that are scalar multiples of hermitian matrices.  It is also identical to the set of $\nu\in G=\GL_n (F)$ that are scalar multiples of hermitian matrices and satisfy $\sigma^*_\nu|E\in {\rm Gal} (E/L)$.  If $x\in E^\times$ and $\nu = J_n \, x$ then
$$\det\nolimits_G(\nu) = (-1)^{n(n-1)/2} N_{E/F}(x).$$
\end{proposition}

\begin{proof}
Suppose $\nu$ is a scalar multiple of a hermitian matrix and $\sigma^*_\nu |E\in {\rm Gal}(E/L)$.  If $x\in L$ then $$x= \sigma^*_\nu (x) = \nu^{-1}\, {}^t x\, \nu = \nu^{-1}\, J_n\, x\, J_n\, \nu = (J_n\, \nu)^{-1}\, x\, (J_n\, \nu).$$  So $J_n \, \nu$ centralizes $L$.  Since $J_n\, \nu$ also commutes with $\sqrt{\tau}$, it follows that $J_n \, \nu$ centralizes $E$.  Hence $J_n\, \nu \in E^\times$ or, in other words, $\nu \in J_n\, E^\times$.

Now suppose $x\in E^\times$ and let $\nu = J_n\, x$.  
Then $\nu$ is a scalar multiple of a hermitian matrix if and only if there exists $z\in F^\times$ such that $z\nu$ is hermitian.  But $z\nu = zJx$ and $(z\nu)^* = \bar z\nu^* = \bar z x^*J = \bar z J\bar x$.
Hence, $\nu$ is a scalar multiple of a hermitian matrix exactly when there exists $z\in F^\times$ such that $zx = \bar z\bar x$ or, in other words, $zx\in L$.  So $\nu$ is a scalar multiple of a hermitian matrix precisely when $x\in L^\times F^\times$.

Assume $x\in L^\times F^\times$ and let $\nu = J_n x$.  If $u\in L$ then $\sigma_\nu^* (u) = \nu^{-1}\cdot {}^t u\cdot \nu = \nu^{-1}J_nuJ_n\nu = x^{-1}ux= u$.  In addition, $\sigma_\nu^* (\sqrt{\tau}) = -\sqrt{\tau}$.  Hence, $\sigma_\nu^*$ is the nontrivial element of ${\rm Gal}(E/L)$.
\end{proof}

\subsection{Restriction and extension of involutions}

Our definition of ``orthogonal involution of $G$'' may be rephrased as follows:  an $F$-automorphism of  $\bG = \GL_n$ is  an orthogonal involution of $G =\GL_n (F)$ if it has the form $\theta_\nu (g) = \sigma_\nu (g^{-1})$ for some symmetric matrix $\nu\in G$.  
Similarly, if $F/F'$ is quadratic, we say that an $F'$-automorphism of  $\bG$ is  a ``unitary involution of $G$'' if it has the form $\theta_\nu (g) = \sigma^*_\nu (g^{-1})$ for some  $\nu\in G$ that is a scalar multiple of a hermitian matrix.  

Suppose $L$ is an extension of $F$ of degree $r<n$ and suppose $G'$ is a suitable subgroup of $G$ isomorphic to $\GL_r (L)$.  We now consider the problem of restricting orthogonal and unitary involutions of $G$ to orthogonal and unitary involutions of $G'$.  We start with restrictions of orthogonal involutions:

\begin{proposition}\label{symrestrict}
Let $L$ be an intermediate field of  $E/F$ with $r = [L:F]$ and $s = n/r = [E:L]$.
Suppose $L$ is embedded in $\M_r(F)$ via a $J_r$-symmetric embedding and suppose $E$ is embedded in $\M_s(L)$ via a $J_s$-symmetric embedding. 
 Let $G= \GL_n(F)$ and $G'=\GL_s(L)$.
\begin{enumerate}
\item
Then for $\nu\in G$ the following conditions are equivalent:
\begin{enumerate}
\item 
${}^t\nu = \nu$ and 
$\sigma_\nu (x)= x$ 
for all $x\in L$.
\item The matrix ${\nu'} =J^s_r \nu$ lies in $G'$ and, viewed as an element of $G'$, it is symmetric.
\end{enumerate}
If (a) and (b) hold then $\sigma_\nu | \M_s(L)=\sigma_{\nu'}$ and $$\det\nolimits_G (\nu) = (-1)^{n(r-1)/2}N_{L/F}(\det\nolimits_{G'}({\nu'})).$$
\item  Assume $L$ is a quadratic extension $L'[\sqrt{\tau}]$ of an intermediate field $L'$ of $L/F$.  Embed $L$ in $\M_2(L')$ using the $J_2$-symmetric embedding
$$x+y\sqrt{\tau}  = \begin{pmatrix}x&y\tau\\ y&x\end{pmatrix},$$ for all $x,y\in L'\subset \M_{r/2}(F)$, and embed $L'$ in $\M_{r/2}(F)$ using a $J_{r/2}$-symmetric embedding.
Let $\beta =  \begin{pmatrix}-\tau J_{r/2}&0\\ 0&J_{r/2}\end{pmatrix} $
and $\gamma = \beta \oplus \cdots \oplus
\beta \in G$.
Let $z\mapsto \bar z$ be the nontrivial Galois automorphism of $L/L'$.   Then for $\nu\in G$ the following conditions are equivalent:
\begin{enumerate}
\item 
${}^t\nu = \nu$ and 
$\sigma_\nu (x)= \bar x$ 
for all $x\in L$.
\item The matrix ${\nu'} =\gamma\, \nu$ lies in $G'$ and, viewed as an element of $G'$, it is hermitian with respect to $L/L'$.
\end{enumerate}
If (a) and (b) hold then $\sigma_\nu | \M_s(L)=\sigma^*_{\nu'}$ and  $$\det\nolimits_G (\nu) = (-1)^{n/2}N_{L'/F}(\tau)^s N_{L/F}(\det\nolimits_{G'}({\nu'})).$$
\end{enumerate}
\end{proposition}

\begin{proof}
Our proof builds on the proof of Lemma \ref{composite} and we use some of the same identities.
As in the proof of Lemma \ref{composite}, if $x$ is a matrix with entries in $F$, we let ${}^t x$ denote the transpose of $x$.  If $x$ is a matrix with entries in $L$, we let ${}^T x$ denote the transpose.  
Since $L$ is embedded in $\M_r(F)$, a matrix with entries in $L$ will have both types of transpose and if  $x\in \M_s(L)$ then these transposes are related by $${}^t x = J_r^s \ {}^T x \ J_r^s.$$

Assume $\nu\in G$ and write $\nu$ as an $s$-by-$s$ block matrix in which the $ij$-th block $\nu_{ij}$ lies in $\M_r(F)$.  
Condition 1(a) is equivalent to saying that whenever $1\le i,j\le s$, we have ${}^t \nu_{ji} = \nu_{ij}$ and ${}^t x\  \nu_{ij} = \nu_{ij}\ x$,   for all $x\in L$.  But for $x\in L$, we have ${}^t x = J_r \ x\ J_r$.  It is now easy to see that the following conditions are equivalent:
\begin{itemize}
\item ${}^t x\  \nu_{ij} = \nu_{ij}\ x$, for all $x\in L$,
\item $ J_r \ x\ J_r\  \nu_{ij} = \nu_{ij}\ x$, for all $x\in L$,
\item $J_r\ \nu_{ij}$ centralizes $L$,
\item $J_r\ \nu_{ij}\in L$.
\end{itemize}
So condition 1(a) is equivalent to saying ${}^t\nu_{ji} = \nu_{ij}$ for all $i,j$ and $J^s_r\nu\in G'$.  But if $J^s_r\nu\in G'$ then ${}^t\nu_{ji} = ({}^t \nu_{ji}\ J_r) J_r = {}^t (J_r \nu_{ji})J_r= J_r (J_r\nu_{ji}) = \nu_{ji}$.  It follows that condition 1(a) is equivalent to condition 1(b).

Now assume conditions 1(a) and 1(b) are satisfied and $x\in \M_s(L)$.
Then we have
$$
\sigma_\nu (x) = \nu^{-1}\ {}^tx\ \nu = {\nu'}^{-1}\ J^s_r\ {}^t x\ J^s_r\ {\nu'} \\
= {\nu'}^{-1}\  {}^Tx\ {\nu'} =\sigma_{\nu'}(x).$$
Therefore, $\sigma_\nu | \M_s(L) = \sigma_{\nu'} $.  The determinant identity follows from the transitivity of norms formula in \S7.4 of \cite{J}.  More specifically, we have \begin{eqnarray*}
N_{L/F}(\det\nolimits_{G'}({\nu'}))&=& \det\nolimits_G({\nu'}) = \det\nolimits_G(J^s_r)\, \det\nolimits_G(\nu)\\ &=&   ((-1)^{r(r-1)/2})^s \, \det\nolimits_G (\nu)\\
&=& (-1)^{n(r-1)/2}\, \det\nolimits_G (\nu).
\end{eqnarray*}

We next consider the proof of assertion 2.  Let
$$\varepsilon = \begin{pmatrix}1_{r/2}&0\\ 0&-1_{r/2}\end{pmatrix}$$ and observe that $\bar z = \varepsilon z\varepsilon$ for all $z\in L$.

Assume $\nu\in G$ and write $\nu$ as an $s$-by-$s$ block matrix $(\nu_{ij})$ with $\nu_{ij}\in \M_r(F)$.  
Condition 2(a) says ${}^t \nu_{ji} = \nu_{ij}$ and ${}^t x\  \nu_{ij} = \nu_{ij}\ \varepsilon\, x\,\varepsilon$,   for all $x\in L$ and all $i,j$.  The following conditions are easily seen to be equivalent:
\begin{itemize}
\item ${}^t x\  \nu_{ij} = \nu_{ij}\ \varepsilon\, x\,\varepsilon$,  for all $x\in L$,
\item $x\ J_r\  \nu_{ij}\ \varepsilon =J_r\  \nu_{ij}\ \varepsilon\ x$, for all $x\in L$,
\item $J_r\ \nu_{ij}\ \varepsilon$ centralizes $L$,
\item $J_r\ \nu_{ij}\ \varepsilon\in L$,
\item $\varepsilon\ J_r\ \nu_{ij}\in L$,
\item $\sqrt{\tau}\, \varepsilon J_r\ \nu_{ij}\in L$.
\end{itemize}
Since
$$\sqrt{\tau}\, \varepsilon\, J_r = \begin{pmatrix}0&\tau\\ 1&0\end{pmatrix}\begin{pmatrix}1_{r/2}&0\\ 0&-1_{r/2}\end{pmatrix}\begin{pmatrix}0&J_{r/2}\\  J_{r/2}&0\end{pmatrix} = \beta, $$
condition 2(a) is equivalent to saying $\nu$ is symmetric and ${\nu'} = \gamma\nu\in G'$.  But if $\nu$ is symmetric and ${\nu'} \in G'$ then 
\begin{eqnarray*}
\overline{{\nu'}_{ij}} &=&\varepsilon\, \nu'_{ij}\, \varepsilon
=\varepsilon\, \beta\, \nu_{ij}\, \varepsilon
=\varepsilon\, \beta\, {}^t\nu_{ji}\, \varepsilon
=\varepsilon\, \beta\, ({}^t\nu_{ji}\, {}^t\beta)\, {}^t\beta^{-1}\, \varepsilon\\
&=&\varepsilon\, \beta\, {}^t\nu'_{ji}\, {}^t\beta^{-1}\, \varepsilon
=\varepsilon\, \beta\, J_r\, \nu'_{ji}\, J_r\, {}^t\beta^{-1}\, \varepsilon
=-\sqrt{\tau}\, \nu'_{ji}\, J_r\, {}^t\beta^{-1}\, \varepsilon\\
&=&-\nu'_{ji}\, \sqrt{\tau}\, J_r\, {}^t\beta^{-1}\, \varepsilon
=-\nu'_{ji}\, \sqrt{\tau}\, J_r\, {}^t\sqrt{\tau}^{-1}\, \varepsilon\, J_r\, \varepsilon\\
&=&\nu'_{ji}\, \sqrt{\tau}\, J_r\, {}^t\sqrt{\tau}^{-1}\, J_r
=\nu'_{ji}\, \sqrt{\tau}\,  \sqrt{\tau}^{-1} 
=\nu'_{ji}
\end{eqnarray*}
 for all $i,j$.
 It follows that condition 2(a) is equivalent to condition 2(b).

Now suppose $x\in \M_s(L)$.  Using block matrices in $\GL_s(\M_r(F))=G$, define matrices  $\xi = \varepsilon \oplus \cdots \oplus \varepsilon$  and $\zeta = \sqrt\tau\oplus\cdots \oplus \sqrt{\tau}$ in $G$.  
  Then
\begin{eqnarray*}
\sigma_\nu (x) &=& \nu^{-1}\, {}^tx\, \nu = \nu^{-1}\, J^s_r\, {}^Tx\, J^s_r\, \nu = \nu^{-1}\, J^s_r\, \xi \, {}^T\bar x\, \xi\, J^s_r\, \nu\\
&=& \nu^{-1}\, J^s_r\, \xi \,  \zeta^{-1}\, {}^T\bar x\, \zeta\, \xi\, J^s_r\, \nu = {\nu'}^{-1}\, {}^T\bar x\, {\nu'} = \sigma^*_{\nu'} (x).
\end{eqnarray*}
The verification of the determinant formula is similar to the previous determinant formula, except that the determinant of $J_r^s$ is replaced by  
\begin{eqnarray*}
\det\nolimits_G (\gamma)
&=& \det\nolimits_{\GL_r(F)} (\beta)^s \\
&=& \det\nolimits_{\GL_{r/2}(F)} (J_{r/2})^{2s} 
\det\nolimits_{\GL_{r/2}(F)}(-\tau)^s\\
&=&(-1)^{(r/2)((r/2)-1)s} N_{L'/F}(-\tau )^s\\
&=&(-1)^{(n/2)((r/2)-1)} (-1)^{n/2}N_{L'/F}(\tau )^s\\
&=&(-1)^{(nr)/4} N_{L'/F}(\tau )^s\\
&=&(-1)^{s(r/2)^2} N_{L'/F}(\tau )^s\\
&=&(-1)^{sr/2} N_{L'/F}(\tau )^s\\
&=&(-1)^{n/2} N_{L'/F}(\tau )^s.
\end{eqnarray*}
\end{proof}

We now consider restrictions of unitary involutions:

\begin{proposition}
Assume $F/F'$ is quadratic and choose $\tau\in F'$ so that $F= F' [\sqrt{\tau}]$.   
Assume $L'$ is a degree $r$ extension of $F'$ that is embedded in  
 $\M_r(F')$ via a $J_r$-symmetric embedding.  Assume $E'$ is a degree $s$ extension of $L'$ that is embedded in  
 $\M_s (L')$ via a $J_s$-symmetric embedding.  Let $n=rs$.
 
 Assume $E'\cap F = F'$ and let $E$ be the degree $n$ extension of $F$ given by $E= E' F = E' [\sqrt{\tau}]$.
Embed $E$ in $\g = \M_n(F)$ 
via 
$$x+y\, \sqrt{\tau} = \begin{pmatrix}x&y\tau\\ y&x\end{pmatrix},$$ 
where $x,y\in E' \subset \M_n (F')$.
Let $L = L'[\sqrt{\tau}]$.
 
  Let $G= \GL_n(F)$ and $G'=\GL_s(L)$.
Then for $\nu\in G$ the following conditions are equivalent:
\begin{enumerate}
\item[a.]
$\nu$ is a hermitian matrix and 
$\sigma^*_\nu |L\in {\rm Gal}(L/L')$.
\item[b.] The matrix ${\nu'} =J^s_r \nu$ lies in $G'$ and, viewed as an element of $G'$, it is hermitian.
\end{enumerate}
If (a) and (b) hold then $\sigma^*_\nu | \M_s(L)=\sigma^*_{\nu'}$ and $$\det\nolimits_G (\nu) = (-1)^{n(r-1)/2}N_{L/F}(\det\nolimits_{G'}({\nu'})).$$
\end{proposition}

\begin{proof}
Assume $\nu\in G$ and write $\nu$ as an $s$-by-$s$ block matrix in which the $ij$-th block $\nu_{ij}$ lies in $\M_r(F)$.  
Then $\nu$ is hermitian precisely when $\nu_{ji}^* = \nu_{ij}$ whenever $1\le i,j\le s$.  
On the other hand, if $\nu$ is hermitian then $\sigma_\nu^* |L'$ is trivial precisely when $J_r\nu_{ij}$ centralizes $L'$ for all $i,j$.  But since $J_r\nu_{ij}$ automatically commutes with $\sqrt{\tau}$, to say that $J_r\nu_{ij}$ centralizes $L'$ is the same as saying it centralizes $L$ which, in turn, is the same as saying $J_r\nu_{ij}\in L$.

So condition (a) is equivalent to saying $\nu_{ji}^* = \nu_{ij}$ for all $i,j$ and $J^s_r\nu\in G'$.  But if $J^s_r\nu\in G'$ then $\nu_{ji}^* = (\nu_{ji}^*\ J_r) J_r =  (J_r \nu_{ji})^* J_r= J_r (J_r\overline{\nu_{ji}}) = \overline{\nu_{ji}}$.  It follows that condition (a) is equivalent to condition (b).

Now assume conditions (a) and (b) are satisfied and $x\in \M_s(L)$.
Then we have
$$
\sigma^*_\nu (x) = \nu^{-1}\ x^*\ \nu = {\nu'}^{-1}\ J^s_r\ x^*\ J^s_r\ {\nu'} \\
= {\nu'}^{-1}\  x^\bullet\ {\nu'} =\sigma_{\nu'}(x),$$ where $x^*$ is the conjugate transpose of $x$ in $\g$ and $x^\bullet$ is the conjugate transpose of $x$ in $\M_s (L)$.
Therefore, $\sigma_\nu | \M_s(L) = \sigma_{\nu'} $.  The determinant identity is proved in a similar fashion to the proof to the corresponding identity in Proposition \ref{symrestrict}.
\end{proof}

%
%

\section{Quadratic spaces and orthogonal involutions}\label{sec:quadraticspaces}

In this section, we recall well known facts about quadratic forms, quadratic spaces, orthogonal involutions, and orthogonal groups over a field $F$ that is either:
\begin{itemize}
\item a finite field of odd characteristic,
\item a finite extension of $\Q_p$ with $p\ne 2$.
\end{itemize}
Assume we have fixed such a field $F$.  For convenience, we give a unified treatment of the latter two cases.

\begin{definition}
The \textbf{Hilbert symbol} is the pairing $F^\times\times F^\times \to \{ \pm 1\}$  defined
by
 $${\rm Hilbert}(a,b)=
    \begin{cases}
    1,&\mbox{if }z^2=ax^2+by^2\mbox{ has a solution }(x,y,z)\in F^3-\{0\};\\
    -1,&\mbox{otherwise.}\end{cases}$$
\end{definition}
When $F$ is finite, ${\rm Hilbert}(a,b)=1$ for all $a,b\in F^\times$ since every  quadratic form in three variables is isotropic.

Let $G = \GL_n (F)$, where $n>1$, and let $\cS$ be the set of symmetric matrices in $G$.  Then $G$ acts on the right of $\cS$ by $\nu\cdot g = {}^t g\, \nu\, g$.  Two elements of $\cS$ are said to be ${\bf similar}$ when they are in the same $G$-orbit.  It is well known that the group $A$ of diagonal matrices in $\cS$ contains a set of representatives for the similarity classes in $\cS$.

\begin{definition}\label{discHassedef}
The \textbf{discriminant} ${\rm disc}(\nu)$ of $\nu\in \cS$ is the image of $\det\nu$ in $F^\times/(F^\times)^2$.
The \textbf{Hasse invariant (a.k.a., Hasse-Witt invariant)}  ${\rm Hasse}(\nu)$ of  $\nu \in \cS$ is given by 
$${\rm Hasse}(\nu) = \prod_{i\le j} {\rm Hilbert}(a_i,a_j),$$ where ${\rm diag}(a_1,\dots , a_n)$ is a diagonal matrix similar to $\nu$.
\end{definition}

The latter definitions are motivated by the following standard result:

\begin{lemma}\label{simclassification}
Two elements $\nu_1,\nu_2\in\cS$ are similar precisely when ${\rm disc}(\nu_1) = {\rm disc}(\nu_2)$ and ${\rm Hasse}(\nu_1) = {\rm Hasse}(\nu_2)$.
\end{lemma}

Unfortunately, the literature is not consistent in its use of the terminology ``discriminant'' or ``Hasse invariant.''  
The ``discriminant'' of $\nu\in \cS$ is sometimes defined as 
$${\rm disc}_0 (\nu) = (-1)^{n(n-1)/2}\disc (\nu) = \det (J_n) \disc (\nu).$$
The ``Hasse invariant'' is sometimes defined as above except that the product is over $i<j$ instead of $i\le j$, that is,
$${\rm Hasse}_0(\eta) = \prod_{i< j} {\rm Hilbert}(a_i,a_j).$$ 
We have $$\Hasse (\eta) 
 =  \Hasse_0 (\eta)\prod_{i=1}^n (a_i,a_i) 
 = \Hasse_0 (\eta)\cdot  {\rm Hilbert} (-1,\disc \eta) .$$
Note that Lemma \ref{simclassification} would still be true regardless of which variants of the definitions one chooses.
More precisely, if $\nu_1,\nu_2\in \cS$ then the statement ``$\nu_1$ and $\nu_2$ have the same discriminant'' is independent of which of the two definitions of ``discriminant'' one chooses.  Moreover, if $\nu_1$ and $\nu_2$ have the same discriminant then the statement 
``$\nu_1$ and $\nu_2$ have the same Hasse invariant'' is independent of which of the two definitions of ``Hasse invariant'' one chooses.  (Therefore, to avoid confusion, we try to make such comparative statements rather than referring directly to the discriminant or Hasse invariant of an element of $\cS$.)

Of the many identities regarding Hilbert symbols and Hasse invariants, we emphasize the following fact:  
\begin{quote}
If $F$ is is finite extension of $\Q_p$ with $p$ odd and $\nu \in \cS\cap \GL_n (\gO_F)$ then ${\rm Hasse}(\nu) = {\rm Hasse}_0 (\nu)=1$.
\end{quote}
(See Lemma 5.9 \cite{HL1}.)

Suppose $F$ is $p$-adic.  Then $F^\times/(F^\times)^2$ has order four (since $p\ne 2$) and thus there are four possibilities for the discriminant of an element of $\cS$.  On the other hand, there are two possibilities for the Hasse invariant.  Therefore, if $\nu\in \cS$ there are eight possibilities for $({\rm disc}(\nu),{\rm Hasse}(\nu))$.  Hence, there are at most eight $G$-orbits (or, in other words, similiarity classes) in $\cS$.  Using diagonal matrices, it is easy to see that there are exactly eight $G$-orbits in $\cS$ when $n>2$, but there are only seven orbits when $n=2$ since in this case if $\disc (\nu) = \disc (J_n)$ then $\nu$ is similar to $J_n$.

When $F$ is finite, $F^\times /(F^\times)^2$ has two elements and  the Hasse invariant is always trivial.  It turns out that there are always two $G$-orbits in $\cS$ and they are characterized by the discriminant.

\begin{definition}
An \textbf{orthogonal involution of $G$} is an $F$-automorphism of $\GL_n$ defined by 
$$\theta_\nu (g) = \nu^{-1}\cdot {}^tg^{-1} \cdot \nu,$$
for some $\nu\in G$.
\end{definition}

Given $\nu_1, \nu_2\in \cS$, it is elementary to see that $\theta_{\nu_1} = \theta_{\nu_2}$ precisely when $\nu_1 Z = \nu_2 Z$, where $Z$ is the center of $G$.
So $\nu Z\mapsto \theta_\nu$ determines a bijection between the set $\cS/Z$ and the set of orthogonal involutions of $G$.
 
The group $G$ acts on the set of its orthogonal involutions by:
$$g\cdot \theta = {\rm Int}(g)\circ \theta\circ {\rm Int}(g^{-1}).$$
We have
$$g\cdot \theta_\nu = \theta_{{}^t g^{-1} \nu g^{-1}}.$$
The (right) action of $G$ on $\cS$ yields a right action of $G$ on $\cS/Z$.

\begin{lemma} The map that sends the $G$-orbit of $\theta_\nu$ to the $G$-orbit of $\nu Z$ defines a bijection between the set of $G$-orbits of orthogonal involutions and the set of $G$-orbits in $\cS/Z$.
\end{lemma}

Suppose  $F$ is finite.  If $n$ is odd then the two $G$-orbits in $\cS$ merge into a single $G$-orbit in $\cS/Z$, but when $n$ is even they remain separate when $n$ is even.  We obtain:

\begin{lemma}\label{finiteGorbs}
Suppose $F$ is finite.
If $n$ is odd there is a single $G$-orbit of orthogonal involutions.  If $n$ is even there are two $G$-orbits.
\end{lemma}

When $F$ is $p$-adic, we have:

\begin{proposition}\label{orthinvGorbs} Suppose $F$ is $p$-adic.  If $n>1$ is odd then there are precisely two $G$-orbits of orthogonal involutions of $G$.  If two elements of $\cS$ have the same discriminant but different Hasse invariants then their cosets in $\cS/Z$ lie in different $G$-orbits.

If $n\ge 2$ is even then the discriminant determines a map from $\cS/Z$ onto $F^\times /(F^\times)^2$.
If $n=2$ then there are four $G$-orbits in $\cS/Z$ and they are precisely the fibers of the discriminant map $\cS/Z \to F^\times/(F^\times)^2$.
If $n>2$ and even then there are five $G$-orbits in $\cS/Z$.  Three of these orbits are fibers of the discriminant map $\cS/Z \to F^\times/(F^\times)^2$.
But the fiber of the discriminant of the matrix $$J=
\begin{pmatrix}
& &1\\
&\raisebox{-.1ex}{.} \cdot \raisebox{1.2ex}{.}&\\
1&&
\end{pmatrix}$$ breaks up into two orbits and these orbits are distinguished from each other by the Hasse invariant.
\end{proposition}

Let $\bG$ be the $F$-group $\GL_n$.
If $\theta$ is an orthogonal involution of $G$ then we let $\bG^\theta$ denote the $F$-group consisting of the the fixed points of $\theta$ in $\bG$.  We let $G^\theta = \bG^\theta (F)$ denote the group of fixed points of $\theta$ in $G$.

\begin{lemma}
Suppose $\theta_1$ and $\theta_2$ are orthogonal involutions of $G = \GL_n(F)$ such that $G^{\theta_1} = G^{\theta_2}$.  Then $\theta_1 = \theta_2$.  Consequently, $\theta\mapsto G^\theta$ leads to a bijection between the set of $G$-orbits of orthogonal involutions of $G$ and the set of $G$-conjugacy classes of orthogonal groups in $n$ variables in $G$.
\end{lemma}

\begin{proof}
We begin by recalling Proposition 1.2 \cite{HW}:
\begin{quote}
Let $\bG$ be a connected semi-simple algebraic group and $\theta_1$ and $\theta_2$ be involutions of $\bG$.  If the identity components of the fixed point groups $\bG^{\theta_1}$ and $\bG^{\theta_2}$ are identical, then $\theta_1 =\theta_2$. 
\end{quote}
Now suppose, under the hypotheses of our lemma, $G^{\theta_1} = G^{\theta_2}$.  Then $H^{\theta_1} = H^{\theta_2}$, where $H = {\rm SL}_n (F)$.  This implies $\bH^{\theta_1} = \bH^{\theta_2}$, where $\bH = {\rm SL}_n$.  This implies $\theta_1 | {\rm SL}_n = \theta_2 |{\rm SL}_n$, which implies $\theta_1 = \theta_2$.
\end{proof}

\section{Orthogonal similitudes}

\subsection{Generalities}\label{sec:generalities}

In this section, $G$ will denote a general linear group $\GL_n (F)$, with $n>1$, where $F$ is a field whose characteristic is not 2.  As in  \S\ref{sec:quadraticspaces}, we really only need to study  the case in which $F$ is either a finite field or a finite extension of some $p$-adic field $\Q_p$, but our initial results hold for arbitrary $F$.

We are interested in special cases of a very general and basic  problem involving orthogonal similitude groups.  Fix a symmetric matrix $\nu\in G$ and define an orthogonal involution $\theta$ by
$\theta (g) = \nu^{-1}\, {}^t g^{-1}\, \nu$.  As usual, $G^\theta$ will denote the orthogonal group comprised of the fixed points of $\theta$ in $G$ and, in addition, $G_\theta$ will denote the orthogonal similitude group.  Thus $G_\theta$ consists of the elements $g\in G$ such that  $g\theta (g)^{-1}$ lies in the center $Z$ of $G$.
The similitude ratio defines a homomorphism $\mu : G_\theta \to Z : g\mapsto g\theta (g)^{-1}$ and yields exact sequences
$$1\to G^\theta\to G_\theta \to \mu (G_\theta)\to 1$$  and
$$1\to G^\theta /Z^\theta \to G_\theta /Z\to \mu (G_\theta)/\mu (Z)\to 1$$
and isomorphisms
\begin{eqnarray*}
G_\theta/G^\theta &\cong& \mu (G_\theta),\\
G_\theta /(ZG^\theta) &\cong& \mu (G_\theta)/\mu (Z).
\end{eqnarray*}
We are interested in special cases of the following:

\begin{problem}
Given a subgroup $H$ of $G$, compute the abelian group $$G_\theta /((H\cap G_\theta)ZG^\theta).$$
\end{problem}

We actually only need to know the order of the latter groups or, in other words, the constants
$$m_H (\theta ) := [G_\theta : (H\cap G_\theta)ZG^\theta].$$
Since the similitude ratio gives an isomorphism
$$G_\theta /((H\cap G_\theta)ZG^\theta)\cong \mu (G_\theta)/ \mu ((H\cap G_\theta)Z),$$
we have
$$m_H (\theta ) = [\mu (G_\theta) : \mu ((H\cap G_\theta)Z)].$$

There is one special class of examples of particular interest.  Suppose $E$ is a degree $n$ extension of $F$ that is embedded in $\M_n (F)$ via an $F$-linear embedding.  Let $T$ be the torus $E^\times$ in $G$ and assume that $T$ is $\theta$-split in the sense that $\theta (t) = t^{-1}$, for all $t\in T$.  
In this situation, we are interested in computing $m_T(\theta)$.
It will turn out that $m_T(\theta)$ depends on the number of quadratic extensions of $F$ that are contained in $E$.

\begin{definition}
The involution $\theta$ is a \textbf{distinguishing involution} if there exists a degree $n$ field extension $E$ of $F$ that is embedded in $\M_n (F)$ via an $F$-embedding such that the torus $T = E^\times$ is $\theta$-split.
\end{definition}

Roughly speaking, the involutions that are not distinguishing involutions do not contribute to the theory of distinguished representations of $G$ (at least in the setting we consider).  Thus we assume that our involution $\theta$ is a distinguishing involution  and we assume that 
 we have fixed $E$ and $T$ as just described.

If $F'$ is any extension of $F$, we will define $$y_{F'/F} = 1+\# \{ \text{quadratic extensions of $F$ contained in $F'$}\}.$$
In particular, taking $\overline{F}$ to be an algebraic closure of $F$, we have
$$y_{\overline{F}/F} = [F^\times :(F^\times)^2],$$ since a nontrivial square class $x(F^\times)^2$ corresponds to the quadratic extension $F[\sqrt{x}]$.  Similarly, for arbitrary $F'$ we have 
$$y_{F'/F} = [(F'^\times)^2\cap F^\times  :(F^\times)^2].$$

We now enumerate some elementary facts that apply to general $F$:

\begin{lemma}\label{mTbasics}
\begin{enumerate}
\item $\mu (Z) = Z^2 = (F^\times)^2$.
\item $\mu (T\cap G_\theta) = (E^\times)^2 \cap F^\times$.
\item $m_H (\theta)$ is a divisor of $y_{\overline{F}/F}$, for all $H$.
\item If $H_1\subset H_2$ then $m_{H_2}(\theta)$ is a divisor of $m_{H_1}(\theta)$.
\item $m_Z(\theta) = y_{E/F}\cdot m_T(\theta)$.
\end{enumerate}
\end{lemma}

The proofs of the latter facts are obvious.  The next result is standard, but we include a proof for completeness.

\begin{lemma}\label{mTodd}
 If  $n$ is odd then $\mu ((H\cap G_\theta)Z) =Z^2 = (F^\times)^2$
and $m_H (\theta ) =1$ for all $H$.
\end{lemma}

\begin{proof}  If $z\in Z$ then $\det z = z^n \equiv z$ (mod $Z^2$).  In particular, if $g \in G_\theta$, we may take $z = \mu (g)$.  Since $\det \mu(g)\in Z^2$, we deduce that $\mu (g)\in Z^2$.  Thus, for all $H$, we have
$$Z^2\subseteq  \mu ((H\cap G_\theta)Z)  \subseteq \mu (G_\theta) \subseteq Z^2$$
which implies 
 $\mu ((H\cap G_\theta)Z) = Z^2$ 
and $m_H(\theta ) =1$.
\end{proof}

\begin{lemma}\label{mTJ}
If  $n$ even and $\theta = \theta_J$ is associated to 
$$J=
\begin{pmatrix}
& &1\\
&\raisebox{-.1ex}{.} \cdot \raisebox{1.2ex}{.}&\\
1&&
\end{pmatrix}$$ then  $\mu (G_\theta) = Z = F^\times$.  Therefore, $m_Z(\theta )= y_{\overline{F}/F}$ and  $m_T(\theta ) = y_{\overline{F}/F}/y_{E/F}$.
\end{lemma}

\begin{proof}  Our claim follows from the fact that if $z\in F^\times$ and 
$$g_z = {\rm diag}(z,\dots, z,1,\dots ,1),$$ where the first $n/2$ diagonal entries equal $z$, then $\mu (g_z) = z$.
\end{proof}

\subsection{Finite fields}

When $F$ is a finite field of odd characteristic, we have the following consequence of Lemma \ref{mTbasics}:

\begin{lemma}\label{mTfinite}
Suppose $q$ is a power of an odd prime.
If $F= \F_q$ and $E= \F_{q^2}$ then $y_{\overline{F}/F} =2$ and
\begin{enumerate}
\item $m_Z(\theta)= y_{E/F}  = 
\begin{cases}
1,&\text{if $n$ is odd},\\
2,&\text{if $n$ is even.}
\end{cases}$
\item 
$m_T( \theta) = 1$.
\item 
$\mu (T\cap G_\theta) = (\F^\times_{q^n})^2\cap \F_q^\times = 
\begin{cases}
(\F_q^\times)^2,&\text{if $n$ is odd},\\
\F_q^\times,&\text{if $n$ is even}.
\end{cases}$
\end{enumerate}
\end{lemma}

\begin{proof}
We observe that
$y_{\overline{F}/F} =2$, since $(\F_q^\times)^2$ has index 2 in $\F_q^\times$.  
Statement (3) follows from this together with Lemma \ref{mTbasics} (2) and it implies
$$y_{E/F} = [(\F_{q^2}^\times)^2 \cap \F_q^\times : (\F_q^\times)^2] = \begin{cases}
1,&\text{if $n$ is odd},\\
2,&\text{if $n$ is even.}
\end{cases}$$ 
Statements (1) and (2) follow from Lemma \ref{mTodd} when $n$ is odd and they follow from
parts (3) and (5) of 
Lemma \ref{mTbasics} when $n$ is even.
\end{proof}

\subsection{$p$-adic fields}

\begin{proposition}\label{indices}
If $F$ is a finite extension of $\Q_p$ with $p\ne 2$ then $y_{\overline{F}/F}=4$ and:
\begin{enumerate}
\item If $y_{E/F}=1$ then $n$ is odd and $m_H(\theta) =1$ for all $H$.
\item If $y_{E/F}=2$ and $L$ is the unique quadratic extension of $F$ in $E$ then $\mu (T\cap G_\theta) = (E^\times)^2\cap F^\times = (L^\times)^2 \cap F^\times$ and
\begin{enumerate}
\item If $\theta$ is in the $G$-orbit of $\theta_J$ then $\mu (G_\theta) = Z= F^\times$ and $m_Z (\theta)=4$ and $m_T(\theta)=2$.
\item Let $\theta = \theta_\nu$ be an orthogonal involution of $G$ associated to a symmetric matrix $\nu\in G$ such that $(-1)^{n/2}\det\nu \in N_{L/F}(L^\times) - (F^\times)^2$.
Then $\mu (G_\theta) = \mu (T\cap G_\theta)$ and $m_Z(\theta)=2$ and $m_T(\theta)=1$.
\end{enumerate}
\item If $y_{E/F}=4$ then $m_Z(\theta)=4$ and $m_T(\theta)=1$.
\end{enumerate}
\end{proposition}

\begin{proof}
Since $p$ is odd, $[F^\times : (F^\times)^2]=4$ and hence
$y_{\overline{F}/F}=4$.  This implies that $y_{E/F}$ is a divisor of 4.  Lemma \ref{fieldsA} implies that $y_{E/F}=1$ precisely when $n$ is odd.  
Thus statement (1) follows from Lemma \ref{mTodd}.
Statement (3) follows from part (4) of Lemma \ref{mTbasics}.

Now assume $y_{E/F}=2$.  We have $\mu (T\cap G_\theta) = (E^\times)^2\cap F^\times$, according to part (2) of Lemma \ref{mTbasics}.  Suppose $\alpha \in (E^\times)^2 \cap F^\times$ and $\alpha \not\in (F^\times)^2$.  Choose $\beta\in E^\times$ such that $\alpha = \beta^2$.  The field $F[\beta]$ is a quadratic extension of $F$ contained in $E$ and hence it equals $L$.  This implies $\alpha \in (L^\times)^2 \cap F^\times$ and hence $(E^\times)^2 \cap F^\times = (L^\times)^2 \cap F^\times$.

Statement 2(a) is a special case of Lemma \ref{mTJ}.  Assume now we are in the setting of statement 2(b).  
Because $\nu$ and $J$ have different discriminants, the $G$-orbit of $\theta$ is determined by the discriminant of $\nu$.
(See Proposition \ref{orthinvGorbs}.)  
Since we are free to replace $\theta$ by another element of its $G$-orbit, we may assume that $\nu$ is a diagonal matrix of the form
$$\nu = \text{diag}(1,\dots ,1,\tau),$$ with $\tau\in F^\times$.

We wish to determine the image of the similitude homomorphism $\mu : G_\theta \to F^\times$.  This is equivalent to determining the set of all $z\in F^\times$ such that the equation
$$g\, \nu^{-1} \, {}^tg\, \nu  = z$$ is solvable for $g\in G$.
Let us rewrite the latter equation as 
$$g\, (z\nu)^{-1}\, {}^t g = \nu^{-1}.$$
Let $e_1,\dots , e_n$ be the rows of $g$.  View these rows as vectors in the quadratic space $V = F^n$ with inner product associated to the diagonal matrix $(z\nu)^{-1}$.  

Let us consider the Hasse invariant of $V$ (as defined in Definition \ref{discHassedef}).   To compute the Hasse invariant, one needs to diagonalize the quadratic form on $V$ or, in other words, one needs to choose an orthogonal basis.  If one uses the basis $e_1,\dots ,e_n$ just specified then one obtains $\Hasse (V) =1$.  On the other hand, if one uses the standard basis of $F^n$  one gets
\begin{eqnarray*}
\Hasse (V)&=&
(z,z)^{(n-1)(n-2)/2} (z,z\tau)^{n-1}\\
 &=& (z,z)^{(n-2)/2}(z,z\tau)\\
&=& (z,z)^{n/2}(z,\tau)\\
&=&(z,-1)^{n/2}(z,\tau)\\
&=&(z,(-1)^{n/2}\tau)\\
&=&(z,(-1)^{n/2}\det\nu ).
\end{eqnarray*}
Therefore 
  if $z\in \mu (G_\theta)$ then $(z,(-1)^{n/2}\det\nu)=1$ or, in other words, $z\in N_{K/F}(K^\times )$, where $K= F[\sqrt{\beta}]$ and $\beta= (-1)^{n/2}\det\nu$.  Note that since we assumed $\nu$ and $J$ have distinct discriminants, we know that $\beta$ cannot be a square in $F^\times$.
  
We observe that the sets $((K^\times)^2\cap F^\times) - (F^\times)^2$ and $N_{L/F}(L^\times)- (F^\times)^2$ each comprise a single square class in $F^\times / (F^\times)^2$.  Since these square classes both contain $\beta$, they must be identical. 
So we have $(K^\times)^2\cap F^\times= N_{L/F}(L^\times)$.
But the latter condition is equivalent to the condition ${\rm Hilbert} (\alpha, \beta)=1$, where $L = F[\sqrt{\alpha}]$ for some square root $\sqrt{\alpha}$ of some nonsquare $\alpha$ in $F^\times$.  By symmetry of the Hilbert symbol, we deduce  $(L^\times)^2\cap F^\times = N_{K/F}(K^\times)$.  Thus $z\in (L^\times)^2\cap F^\times = \mu (T\cap G_\theta)$.  Since we have shown $\mu (T\cap G_\theta) \subseteq \mu (G_\theta)\subseteq \mu (T\cap G_\theta)$, we deduce that $\mu(G_\theta) = \mu (T\cap G_\theta)$.  Consequently, $m_T(\theta)=1$ and, according to part (5) of Lemma \ref{mTbasics}, $m_Z(\theta)=2$.  This completes the proof of  statement 2(b).
\end{proof}

\begin{lemma}\label{indicestwo}
Assume $F$ is a finite extension of $\Q_p$ with $p\ne 2$.  Let $E$ be a tamely ramified degree $n$ extension of $F$ that contains a unique quadratic extension $L$ of $F$.  Let $E_0$ be an intermediate field of $E/F$ such that $E/E_0$ is unramified.   
Embed $E$ in $\M_{n_0}(E_0)$, with $n_0 = [E:E_0]$, via a $J$-symmetric embedding and embed $E_0$ in $\M_{m_0}(F)$, with $m_0= [E_0:F]$, via a $J$-symmetric embedding.
Let $H = E_0^\times \GL_{n_0}(\gO_{E_0})$ and assume $\theta$ is the orthogonal involution $\theta_J$ of $G= \GL_n(F)$.
Then $m_H (\theta)= m_T(\theta)= 2$.
\end{lemma}

\begin{proof}
Assume the hypotheses in the statement of the lemma.  In particular, we assume that $\theta = \theta_{J_n}$.  According to part 1(a) of Proposition \ref{symrestrict}, the restriction of $\theta$ to $G^0 = \GL_{n_0}(E_0)$ is the orthogonal involution $\theta_0 = \theta_{J_{n_0}}$ of $G^0$ associated to $J_{n_0}$.  
Let $\mu_0 :G^0_{\theta_0}\to E_0^\times$ be the similitude homomorphism associated to $\theta_0$.

Suppose $h\in \GL_{n_0}(\gO_{E_0})$, $z\in E_0^\times$ and $g = hz$.  Then $g\in G^0_{\theta_0}$ if and only if $h\in G^0_{\theta_0}$.  
If $h\in \GL_{n_0}(\gO_{E_0})\cap G^0_{\theta_0}$ then, taking determinants, we see that $\mu_0 (h)\in \gO_{E_0}^\times$.  

Arguing as in the proof of  Lemma \ref{mTodd}, we see that when $n_0$ is odd we have $\mu_0 (\GL_{n_0}(\gO_{E_0})\cap G^0_{\theta_0})= (\gO_{E_0}^\times)^2$.  
Arguing as in the proof of  Lemma \ref{mTJ}, we see that when $n_0$ is even we have $\mu_0 (\GL_{n_0}(\gO_{E_0})\cap G^0_{\theta_0})= \gO_{E_0}^\times$.  
It follows that
$$\mu_0 (H\cap G^0_{\theta_0})= 
\begin{cases}
(E_0^\times)^2,&\text{if $n_0$ is odd},\\
\gO_{E_0}^\times (E_0^\times)^2,&\text{if $n_0$ is even}.
\end{cases}$$
Thus $$\mu (H\cap G_\theta) = \mu_0(H\cap G^0_{\theta_0})\cap F^\times
=\begin{cases}
(E_0^\times)^2\cap F^\times,&\text{if $n_0$ is odd},\\
\gO_{E_0}^\times (E_0^\times)^2\cap F^\times,&\text{if $n_0$ is even}.
\end{cases}$$
On the other hand,
$m_H(\theta) = [\mu (G_\theta) : \mu (H\cap G_\theta)]$ and, according to Lemma \ref{mTJ},  we have 
$\mu (G_\theta)  = F^\times$.

If $n_0$ is odd then  $y_{E_0/F}$ must be 2, according to  Lemma \ref{fieldsA},  and thus $$m_H(\theta) = [F^\times: (E_0^\times)^2\cap F^\times ] =
\frac{4}{y_{E_0/F}} =2.$$

Now suppose $n_0$ is even.  Then $$m_H(\theta) = [F^\times: \gO_{E_0}^\times (E_0^\times)^2\cap F^\times ].$$
We observe that $E$ must contain an unramified quadratic extension of $F$ and thus $L/F$ is unramified.  
Since $L$ is the unique quadratic extension of $F$ in $E$, we have $(E^\times)^2 \cap F^\times = (L^\times)^2\cap F^\times$.  Since $E$ contains the unique unramified quadratic extension of $E_0$, it follows that $\gO_{E_0}^\times\subset (E^\times)^2$.  Since $L$ is an unramified quadratic extension of $F$, it must be the case that $(L^\times)^2\cap F^\times = \gO_F^\times (F^\times)^2$.
Therefore,
$$(L^\times)^2\cap F^\times=  \gO_F^\times (F^\times)^2 \subseteq \gO_{E_0}^\times (E_0^\times)^2\cap F^\times \subseteq(E^\times)^2 \cap F^\times = (L^\times)^2 \cap F^\times.$$
We now have
$$m_H(\theta) = [F^\times : (L^\times)^2\cap F^\times] =2.$$
Our assertion now follows from part 2(a) of Proposition \ref{indices}.
\end{proof}

\section{Split $T$-orbits of orthogonal involutions}

\subsection{Generalities}\label{sec:splitTgeneralities}

Let $E/F$ be a (finite) degree $n$ field extension.
Embed $E$ in $\M_n (F)$ via a $J$-symmetric embedding.  Let $G= \GL_n (F)$ and $T= E^\times$.  

\begin{definition}
A \textbf{split $T$-orbit} is a $T$-orbit of orthogonal involutions of $G$ such that $T$ is $\theta$-split.
\end{definition}

The main objective of this section is to study the split $T$-orbits in $G$, especially in the following two cases:
\begin{itemize}
\item  $F$ is a finite extension of $\Q_p$ for some odd prime $p$ and $E/F$ is tamely ramified.
\item $E$ and $F$ are finite fields of odd order.
\end{itemize}
In particular, we would like to determine the number of split $T$-orbits that lie within a given $G$-orbit of orthogonal involutions.

Let $$Y_{E/F} = E^\times /((E^\times)^2F^\times)$$ and let $y_{E/F}$ denote the cardinality of $Y_{E/F}$.  Let ${\cal O}^T$ denote the set of split $T$-orbits of orthogonal involutions of $G$. 

The following generalizes the statement of  Lemma 5.6 of \cite{HL1}, but the proof is identical:

\begin{lemma}\label{yEFquad}
If $[ F^\times :(F^\times)^2] = [E^\times : (E^\times)^2]< \infty$ then
$y_{E/F}-1$ is the number of quadratic extensions of $F$ contained in $E$.
\end{lemma}

We now recall the statement of
Corollary 5.5  {\bf [HL1]}:

\begin{lemma}\label{muEF}
If $E/F$ is separable then the map
$$\mu_{E/F} :Y_{E/F} \to {\cal O}^T$$
that sends the coset of $x\in E^\times$ to the orbit of the involution $\theta_\nu$ with $\nu_{ij} = {\rm tr}_{E/F}(x\, e_i\,e_j)$ is a bijection. 
\end{lemma}

We will not actually need to use Lemma \ref{muEF} in this paper, but, instead, we use a more convenient variant that uses the theory of $J$-symmetric embeddings:

\begin{lemma}\label{muprime}
The map
$$\mu'_{E/F} :Y_{E/F} \to {\cal O}^T$$
that sends  the coset of $x\in E^\times$ maps to the orbit of $\theta_{J\, x}$ is a bijection.  The involution $\theta_J$ lies in a split $T$-orbit.  
\end{lemma}

\begin{proof}
According to Proposition  \ref{symmetricembedding}, the set of symmetric matrices $\nu$ in $G$ such that $T$ is $\theta_\nu$-split is identical to $J_n T$.
If $x_1,x_2\in T$ then $\theta_{Jx_1} = \theta_{Jx_2}$, precisely when $Jx_2 x_1^{-1} J\in F^\times$ or, equivalently, $x_2x_1^{-1}\in F^\times$.  This shows that $xF^\times\mapsto \theta_{Jx}$ determines a bijection between $E^\times/F^\times$ and the set of orthogonal involutions $\theta$ of $G$ such that $T$ is $\theta$-split.

Now suppose $x,y\in T$.  Then $(y\cdot \theta_{Jx})(g) = y\, x^{-1}\, J\,  {}^ty\, {}^t g^{-1}\, {}^t y^{-1}\, J\, x\, y^{-1} = x^{-1}y^2\, J \, {}^tg^{-1}\, J\, xy^{-2} = \theta_{Jxy^{-2}}(g)$, for all $g\in G$.  Interpreting the identity $y\cdot \theta_{Jx} = \theta_{Jxy^{-2}}$ as a $T$-equivariance property of the aforementioned bijection, we deduce that $\mu'_{E/F}$ is a well-defined bijection.
\end{proof}

Lemma \ref{muprime} may appear to apply to any finite extension, however, we stress that we are assuming throughout this section that $E$ has a $J$-symmetric embedding in $\M_n(F)$.

\subsection{Finite fields}

Assume $E/F$ is an extension of finite fields.  From Lemma \ref{finiteGorbs}, Lemma \ref{yEFquad}, and Lemma \ref{muEF}  (or Lemma \ref{muprime}), we  see that 
\begin{eqnarray*}
\# \cO^T&=& y_{E/F}=\begin{cases}
1,&\text{if $n$ is odd},\\
2,&\text{if $n$ is even,}
\end{cases}\\
&=&\# \{ \text{$G$-orbits of orthogonal involutions}\} .
\end{eqnarray*}

\begin{lemma}\label{finiteuniqueness}
Each $G$-orbit of orthogonal involutions contains a unique split $T$-orbit.
\end{lemma}

\begin{proof}
When $n$ is odd, the unique split $T$-orbit is obviously contained in the unique $G$-orbit.

Now suppose $n$ is even.  If $\nu\in G$ is symmetric then the $G$-orbit of $\theta_\nu$ is characterized by the discriminant of $\nu$, Lemma \ref{finiteGorbs}.  On the other hand, according to Proposition  \ref{symmetricembedding}, $\theta_\nu$ lies in a split $T$-orbit precisely when $\nu\in JT$.  Suppose this is the case, that is, $\nu = Jx$ for some $x\in T$.  The determinant of $x$ is $N_{E/F}(x)$.  But since $N_{E/F}(E^\times) = F^\times$, we see that there must be a single split $T$-orbit in each of the two $G$-orbits.
\end{proof}

\subsection{$p$-adic fields}

In this section, we assume that $E/F$ is a tamely ramified degree $n$ extension of fields that are finite extensions of $\Q_p$, with $p\ne 2$.

According to Lemma \ref{muEF}  (or Lemma \ref{muprime}) and  Lemma \ref{yEFquad}, we have
$$\# \cO^T = y_{E/F}
=1+ \# \{ \text{quadratic extensions of $F$ contained in $E$}\}.
$$
Recall from \S \ref{sec:pureinner} the inventory of $G$-orbits of orthogonal involutions of $G$:
\begin{itemize}
\item  $\Theta_J$ is the $G$-orbit of $\theta_J$, 
\item
if $n$ is odd then $\Theta_{\rm nqs}$ is only other $G$-orbit besides $\Theta_J$,
\item if $n$ is even and greater than two, $\Theta_{\rm nqs}$ denotes the $G$-orbit consisting of  orthogonal involutions $\theta_\nu$ where $\nu$ not similar to $J$ but has the same discriminant,
\item  if $n$ is even there are three additional $G$-orbits of orthogonal involutions and they correspond to the three possible discriminants other than the discriminant of $J$.
\end{itemize}

Having described the $G$-orbits and enumerated the split $T$-orbits, we now describe  the split $T$-orbits in more detail and determine in which $G$-orbits they lie.  
The main result is:

\begin{proposition}
\label{TorbGorb}
The elements of  $\cO^T$ are described as follows:
\begin{itemize}
\item $y_{E/F}=1$ (equivalently, $n$ is odd):  The $T$-orbit of $\theta_J$ is the only element of $\cO^T$.
\item $y_{E/F}=2$: The set $\cO^T$ consists of the $T$-orbit of $\theta_J$ and the $T$-orbit of involutions $\theta_{Jx}$ where $x\in E^\times - ((E^\times)^2F^\times)$.  The two $T$-orbits lie in distinct $G$-orbits.  The discriminants of these orbits are distinct since  $\det x= N_{E/F}(x)\in N_{L/F}(L^\times)  - (F^\times)^2$, where $L$ is the unique quadratic extension of $F$ contained in $E$.
\item $y_{E/F}=4$ (equivalently, $Y_{E/F} = E^\times /(E^\times)^2$):  Three of the four $T$-orbits in $\cO^T$ lie in $\Theta_J$.   Let $\varpi_E$ be a prime element of $E$ such that $\varpi_E^{e(E/F)}$ lies in the maximal unramified extension of $F$ contained in $E$.  Let $u$ be a nonsquare unit in $E$ and let $\nu = Ju\varpi_E$.  Then  the $T$-orbit of $\theta_\nu$ lies in $\Theta_{\rm nqs}$.
\end{itemize}
\end{proposition}

\begin{proof}
When $y_{E/F}=1$ or $2$, our assertions follow directly from Lemma \ref{muprime}.  Therefore, we assume $y_{E/F}=4$.  A set of coset representatives for $Y_{E/F} = E^\times/ ((E^\times)^2F^\times)= E^\times /(E^\times)^2$ is given by $\{ 1,u,\varpi_E, u\varpi_E\}$, where $u$ is a nonsquare unit in $E$ and $\varpi_E$ is a prime element of $E$.

We will choose $u$ and $\varpi_E$ as follows:  $u$ will be a nonsquare root of unity in $E$ and $\varpi_E$ will be chosen as in Section \ref{sec:parahorics}.  We adopt the notations of Section \ref{sec:parahorics}, but, for simplicity, we omit the double underlne notations.  Note that $u$ and $\varpi_L = \varpi_E^e$ must lie in  the maximal unramified extension $L$ of $F$ contained in $E$, where $e = [E: L]$ and $f = [L:F]$.

Recall from Lemma \ref{normimage} that $N_{E/F}(E^\times)\subset (F^\times)^2$.  It follows that all matrices of the form $Jx$, with $x\in E^\times$, have the same discriminant as $J$.  This means that each split $T$-orbit must lie in either $\Theta_J$ or $\Theta_{\rm nqs}$.  The $G$-orbit of a given split $T$-orbit is therefore determined by its Hasse invariant.

Obviously, the $T$-orbit of $\theta_J$ lies in $\Theta_J$.  
According to Lemma 5.9 \cite{HL1}, the Hasse invariant of $Ju$ is trivial and hence $\theta_{Ju}$ also lies in $\Theta_J$.  It remains to determine the $G$-orbits containing $\theta_{J\varpi_E}$ and $\theta_{Ju\varpi_E}$.  It suffices to compute the Hasse invaraints of $J\varpi_E$ and $Ju\varpi_E$.  To give a uniform computation of both of these Hasse invariants, we suppose $w$ is either $1$ or $u$ and we compute the Hasse invariant of $Jw\varpi_E$.  

Let $\varpi_F$ be any prime element in $F$.  Since $y_{E/F}=4$, we may choose a square root $\sqrt{\varpi_F}$ in $E$.  Then there exists a unit $v$ in $E$ such that $\varpi_E^{e/2} = v \sqrt{\varpi_F}$.
Since $\varpi_L = v^2 \varpi_F$, we see that $v^2\in \gO_L^\times$.  In fact, $v^2\in (\gO_L^\times)^2$, since otherwise $L[v]$ would be an unramified quadratic extension of $L$ contained in the totally ramified extension $E/L$.  It follows that $v\in \gO^\times_L$.

Using the explicit form of $\varpi_E$ in Section \ref{sec:parahorics}, we see that
$$J_n\varpi_E = J_{(e-1)f}\oplus J_f \varpi_L.$$
Note that  since $w$ lies in $L$ it may be viewed as an element of $\M_f (F)$ and we have
$$J_nw\varpi_E = J_{(e-1)f}w\oplus J_f \varpi_L w.$$

The matrix $J_f \varpi_L w\, \varpi_F^{-1} = J_f wv^2$ has the same discriminant and Hasse invariant as $J_fw$.  Since $J_f \varpi_L w\, \varpi_F^{-1}$ must be similar to $J_fw$ and since $\varpi_F$ is a scalar matrix, we deduce that $J_f \varpi_Lw$ is similar to $J_f \varpi_Fw$.  

Comparing discriminants and Hasse invariants,  we see that  $J_{(e-1)f}w$ is similar to
$J_{(e-2)f}w\oplus J_f w$.
Likewise, $J_{(e-2)f}w$ and $I_{(e-2)f}$ are similar.  It follows that $Jw\varpi_E$ has the same Hasse invariant (and discriminant) as $J_f w\oplus J_f w\varpi_F$.  Since we are only interested in computing Hasse invariants, we may as well replace $Jw\varpi_E$ by $J_f w\oplus J_f w\varpi_F$.  Since we have essentially just reduced to the case in which $e=2$, we will, for simplicity, assume $e=2$.

There exist units $u_1,\dots , u_f$ in $F$ such that $J_f w$ is similar to the diagonal matrix with diagonal $(u_1,\dots , u_f)$.  Thus  
$J w\varpi_E$ is similar to the diagonal matrix with diagonal $$(u_1,\dots , u_f,u_1\varpi_F,\dots ,u_f \varpi_F).$$

The Hasse invariant of $Jw\varpi_E$ can now be computed using standard properties of the Hilbert symbol, such as the fact that $(a,b)=1$ when $a,b\in \gO_F^\times$ and $(ab,c) = (a,c) (b,c)$, when $a,b,c\in F^\times$.  Using the latter properties, we see that the Hasse invariant is
$$\prod_{i,j} (u_i,u_j\varpi_F)\prod_{i<j} (u_i\varpi_F,u_j\varpi_F).$$
The first product is 
$$\prod_{i,j} (u_i,u_j\varpi_F)=\prod_{i,j} (u_i,\varpi_F) = \left( \prod_i (u_i,\varpi_F)\right)^f =1.$$
Thus the Hasse invariant reduces to:
\begin{eqnarray*}\prod_{i<j} (u_i\varpi_F,u_j\varpi_F)
&=&\prod_{i<j} (u_i,\varpi_F)(u_j,\varpi_F) (\varpi_F,\varpi_F)\\
&=&\left(\prod_{i} (u_i,\varpi_F)\right)^{f-1}(\varpi_F,\varpi_F)^{(f/2)(f-1)}\\
&=&  (\det (J_f w),\varpi_F) (\varpi_F,\varpi_F)^{f/2}\\
&=&  ((-1)^{f/2}N_{L/F}(w),\varpi_F) (\varpi_F,\varpi_F)^{f/2}\\
&=&  ((-\varpi_F)^{f/2}N_{L/F}(w),\varpi_F).
\end{eqnarray*}
We observe that $N_{L/F}$ determines an isomorphism $\gO^\times_L /(\gO^\times_L)^2\cong \gO^\times_F /(\gO^\times_F)^2$.
Thus $N_{L/F}(u)$ represents the nontrivial coset in $\gO^\times_F /(\gO^\times_F)^2$.

Suppose $f/2$ is even.  Then the Hasse invariant reduces to  the Hilbert symbol $(N_{L/F}(w),\varpi_F)$.    If $w=1$ this is trivial.  Suppose $w=u$.  The Hilbert symbol involves solutions to 
$$N_{L/F}(u) x^2 + \varpi_F y^2 = z^2.$$  It suffices to consider solutions with $x,y,z\in \gO_F$ such that $x$, $y$ and $z$ do not all lie in $\gP_F$.  Reducing modulo $\gP_F$ yields $$N_{L/F}(u) x^2\equiv z^2,$$ where $x$ and $z$ do not both lie in $\gP_F$.  Since $N_{L/F}(u)$ is a nonsquare, we deduce that the Hasse invariant is $-1$ when $f/2$ is even and $w=u$.

Suppose $f/2$ is odd.  The Hasse invariant equals the Hilbert symbol $(-\varpi_F N_{L/F}(w),\varpi_F)$.  Again, when $w=1$ the Hilbert symbol is trivial and so we assume $w=u$.  The Hilbert symbol involves solutions to
$$-\varpi_F N_{L/F}(u) x^2 + \varpi_F y^2 = z^2.$$
The latter equation has solutions precisely when
$$-N_{L/F}(u) x^2 +  y^2 = \varpi_F z^2.$$
Reducing modulo $\gP_F$ as before, we deduce that the Hasse invariant is $-1$.  Our assertions now follow.
\end{proof}

\subsection{The number of split $T$-orbits in a $K^0$-orbit}

Let $F$ be a finite extension of $\Q_p$, where $p$ is an odd prime.  Let $E_0$ be tamely ramified finite extension of $F$ of degree $m_0$ and let $E$ be an unramified extension of $E_0$ of degree $n_0= n /m_0$.  Embed $E$ in $\M_{n_0} (E_0)$ in a $J_{n_0}$-symmetric way and embed $E_0$ in $\M_{m_0} (F)$ in a $J_{m_0}$-symmetric way.  
This gives a $J_n$-symmetric embedding of $E$ in $\M_n (F)$.  Let $G = \GL_n (F)$ and $T = E^\times$.  Let $K^0 = E_0^\times \GL_{n_0}(\gO_{E_0})$.

\begin{lemma}\label{nofusion}
Every $K^0$-orbit of orthogonal
involutions of $G$ contains at most one split $T$-orbit.
\end{lemma}

\begin{proof}
Suppose $\zeta$ and $\zeta'$ are two distinct split $T$-orbits that lie in a common $K^0$-orbit.  Then $\zeta$ and $\zeta'$ must lie in a common $G$-orbit $\Theta$, but, according to
Proposition \ref{TorbGorb}, this implies $y_{E/F} =4$ and $\Theta = \Theta_J$.

There are several cases to consider and the arguments we use to handle them are by no means uniform.  The $G$-orbit $\Theta_J$ contains three split $T$-orbits.  Representatives for these orbits are:
\begin{itemize}
\item $\theta_J$, where $J = J_n$,
\item $\theta_{Ju}$, where $u$ is a nonsquare unit in $E$,
\item $\theta_{J\varpi}$, where $\varpi$ is a prime element of $E_0$ such that $\varpi^{e(E/F)}$ lies in the maximal unramified extension of $F$ contained in $E_0$.
\end{itemize}

Suppose $v=1$ or $u$.  We now show that $\theta_{Jv}$ and $\theta_{J\varpi}$ lie in distinct $K^0$-orbits.  Suppose, to the contrary, that there exists $k\in K^0$ such that $k\cdot \theta_{Jv} = \theta_{J\varpi}$.  This is equivalent $$\theta_{{}^t k^{-1}Jvk^{-1}} = \theta_{J\varpi}$$ which is, in turn, equivalent to the condition
$J\, \varpi\,  k\, v^{-1}\, J\, {}^tk\in F^\times$.
Since $y_{E/F}=4$, it must be the case that $e(E_0/F)$ is even.  
Therefore, $F^\times \subset (E_0^\times)^2\gO_E^\times$.
Therefore, we can choose $a\in E_0^\times$ and $b\in \gO_E^\times$ such that 
$$a^2b = J\, \varpi\,  k\, v^{-1}\, J\, {}^tk 
 .$$   Next, we write $k = hy$, with $h\in \GL_{n_0}(\gO_{E_0})$ and $y\in E_0^\times$.  This yields
 $$a^2b = J\, \varpi\,  h\, y\, v^{-1}\, J\, {}^ty\ {}^t h =  
 J\, \varpi\,  h\, y^2\, v^{-1}\, J\, {}^t h 
 .$$  Taking determinants and absolute values in the latter equations, we see that the elements
 $\varpi$, $a$ and $y$ of $E_0$ are related by the fact that $\varpi y^2 a^{-2}$ is a unit.  But this is impossible since the latter element clearly has odd valuation.  This contradiction proves that 
 $\theta_{Jv}$ and $\theta_{J\varpi}$ lie in distinct $K^0$-orbits, if $v\in \{ 1,u\}$.
 
 It remains to show that $\theta_J$ and $\theta_{Ju}$ lie in distinct $K^0$-orbits.  Assume first that $n_0$ is even.  Suppose there exists $k\in K^0$ such that $k\cdot \theta_J = \theta_{Ju}$.  We use
 Proposition \ref{symrestrict} to describe the restrictions of our involutions to $G^0 = \GL_{n_0}(E_0)$.
We have
$$\theta_{J_{n_0}u} = \theta_{Ju}|G^0 = k\cdot \theta_J |G^0 = 
\theta_{{}^t k^{-1}\, J\, k^{-1}} |G^0 = \theta_{{}^Tk^{-1}\, J_{n_0}\, k^{-1}},
$$  where $h\mapsto {}^Th$ is the transpose on $G^0$.  This contradicts the fact that  the symmetric  matrices $J_{n_0}u$ and ${}^Tk^{-1}\, J_{n_0}\, k^{-1}$ (in $G^0$) have distinct discriminants (and $n_0$ is even).

Now assume $n_0$ is odd.
In this case, we can choose $u$ to be a nonsquare unit in $E_0$.
Suppose $k\in K^0$ and $\theta_{Ju} = k\cdot \theta_J = \theta_{{}^t k^{-1}\, J\, k^{-1}}$.  
Then $J\, u\, k\, J\, {}^t k\in F^\times$.  Equivalently, $u\, k\, J\, {}^tk\, J \in F^\times$.
Write $k = hy$, with $h\in \GL_{n_0}(\gO_{E_0})$ and $y\in E_0^\times$.  Then we have
$u\, h\, y\, J\, {}^ty\, {}^t h\, J =
u\, h\, y^2\, J\, {}^t h\, J= u\, y^2\, h\, \theta_J(h)^{-1}\in F^\times$.
(Note that $\theta_{J_n}(h) = \theta_{J_{n_0}}(h)$.)

Let $E'_0$ be the maximal unramified extension of $F$ contained in $E_0$.  Then $E_0 = E'_0 [\root e\of{\alpha\varpi_F}]$, for some 
$\alpha\in \gO_{E'_0}^\times$ and some prime element $\varpi_F$ of $F$.  Here, $e = e (E/F) = e(E_0/F)$.  Let $E'$ be the maximal unramified extension of $F$ contained in $E$.  Then $E' [\root e\of{\alpha\varpi_F}] = E'E_0 = E$.  Since $y_{E/F}=4$, 
Lemma \ref{fieldsB}  implies $\alpha\in (\gO_{E'}^\times)^2$.  But the fact that $n_0$ is odd implies, $\alpha \in (\gO_{E'}^\times)^2\cap \gO_{E'_0}^\times = (\gO_{E'_0}^\times)^2$.
Applying Lemma \ref{fieldsB} again allows us to deduce that $y_{E_0/F}=4$.  This implies $F^\times \subset (\gO_{E_0}^\times)^2$.  Therefore, $h\theta_J(h)^{-1} \in (uy^2)^{-1}(\gO_{E_0}^\times)^2 \subset \gO_{E_0}^\times - (\gO_{E_0}^\times)^2$.
Let $\bar h$ be the image of $h$ in $\GL_{n_0}(\f_{E_0})$.
Then $\bar h$ lies in the similitude group ${\rm GO}_{J_{n_0}} (\f_{E_0})$ and the similitude ratio of $\bar h$ is a nonsquare in $\f_{E_0}$.  This is impossible, according to Lemma \ref{mTfinite}.  Indeed, the latter result implies that 
$${\rm GO}_{J_{n_0}} (\f_{E_0}) = \f_{E_0}^\times {\rm O}_{J_{n_0}} (\f_{E_0})$$ which implies that the homomorphism
$${\rm GO}_{J_{n_0}} (\f_{E_0})/\f_{E_0}^\times \to \f_{E_0}^\times /(\f_{E_0}^\times)^2$$ induced by the similitude ratio
must be trivial since  its kernel is
$\O_{J_{n_0}}/\{ \pm1\}$.
This contradiction completes the proof.
\end{proof}

The proof of the previous lemma involved the special case of the $G$-orbit $\Theta_J$ when $y_{E/F}=4$, because this is the unique case in which multiple split $T$-orbits lie in a common $G$-orbit.  We remark that, in the notations of the proof, the involutions $\theta_J$ and $\theta_{J\varpi}$ have identical restrictions to $G^0$, since $\varpi\in E_0^\times$.  
Moreover, when $n_0$ is odd the involutions in all three involutions
$\theta_J$, $\theta_{Ju}$ and $\theta_{J\varpi}$ have identical restrictions to $G^0$, since $u,\varpi\in E_0^\times$.   Indeed, according to 
Proposition \ref{symrestrict}, the restriction of $\theta_{J_n x}$ to $G^0$ is $\theta_{J_{n_0}x}$ when $x\in E^\times$.  But when $x\in E_0^\times$, we have $\theta_{J_{n_0}x} = \theta_{J_{n_0}}$.

\section{The character $\eta'_\theta$}\label{sec:etaprimetheta}

Suppose $\xi$ is a refactorization class of $G$-data.  Let $K^0 = G^0_{[y]}$ be associated to $\xi$ and suppose $\vartheta$ is a $K^0$-orbit of orthogonal involutions such that $\vartheta \sim \xi$.  (The notation $\vartheta\sim \xi$ is defined in \cite{HL1} and recalled in \S\ref{sec:general}.)  A character $\eta'_\theta : K^{0,\theta}\to \{ \pm 1\}$ is defined in \cite{HM} and slightly reformulated in \cite{HL1}.  In this section, we compute the values of $\eta'_\theta$ we need.

Let us introduce the notations required to state the definition of $\eta'_\theta$.  
The groups we use are all defined with respect to any $\Psi\in \xi$ and $\theta\in \vartheta$.  The choices of $\theta$ and $\Psi$ are irrelevant.  

In Definition \ref{cuspidalGdatum}, there is a $(d+1)$-tuple $\vec r = (r_0 ,\dots , r_d)$.  Given an index $i\in \{ 0,\dots ,d-1\}$, let $s_i = r_i/2$.  The space $ \g^{i+1,d\theta}_{y,s_i:s_i^+}=
 \g^{i+1,d\theta}_{y,s_i}/ \g^{i+1,d\theta}_{y,s_i^+}$
 has a direct sum decomposition 
 $$ \g^{i+1,d\theta}_{y,s_i:s_i^+}= 
 \g^{i,d\theta}_{y,s_i:s_i^+}\oplus
\fW^+_i  ,$$
  where the second summand $\fW^+_i$ is the unique complement of the first summand that is contained in the sum of the root spaces attached to roots in $\Phi (\bG^{i+1},\bT)-\Phi (\bG^i,\bT)$.

Let $\f$ be the residue field of $F$ and let $\f^*\cong \F_p$ denote the prime subfield of $\f$.  We note that $\fW^+_i $ has the structure of an $\f$-vector space.  Moreover, $K^{0,\theta}$ acts by conjugation on $\fW^+_i$ and in this way each element $k\in K^{0,\theta}$ defines an $\f$-linear transformation $\Ad (k)$ of $\fW^+_i $.

Given $k\in K^{0,\theta}$, we define
$$\eta'_\theta (k) = \prod_{i=0}^{d-1} \left( N_{\f/\f^*} \left(\det\nolimits_\f (\Ad (k)|\fW^+_i)\right)\right)^{(p-1)/2}.$$

\begin{lemma}\label{etaprimeidentity}
If $k\in K^{0,\theta} = G^{0,\theta}_{y,0}$ has image in $(\sG_y^{0,\theta})^\circ(\f)$ then 
$$\det\nolimits_\f (\Ad (k)|\fW^+_i)=1$$
for all $i\in \{ 0,\dots ,d-1\}$ and, consequently,
$\eta'_\theta (k)=1$.
\end{lemma}

\begin{proof}
Given an index $j$, let $\mathsf{H}^j$ denote the $\f$-group $(\sG_y^{j,\theta})^\circ$.
We are interested in the action of $\mathsf{H}^0 (\f)$ on a subquotient of $\g^{i+1,\theta}_{y,s_i}$.  (For simplicity, we sometimes abbreviate $d\theta$ as $\theta$ when there is no possibility for confusion.)

Our definition of $\fW_i^+$ is a variant of the definition in \cite{HL1}.  We may regard $\fW^+_i $ as an object that has ``depth $s_i$.''   It will be convenient for us to shift $\fW^+_i $ so that it becomes an object of depth 0.
Given an index $j$ and a real number $s$, we have $\g^{j}_{y,s} = \varpi_E^{s'}\g^{j}_{y,0}$, where $s' = \lceil s e(E/F)\rceil$.

Let us consider the fixed points of $d\theta$ in $\g^j_{y,s}$.
Suppose $X = \varpi_E^{s'}Y$ is an element of $\g^j_{y,s}$.  Then 
\begin{eqnarray*}
d\theta (X) = X
&\Leftrightarrow&\varpi_E^{s'}Y = (d\theta (Y))\ \sigma (\varpi_E^{s'})\\
&\Leftrightarrow&Y = \varpi_E^{-s'}\ (d\theta (Y))\ \sigma (\varpi_E^{s'})\\
&\Leftrightarrow&Y = d\theta_s (Y),
\end{eqnarray*}
where $$d\theta_s (Y) = \varpi_E^{-s'}\ (d\theta (Y))\ \sigma (\varpi_E^{s'}).$$
So $Y\mapsto \varpi_E^{s'}$ determines an abelian group isomorphism from $\g_{y,0}^{j,\theta_s}$ onto $\g^{j,\theta}_{y,s}$.
This isomorphism restricts to an isomorphism of  $\g_{y,0^+}^{j,\theta_s}$ with $\g_{y,s^+}^{j,\theta}$.  We also obtain an $\f$-linear isomorphism of $\g^{j,\theta_s}_{y,0:0^+}$ with $\g^{j,\theta}_{y,s:s^+}$.

Let $\bfr{h}^{j,s} = {\rm Lie} (\sG_{y}^{j,\theta_s})$.  This is  defined over $\f$.  We observe that $\g^{j,\theta_s}_{y,0:0^+} = \h^{j,s} = \bfr{h}^{j,s}(\f)$.

Define an automorphism (of exponent $e(E/F)$) of $\bG$ by $\alpha_s (g) = \varpi_E^{-s'}g \varpi_E^{s'}$.  This yields an automorphism $\bar\alpha_s$ of $\mathsf{H}^j$.  Note that $\alpha_s$ is the identity map on $\bT$ and $\bar\alpha_s$ is the identity on $\mathsf{T}$.

Define an action of $\mathsf{H}^j$ on $\bfr{h}^{j,s}$ by $h\cdot_s X = \alpha_s (h) X h^{-1}$.

Now let $\chi_{j,s}$ be the character $\mathsf{H}^j\to \GL_1/\f$ that sends $h$ to the determinant of the linear transformation $X\mapsto h\cdot_s X$ of $\bfr{h}^{j,s}$.
This character is certainly trivial on the derived group of $\mathsf{H}^j$.  It is also trivial on the center (because $\alpha_s$ is the identity on the center).  So $\chi_{j,s}$ is trivial.

We now write $\bfr{h}^{i+1,s_i}$ as a direct sum of $\bfr{h}^{i,s_i}$ and the unique  complement $(\bfr{h}^{i,s_i})^\perp$ that is contained in the sum of the root spaces associated to roots in $\Phi^{i+1}-\Phi^i$.  The decomposition
$$\bfr{h}^{i+1,s_i} =\bfr{h}^{i,s_i} \oplus (\bfr{h}^{i,s_i})^\perp$$ is defined over $\f_F$ even though the root spaces are not.

Now suppose $h\in \mathsf{H}^i$.  On the one hand, we have $\chi_{i+1,s_i}(h)=1$.  But the action of $h$ preserves $\bfr{h}^{i,s_i}$ and $(\bfr{h}^{i,s_i})^\perp$.  Let $\chi'_{i+1,s_i}(h)$ and $\chi''_{i+1,s_i}(h)$ be the determinants of the linear transformations of the latter summands.

Then $\chi'_{i+1,s_i}(h) = \chi_{i,s_i^+}(h) =1$ and thus  $$\chi''_{i+1,s_i}(h) = \chi_{i+1,s_i}(h)\chi'_{i+1,s_i}(h)^{-1} = 1.$$

We  have now shown that $\chi''_{i+1,s_i}$ is trivial.  This implies that  $ (\sG_y^{0,\theta})^\circ (\f)$ acts in a unimodular fashion on $\fW^+_i\cong (\bfr{h}^{i,s_i})^\perp (\f)$.
\end{proof}

\section{A multiplicity formula parametrized by split $T$-orbits}\label{sec:Torbmultsec}

Let $\xi$ be a refactorization class of $G$-data and let $\Theta$ be a $G$-orbit of orthogonal involutions.  
In this section, we establish a formula for $\langle \Theta ,\xi\rangle_G$ that involves a sum whose summands are parametrized by the elements of the set $\cO^T(\Theta)$ of  split $T$-orbits in $\Theta$.

If $\zeta\in \cO^T(\Theta)$ and $\vartheta$ is the $K^0$-orbit containing $\zeta$, we define 
$$\langle \zeta, \xi\rangle_T = \langle \vartheta ,\xi\rangle_{K^0}$$
and $$m_T(\zeta) = [G_\theta : (T\cap G_\theta)G^\theta],$$ where $\theta$ is any element of $\zeta$.  With these notations, our formula is stated as follows:

\begin{proposition}\label{newmultformula}
$$\langle \Theta ,\xi\rangle_G =\sum_{\zeta \in \cO^T(\Theta)}  m_{T} (\zeta )\  \langle \zeta ,\xi\rangle_{T}.$$
\end{proposition}

A key ingredient in the proof
 of Proposition \ref{newmultformula} is Proposition \ref{uniqueTinKzero}, which is stated and proved below.  Taking  Proposition \ref{uniqueTinKzero} for granted, we can now prove 
 Proposition \ref{newmultformula}.

\begin{proof}
We start by recalling the formula $$\langle \Theta ,\xi\rangle_G =\sum_{\vartheta\sim \xi }  m_{K^0} (\vartheta )\  \langle \vartheta ,\xi\rangle_{K^0}.$$
from \cite{HL1} which  is explained above in \S\ref{sec:general}.  If $\vartheta$ is any $K^0$-orbit in $\Theta$ then, according to Proposition 3.9 \cite{HL1} and Propositions 5.9 and 5.20 \cite{HM}, the condition $\langle \vartheta,\xi\rangle_{K^0}\ne 0$ implies $\vartheta\sim \xi$.  So the previous sum may be regarded as a sum over all $K^0$-orbits $\vartheta$ in $\Theta$, not just the orbits that satisfy $\vartheta\sim \xi$.

Proposition \ref{uniqueTinKzero} tells us that if $\vartheta$ is a $K^0$-orbit contained in $\Theta$ such that $\langle \vartheta ,\xi\rangle_{K^0}$ is nonzero then $\vartheta$ contains a unique split $T$-orbit $\zeta$.
This yields the formula
$$\langle \Theta ,\xi\rangle_G =\sum_{\zeta \in \cO^T(\Theta)}  m_{K^0} (\zeta )\  \langle \zeta ,\xi\rangle_{T}.$$
To complete the proof, we observe that case-by-case analysis, given in Proposition \ref{indices} and Lemma \ref{indicestwo}, shows that if $\zeta \subset \xi$ then $m_{K^0}(\vartheta)  = m_T (\zeta)$.
\end{proof}

\subsection{Reduction to $\theta$-split tori over the residue field}
\label{sec:redtores}

Let $\sG$ be a connected reductive group over $\F_q$, where $q$ is a power of an odd prime.  Let $\sT$ be a maximal torus that is defined over $\F_q$.  Let $F$ be the Frobenius automorphism that defines the $\F_q$-structures.  Then $\sG^F = \sG (\F_q)$ and $\sT^F = \sT (\F_q)$.

Fix an $\F_q$-automorphism $\theta$ of $\sG$ of order two and let $\sG^\theta$ be the group of fixed points of $\theta$.  Let $\sT^\theta = \sT\cap \sG^\theta$ and $\sT_+ = (\sT^\theta)^\circ$.
When $\mathsf{H}$ is an $\F_q$-subgroup of $\bG$, let ${\rm rank}_{\F_q}(\mathsf{H} )$ be the $\F_q$-rank of $\mathsf{H}$.
Let $\mathsf{M}$ be the centralizer of $\sT_+$ in $\sG$.  
If $g\in \sG$ let $Z_\sG^\circ (g)$ denote the identity component of the centralizer of $g$ in $\sG$.

As in \cite{L}, we define a character
 $$\epsilon_\sT : (\sT^\theta)^F /\sT^F_+\to \{ \pm 1\}$$
by
$$\epsilon_\sT (t) = (-1)^{{\rm rank}_{\F_q} (\mathsf{M}) + {\rm rank}_{\F_q} (\mathsf{M} \cap Z_\sG^\circ (t) )}.$$
When $\sT$ is $\theta$-split this reduces to
$$\epsilon_\sT (t) = (-1)^{{\rm rank}_{\F_q} (\sG ) +{\rm  rank}_{\F_q} ( Z_\sG^\circ (t) )}.$$

A character $\lambda$ of $\sT^F$ is said to be {\it nonsingular} if  it  is not orthogonal to any coroot of $\sT$.  (This is Definition 5.15(i) \cite{DL}.)  In other words, if $h: \GL_1 \to \sT$ is a coroot that is defined over $\F_{q^n}$, for some $n\ge 1$,  then $$(\lambda\circ N)| h(\F_{q^n}^\times )\ne 1,$$ where $$N(t) = t\cdot F(t)\cdots F^{n-1}(t),$$ for $t\in \sT (\F_{q^n})$.
To say that $h$ is defined over $\F_{q^n}$ means $$h(\F_{q^n}^\times)\subset \sT (\F_{q^n})$$ or, equivalently, $F^n h(x) = h(x^{q^n})$, for all $x\in \F^\times_{q^n}$.  (See 10.2(a) \cite{L}.)

If $\lambda$ is a character of $\sT^F$ and $\chi$ is a character of $(\sG^\theta)^F$, we define $\Xi_{\sT,\lambda,\chi}$ be  the set of all $\gamma \in \sG^F$ such that $(\gamma \cdot \theta)(\sT) = \sT$ and 
$$\lambda (t) = \chi (\gamma^{-1}t\gamma ) \ \epsilon_{\gamma^{-1}\sT \gamma } (\gamma^{-1}t\gamma),$$ for all $t\in (\sT^{\gamma \cdot \theta})^F$.

Let $\cJ$ denote the (possibly empty) set of maximal tori $\mathsf{S}$ of $\sG$ for which there exists a Borel subgroup $\mathsf{B}$ such that $\mathsf{S} = \mathsf{B}\cap \theta (\mathsf{B})$.
These tori were studied by Vust \cite{Vu} and it follows from his results that if $\cJ$ contains a $\theta$-split maximal torus (as will be the case for our examples) then $\cJ$ is identical to the set of all $\theta$-split maximal tori in $\sG$.
 If $\mathsf{S}$ is a $\theta$-stable maximal torus then (according to 10.1(a) \cite{L}) $\mathsf{S} \in \cJ$ if and only if $\theta$ does not fix any coroots (or, equivalently, roots) of $\mathsf{S}$.

The following generalizes Lemma 10.4 \cite{L}:

\begin{lemma}
\label{tenfour}
If $\lambda$ is a nonsingular character of $\sT^F$ and $\chi$ is a character of $(\sG^\theta)^F$ that is trivial on $((\sG^\theta)^\circ)^F$ then $\gamma\in \Xi_{\sT ,\lambda,\chi}$ implies $\gamma^{-1}\sT \gamma \in \cJ$.
\end{lemma}

\begin{proof}
If $\gamma\in \sG^F$ then $\gamma \sT \gamma^{-1}$, like $\sT$, is a maximal torus in $\sG$ that is defined over $\F_q$ and the character ${}^\gamma\lambda$  of $\gamma \sT^F \gamma^{-1}$ defined by $({}^\gamma \lambda)(\gamma t\gamma^{-1}) = \lambda (t)$, for  $t\in \sT^F$,   is nonsingular.
Since $\gamma \in \Xi_{\sT,\lambda ,\chi}$ if and only if $1\in \Xi_{\gamma^{-1}\sT \gamma ,{}^{\gamma^{-1}}\lambda ,\chi}$, it suffices to prove our assertion when $\gamma =1$.

Suppose $1\in \Xi_{\sT,\lambda,\chi}$.  This means $\sT$ is $\theta$-stable and $\lambda (t)  = \chi (t)\epsilon_{\sT}(t)$, for all $t\in (\sT^\theta)^F$.  We claim that nonsingularity implies $\sT\in \cJ$.  Suppose not.  Then $\theta$ must fix a coroot $h$ of $\sT$.  The image of $h$ is a connected group contained in $\sT^\theta$, hence it is contained in $\sT_+$.  If $h$ is defined over $\F_{q^n}$ then $h (\F^\times_{q^n})\subset \sT_+ (\F_{q^n})$.  Thus $N(h(\F_{q^n}^\times ))\subset N(\sT_+ (\F_{q^n}))\subset  \sT_+^F\subset \ker (\chi \epsilon_\sT |(\sT^\theta)^F)\subset \ker \lambda$. 
But this contradicts nonsingularity.
\end{proof}

\subsection{Lifting finite $\theta$-stable tori to $p$-adic $\theta$-stable tori}

The main result in this section is:

\begin{proposition}\label{liftingtori}
If $\theta$ is an orthogonal involution and $\Psi$ is a $\theta$-symmetric non-toral $G$-datum such that $\sT$ is $\theta$-stable then there exists $k\in G^0_{y,0^+}$ such that $k \bT k^{-1}$ is $\theta$-stable.
\end{proposition}

Our proof is an adaptation of a proof suggested to us by Jeffrey Adler and Joshua Lansky.  A similar approach is used in Appendix A \cite{HL1}.

\begin{lemma}\label{conjugate}
If $\theta$ is an orthogonal involution and $\Psi$ is a $\theta$-symmetric non-toral $G$-datum such that $\sT$ is $\theta$-stable then there exists $k\in G^0_{y,0^+}$ such that $k\bT k^{-1} = \theta (\bT)$.
\end{lemma}

\begin{proof}
Using field embeddings as in Lemma \ref{Jsymtameexistence},
we have $$G^0_{y,0^+} = 1+\M_{n_0}(\gP_{E_0}).$$
Let $\bH$ be the $E_0$-group $\GL_{n_0}$.  Then $\bH (E_0) = G^0$.   There exists a unique $E_0$-torus in $\bT_1$ in $\bH$ such that $\bT_1 (E_0) = T$.  Similarly, there exists a unique $E_0$-torus $\bT_2$ in $\bH$ such that $\bT_2 (E_0) = \theta (T)$.  There exists a point $x\in \cB (\bH ,E_0)$ such that $$\bH (E_0)_{x,0^+} = 1+\M_{n_0}(\gP_{E_0}).$$
Let  $\sH_x$ be the reductive $\f_{E_0}$-group such that $\sH_x (\f_{E_0}) = \bH (E_0)_{x,0:0^+}$.  We have  
$$\sG_y^0 (\f_F) = \GL_{n_0}(\f_{E_0}) = \sH_x (\f_{E_0}).$$
The tori $\bT_i$ correspond to tori $\sT_i$ in $\sH_x$ such that $$\sT_1 (\f_{E_0}) = \sT(\f_F) = \theta (\sT )(\f_F) = \sT_2(\f_{E_0}).$$

Lemma 2.2.2 \cite{D} applies to $\bH$, $E_0$, $x$, $\bT_1$ and $\bT_2$ and implies that there exists $k\in H_{x,0^+}$ such that $k\bT_1 k^{-1} = \bT_2$.  But $k$ lies in $G^0_{y,0^+}$ and  $k Tk^{-1} = \theta (T)$, which implies $k\bT k^{-1} = \theta (\bT)$.
\end{proof}

\begin{lemma}\label{abstractsetup}
Let $\cG_1\supset \cG_2\supset \ldots$ be a sequence of groups such that $\cG_{i+1}$ is normal in $\cG_i$ and $[\cG_i:\cG_{i+1}]$ is (finite and) odd for all $i$.  Let $\cS$ be a nonempty set on which  $\cG_1$ acts transitively such that $\cap_i\cO_i$ is a singleton set whenever $\cO_1\supset \cO_2\supset\ldots$ is such that $\cO_i$ is a $\cG_i$-orbit in $\cS$ for all $i$.  Let $\alpha$ be a permutation of $\cS$ of order two such that $\alpha (\cG_i\cdot x) = \cG_i\cdot\alpha(x)$, for all $i$ and for all $x\in \cS$.  Then $\alpha$ must have a fixed point in $\cS$.
\end{lemma}

\begin{proof}
We first describe a recursive process for choosing a sequence $\cO_1\supset \cO_2\supset\ldots$  as in the statement of the lemma, but with the added property that each $\cO_i$ is $\alpha$-stable.

Let $\cO_1 = \cS$.    To construct our sequence, we need to show that once  $\cO_i$ is defined, we can choose an $\alpha$-stable element $\cO_{i+1}$ in the set    $$\cS_i=\{ \text{$\cG_{i+1}$-orbits in $\cO_i$}\}.$$   Since $\cG_{i+1}$ is normal in $\cG_i$, it must be the case that $\cG_i$ acts transitively on $\cS_i$.  The cardinality of $\cS_i$ must divide $[\cG_i:\cG_{i+1}]$ and thus it must be odd.  Since $\alpha(\cG_{i+1}x) = \cG_{i+1} \alpha(x)$, for all $x\in \cO_i$, it must be the case that $\alpha$ defines a permutation of $\cS_i$.  Suppose $\cS_i$ does not contain  an $\alpha$-stable element. 
Then it can be partitioned into sets of order two of the form $\{ \cG_i \cdot x,\cG_i\cdot\alpha(x)\}$.  But this contradicts the fact that $\cS_i$ has odd cardinality.  So, in fact, we an choose an $\alpha$-fixed orbit $\cO_{i+1}$ in $\cS_i$.

Having established that a sequence $\cO_1\supset \cO_2\supset\ldots$ of the desired type may be chosen, we can consider $\cap_i \cO_i$.  This is a singleton set $\{ x\}$, where $x$ must be a fixed point of $\alpha$ in $\cS$.
\end{proof}

\begin{proof}[Proof of Proposition \ref{liftingtori}]
Apply Lemma \ref{abstractsetup} as follows.  
Take $\cG_1 = G^0_{y,0^+}$.  If $\cG_i$ is defined and $\cG_i = G^0_{y,a_i}\ne G^0_{y,a_i^+}$ then let $\cG_{i+1} = G^0_{y,a_i^+}$.
Take $\cS$ to be the $G^0_{y,0^+}$-orbit of $\bT$.  Since we have shown $\theta (\bT)\in \cS$, it follows that $\cS$ is $\theta$-stable.  So we can take $\alpha$ to be the permutation of $\cS$ associated to $\theta$.
\end{proof}

\subsection{Reduction to $\theta$-stable tori}

We now apply Lemma \ref{tenfour} and Proposition \ref{liftingtori} to prove a result that essentially says that every relevant $K^0$-orbit of orthogonal involutions contains at least one favorable $T$-orbit.

We will consider a fixed refactorization class $\xi$ of $G$-data.  Associated to $\xi$  is the reductive group $\sG_y^0$ defined over the residue field $\f$ of $F$ such that $\sG_y^0 (\f) = G^0_{y,0:0^+}$.   The role of the group $\sG$ in \S\ref{sec:redtores} is played by $\sG_y^0$.  The torus $\sT$ is the  $\f$-torus $\sT$ in $\sG_y^0$ such that $\sT(\f)$ is the image of $\bT (\gO_E)$ in $\sG_y^0 (\f)$.

\begin{lemma}\label{thetastableT}
Suppose $\xi$ is a refactorization class of $G$-data and $\vartheta$ is a $K^0$-orbit of orthogonal involutions of $G$ such that $\langle \vartheta ,\xi\rangle_{K^0}$ is nonzero.  Then there exists $\theta\in \vartheta$ such that $\bT$ is $\theta$-stable.  Given one such $\theta$, every involution in the $T$-orbit of $\theta$ must also have the same property.
\end{lemma}

\begin{proof}
Consider first the toral case.  As is explained in the proof of Proposition \ref{newmultformula}, the condition $\langle \vartheta ,\xi\rangle_{K^0}\ne 0$ implies that $\vartheta\sim \xi$.  Then Proposition 3.9 \cite{HL1} implies that we may choose $\Psi\in \xi$ such that $\Psi$ is $\theta$-symmetric for some $\theta\in \vartheta$. Since $\bG^0=\bT$ in the toral case, this implies that $\bT$ is $\theta$-stable.   If $\theta'$ is another element of the $T$-orbit of $\theta$, then $\bT$ must also be $\theta'$-stable.  Thus we have proved what is required in the toral case.

Now assume we are in the nontoral case.  Again, we can choose $\theta\in \vartheta$ and $\Psi\in \xi$ such that $\Psi$ is $\theta$-symmetric.  
Then, as in \S\ref{sec:ourexamples} and \S6.2 \cite{HL1}, we have the formula
$$\langle \vartheta,\xi\rangle_{K^0} = \dim \Hom_{\sG^{0,\theta}_y(\f)}((-1)^{n_0+1}R^\lambda_{\sT(\f)},\eta_\theta)$$
which can be evaluated using
Theorem 3.11 \cite{HL1}, a generalization of Theorem 3.3 \cite{L}.
We obtain
 $$\langle \vartheta ,\xi \rangle_{K^0} = (-1)^{n_0}\sum_{\gamma \in \mathsf{T}(\f)\bs \Xi_{\mathsf{T},\lambda,\eta_\theta}/\sG_y^0(\f)^\theta} \sigma \left(Z_{\sG_y^0}\left( (\gamma^{-1}\mathsf{T} \gamma \cap \sG_y^{0,\theta})^\circ\right)\right),$$ where $\Xi_{\mathsf{T},\lambda,\eta_\theta}$ is  the set of all $\gamma \in \sG_y^0(\f)$ such that $(\gamma \cdot \theta)(\mathsf{T}) = \mathsf{T}$, and
$$\lambda (t) = \eta_\theta (\gamma^{-1}t\gamma ) \ \varepsilon_{\gamma^{-1}\mathsf{T} \gamma } (\gamma^{-1}t\gamma),$$ for all $t\in \mathsf{T}(\f)^{\gamma \cdot \theta}$, and we use Lusztig's notation $\sigma (\mathsf{H}) = (-1)^{{\rm rank}_{\f}(\mathsf{H})}$, when $\mathsf{H}$ is an $\f$-subgroup of $\sG^0_y$. 

Since, by assumption, $\langle \vartheta ,\xi \rangle_{K^0}$ is nonzero,  $\Xi_{\mathsf{T},\lambda,\eta_\theta}$ must be nonempty. 
Choose $\gamma \in \Xi_{\mathsf{T},\lambda,\eta_\theta}$.  Then we can, and will, replace $\gamma\cdot \theta$ by $\theta$.  This causes  $1$ to lie in $\Xi_{\mathsf{T},\lambda,\eta_\theta}$, while it does not destroy our assumptions that $\theta\in \vartheta$ and  $\Psi$ is $\theta$-symmetric.  
Since $1\in \Xi_{\mathsf{T},\lambda,\eta_\theta}$, it must be the case that $\sT$ is $\theta$-stable.

Proposition \ref{liftingtori} now implies that there exists a $k\in G^0_{y,0^+}$ such that  $k\bT k^{-1}$ is $\theta$-stable or, equivalently, $\bT$ is $(k^{-1}\cdot \theta)$-stable.
This says that we can slightly alter $\bT$ so that it becomes nicer with respect to $\theta$.  Replacing $\theta$ by $k^{-1}\cdot \theta$, we obtain $\theta\in \vartheta$ such that $\bT$ is $\theta$-stable.   If $\theta'$ is another element of the $T$-orbit of $\theta$, then $\bT$ must also be $\theta'$-stable. 
\end{proof}

\subsection{Reduction to $\theta$-split tori}

We now establish the following $p$-adic analogue of Lemma \ref{tenfour}:

\begin{lemma}\label{thetasplitreduction}
Suppose $\xi$ is a refactorization class of $G$-data and $\vartheta$ is a $K^0$-orbit of orthogonal involutions of $G$ such that $\langle \vartheta ,\xi\rangle_{K^0}$ is nonzero.  Suppose 
$\Psi\in \xi$ and $\theta\in \vartheta$  are chosen  such that $\Psi$ is $\theta$-symmetric  $\bT$ is $\theta$-stable. Then  $\bT$ is $\theta$-split.  
 Given $\theta$ such that $\bT$ is $\theta$-split, every involution in the $T$-orbit of $\theta$ must also have the same property.
\end{lemma}

\begin{proof}
Fix a $G$-datum $\Psi$ and $\theta\in \vartheta$ such that $\Psi$ is $\theta$-symmetric and $\bT$ is $\theta$-stable.
Let $\sigma$ be the $F$-auto\-mor\-phism of $E$ such that $\sigma (x) = \theta(x)^{-1}$, for all $x\in E^\times =T$.  Our assertion is equivalent to the assertion that $\sigma$ must be the identity automorphism.  (Clearly, if  $\bT$ is $\theta$-split then it  is $\theta'$-split for all $\theta'$ in the $T$-orbit of $\theta$.)

Suppose $\sigma$ is nontrivial.  Then it must have order two or, in other words, $E$ must be a quadratic extension of
 the fixed field $E^\sigma$ of $\sigma$.  Note that  $E^\sigma$ contains $F$.

Consider our Howe factorization of $\varphi$:
$$\varphi = \varphi_{-1}\  \prod_{i=0}^d (\varphi_i \circ N_{E/E_i}) = \varphi_{-1}\ \prod_{i=0}^d (\phi_i |E^\times)$$ and let
$$\varphi' = \phi | E^\times = (\varphi_{-1})^{-1} \varphi = \prod_{i=0}^d (\phi_i |E^\times).$$
Since each $\phi_i$ is $\theta$-symmetric, so is $\varphi'$.

We have:
\begin{eqnarray*}
\varphi'\text{ is $\theta$-symmetric}
&\Leftrightarrow&
\varphi'(x) = \varphi' (\sigma (x))\text{ for all }x\in E^\times\\
&\Leftrightarrow&
\varphi'(x\sigma (x)^{-1})=1\text{ for all }x\in E^\times\\
&\Leftrightarrow&
\varphi' | \U_1(E/E^\sigma)=1\\
&\Leftrightarrow&
\varphi'\text{ factors through }N_{E/E^\sigma}\\
&\Rightarrow&\varphi'\text{ is not admissible over }F.
\end{eqnarray*}
If $\Psi$ is toral then $\varphi_{-1}=1$ and thus $\varphi' = \varphi$.  This means we have shown that $\varphi$ is not admissible over $F$.  This contradiction shows that $\sigma$ is indeed trivial in the toral case.

Now assume $\Psi$ is nontoral.  Then:
\begin{eqnarray*}
\varphi_{-1} | (1+ \gP_E)=1
&\Rightarrow&
\varphi'| (1+\gP_E) = \varphi | (1+\gP_E)\\
&\Rightarrow&
\varphi | (1+\gP_E)\text{ factors through }N_{E/E^\sigma}\\
&\Rightarrow&
 E/E^\sigma
 \text{ is unramified (by admissibility of $\varphi$)}.
\end{eqnarray*}
We therefore have $E = E^\sigma [\sqrt{\epsilon}]$, for some square root $\sqrt{\epsilon}$ of some nonsquare unit $\epsilon$ in $E^\sigma$.

Let $E_0^\sigma = E_0\cap E^\sigma$.  Suppose $\f_{E_0} = \f_{E_0^\sigma}$.  We claim that that $\theta$ must yield an orthogonal involution of $\sG_y^0(\f)$.  First, we observe that, according to Proposition \ref{symrestrict}, $\theta$ must restrict to either an orthogonal or unitary involution of $G^0 = \GL_{n_0}(E_0)$.
If $\theta |G^0$ is an orthogonal involution then $E_0 = E_0^\sigma$ and we apply Proposition  \ref{symembedB}.  This says that if $E$ is embedded in $\M_2 (E^\sigma)$ as the set of matrices $$\begin{pmatrix}x&y\epsilon\\ y&x\end{pmatrix}$$ with $x,y\in E^\sigma$ and if $E^\sigma$ is embedded in a $J$-symmetric way in $\M_{n_0/2}(E_0)$
then $\theta| G^0 = \theta_{\nu'}$ for some symmetric matrix $$\nu' = \begin{pmatrix}Jy &0\\ 0&-Jy\epsilon \end{pmatrix},$$ with $y\in (E^\sigma)^\times$.  Since $(E^\sigma)^\times = E_0^\times \gO_{E^\sigma}^\times$ and $E_0^\times$ is the center of $G^0$, we may assume $y$ lies in $\gO_{E^\sigma}^\times$.  This implies $\theta$ determines an orthogonal involution of $\sG_y^0 (\f) = \GL_{n_0} (\f_{E_0})$ in the present case.

On the other hand, if $\theta |G^0$ is a unitary involution then $E_0/E_0^\sigma$ is a ramified quadratic extension (since we are assuming $\f_{E_0} = \f_{E_0^\sigma}$) and we apply Proposition \ref{symembedC}.   Choosing embeddings as in Proposition \ref{symembedC} and arguing as in the previous case,
we see that $\theta |G^0$ is the unitary involution associated to some  $\nu'\in J\gO_{E^\sigma}^\times$.  The matrix $\nu'$ is a scalar multiple of a hermitian matrix in $G^0$.  It therefore reduces to either a symmetric or skew-symmetric matrix in $\GL_{n_0}(\f_{E_0})$.  The skew-symmetric case can be excluded by the results of Klyachko \cite{K}.

We have now shown that our assumption that $\f_{E_0}=\f_{E_0^\sigma}$ implies that $\theta$ yields an orthogonal involution of $\sG_y^0 (\f) = \GL_{n_0}(\f_{E_0})$.
Thus $\sG_y^{0,\theta}(\f)$ is an orthogonal group.  Lemma \ref{tenfour} then implies that  $\mathsf{T}$ is $\theta$-split and hence $$\sigma (\sqrt{\epsilon}) \equiv \sqrt{\epsilon}\quad\text{(mod $\gP_E$)}.$$
But $\sigma (\sqrt{\epsilon}) = -\sqrt{\epsilon}$ implies $-\sqrt{\epsilon}\equiv \sqrt{\epsilon}$ which is impossible.  It follows that $\f_{E_0}$ is a quadratic extension of $\f_{E_0^\sigma}$.  Equivalently, $E_0/E_0^\sigma$ must be an unramified quadratic extension.

Since $E/E_0$ is unramified and $E_0/E^\sigma_0$ is unramified, we now know that $E/E_0^\sigma$ is unramified.  Let $n_0 = [E:E_0] = [E^\sigma : E_0^\sigma ]$.  Suppose $n_0$ is even.  Then $E^\sigma$ contains an unramified quadratic extension of $E_0^\sigma$.  But since $E_0$ is the unique unramified quadratic extension of $E_0^\sigma$ in $E$, we deduce that $E_0$ is contained in $E_0^\sigma$, which is absurd.  This implies that $n_0$ must be odd.

Above we defined a quasicharacter $\varphi'$ and showed that $\varphi'$ factors through $N_{E/E^\sigma}$.  Since $\varphi = \varphi_{-1} \varphi'$, it suffices now to show that $\varphi_{-1}$ also factors through $N_{E/E^\sigma}$.  This would then imply that $\varphi$ factors through $N_{E/E^\sigma}$ contradicting the admissibility of $\varphi$ over $F$.

To show that $\varphi_{-1}$ factors through $N_{E/E^\sigma}$, we use our proof that $\eta'_\theta$ is trivial in the present case and then we use Proposition 4.2 \cite{HM2} (which uses results of Gow)  to  show that $\lambda$ is trivial on $\U_1(\f_E/\f_{E^\sigma})$ or, equivalently, that $\lambda$ is $\theta$-symmetric.  But since $\varphi_{-1}|(1+\gP_E)=1$ and $\varphi_{-1}$ (mod $1+\gP_E$) is $\theta$-symmetric, we see that $\varphi_{-1}$ is $\theta$-symmetric.  But this implies $\varphi_{-1}$ factors through $N_{E/E^\sigma}$ and we are done.
\end{proof}

\subsection{Reduction to orthogonal involutions on Levi subgroups and finite groups}

For such $\theta$ and $\bT$ it must be the case that $\bT$ is $\theta$-split.  
In particular, there exists $\theta\in \vartheta$ such that $\bT$ is $\theta$-split.  Given $\theta$ such that $\bT$ is $\theta$-split, every involution in the $T$-orbit of $\theta$ must also have the same property.

\begin{lemma}\label{restsareorthogonal}
Suppose $\xi$ is a refactorization class of $G$-data and $\vartheta$ is a $K^0$-orbit of orthogonal involutions of $G$ such that $\langle \vartheta ,\xi\rangle_{K^0}$ is nonzero. 
Then each $\theta\in \vartheta$ restricts to an orthogonal involution of each of the subgroups $G^i\cong \GL_{n_i}(E_i)$ associated to $\xi$ and, in addition, $\theta$ yields an orthogonal involution of $\sG_y^0(\f)\cong \GL_{n_0}(\f_{E_0})$.
\end{lemma}

\begin{proof}
According to Lemma \ref{thetasplitreduction}, it is possible to choose
$\theta\in \vartheta$ such that $\bT$ is $\theta$-split.  
Lemma \ref{symmetricembedding} then implies that there exists $x\in T$ such that $\theta = \theta_{Jx}$.  Since $T= E^\times = E_0^\times \gO_{E}^\times$, we can choose $z\in E_0^\times$ and $y\in \gO_E^\times$ such that $x=zy$.  Proposition \ref{symrestrict} implies that the restriction of $\theta_{Jx}$ to $G^0$ is $\theta_{J_{n_0}x}$.  Since $z$ is in the center of $G^0$, we have $\theta_{J_{n_0}x} =\theta_{J_{n_0}y}$.  But $J_{n_0}y$ is a symmetric matrix in $\GL_{n_0}(\gO_{E_0})$ and, consequently, $\theta$ yields an orthogonal involution of $\sG_y^0(\f) = \GL_{n_0} (\f_{E_0})$.
We have therefore established that there exists at least one $\theta\in \vartheta$ with the asserted properties.  But then it follows that every element of $\vartheta$ has these properties.
\end{proof}

\subsection{Reduction from $K^0$-orbits to split $T$-orbits}

Fix a refactorization class $\xi$ of $G$-data and a $K^0$-orbit $\vartheta$ of orthogonal involutions such that $\langle \vartheta , \xi\rangle_{K^0}\ne 0$.

For each $\theta\in \vartheta$, there is an associated involution $\bar\theta$ of $\sG_y^0(\f)$.  (In other sections in this paper, we abbreviate $\bar\theta$ as $\theta$.)   Let $\bar\vartheta$ be the image of $\vartheta$ under $\theta\mapsto \bar\theta$.

Since $T = E^\times  = E_0^\times \gO_E^\times$ and $E_0^\times$ is the center of $G^0$, we see that $\theta\mapsto \bar\theta$ maps $T$-orbits in $\vartheta$ to $\sT(\f)$-orbits in $\bar\vartheta$.
Since $K^0 = TG^0_{y,0} = G^0_{y,0}T$,
we see that $\bar\vartheta$ must be a single $\sG_y^0(\f)$-orbit.
 Therefore, we see that the map $\theta\mapsto \bar\theta$ yields a surjection from the set of $T$-orbits in the $K^0$-orbit $\vartheta$ onto the set of $\mathsf{T}(\f)$-orbits in the $\sG_y^0(\f)$-orbit $\bar\vartheta$.

\begin{proposition}\label{uniqueTinKzero}
If $\xi$ is a refactorization class of $G$-data and $\vartheta$ is a $K^0$-orbit of orthogonal involutions such that $\langle \vartheta , \xi\rangle_{K^0}\ne 0$ 
then \begin{itemize}
\item $\vartheta$ contains a unique split $T$-orbit $\zeta$,
\item  $\bar\vartheta$ contains a unique split $\sT (\f)$-orbit $\bar\zeta$, and 
\item $\bar\zeta$ is the image of $\zeta$ under $\theta\mapsto \bar\theta$.
\end{itemize}
\end{proposition}

\begin{proof}
The existence of a split $T$-orbit $\zeta$ in $\vartheta$ is a consequence of Lemma \ref{thetasplitreduction}.  The uniqueness of $\zeta$ follows from Lemma \ref{nofusion}.  Clearly, the image $\bar\zeta$ of $\zeta$ under $\theta\mapsto \bar\theta$ is a split $\sT(\f)$-orbit in $\bar\vartheta$.
 But, according to Lemma \ref{finiteuniqueness}, there must be a unique split $\sT(\f)$-orbit in $\bar\vartheta$.
\end{proof}

\section{Evaluating the formula}\label{sec:evalsec}

Let $\xi$ be a refactorization class of $G$-data and let $\Theta$ be a $G$-orbit of orthogonal involutions of $G$.  Suppose $\zeta\in \cO^T(\Theta)$.
Proposition \ref{newmultformula} establishes the formula
$$\langle \Theta ,\xi\rangle_G =\sum_{\zeta \in \cO^T(\Theta)}  m_{T} (\zeta )\  \langle \zeta ,\xi\rangle_{T}.$$
We now evaluate this formula.

\subsection{Evaluating $\langle \zeta ,\xi\rangle_T$}

Assume, with our usual notation, that $\Psi$ comes from an $F$-admissible quasicharacter $\varphi :E^\times \to \C^\times$.  

\begin{proposition}\label{sumofmTs}
If $\varphi (-1) \ne 1$ then $\langle \Theta, \xi\rangle_G = 0$.  If $\varphi (-1) =1$ then $$\langle \Theta ,\xi\rangle_G = \sum_{\zeta\in \cO^T(\Theta)} m_T (\zeta).$$
\end{proposition}

\begin{proof}
As observed in \S\ref{sec:general}, if $\varphi (-1)\ne 1$ then $\langle \Theta, \xi\rangle_G = 0$.  Therefore we assume $\varphi (-1) =1$.

For $\theta\in \zeta\in \cO^T(\Theta)$, we have
$$\langle \zeta ,\xi \rangle_{T} =\dim \Hom_{K^{0,\theta}}(\rho' ,\eta'_\theta).$$

In the toral case, $K^{0,\theta} = \{ \pm 1\}$ and $\rho' (-1) = \phi (-1) =\varphi (-1) =1 = \eta'_\theta (-1)$.  (Since $-1$ is in the center of $G$,  conjugation by $-1$ is trivial and hence $\eta'_\theta (-1)=1$.)  Our claim in the toral case follows.

Now assume $\xi$ is nontoral.  From Lemma \ref{restsareorthogonal}, we know that
 $\bG^{0,\theta}_y (\f)$ is an orthogonal group in $\sG_y^0(\f)= \GL_{n_0}(\f_{E_0})$.  We also know that $$\langle \zeta,\xi\rangle_T = \dim \Hom_{\sG_y^{0,\theta}(\f)}((-1)^{n_0+1}R^\lambda_{\mathsf{T}(\f)} ,\eta_\theta),$$
using the notations of  \S\ref{sec:ourexamples}.

We now apply Theorem 3.11 \cite{HL1}.  As in the proof of Lemma \ref{thetastableT}, we obtain the formula
$$\langle \zeta ,\xi \rangle_{T} = (-1)^{n_0}\sum_{\gamma \in \mathsf{T}(\f)\bs \Xi_{\mathsf{T},\lambda,\eta_\theta}/\sG_y^0(\f)^\theta} \sigma \left(Z_{\sG_y^0}\left( (\gamma^{-1}\mathsf{T} \gamma \cap \sG_y^{0,\theta})^\circ\right)\right).$$ Recall that $\Xi_{\mathsf{T},\lambda,\eta_\theta}$ is  the set of all $\gamma \in \sG_y^0(\f)$ such that $(\gamma \cdot \theta)(\mathsf{T}) = \mathsf{T}$ and
$$\lambda (t) = \eta_\theta (\gamma^{-1}t\gamma ) \ \varepsilon_{\gamma^{-1}\mathsf{T} \gamma } (\gamma^{-1}t\gamma),$$ for all $t\in \mathsf{T}(\f)^{\gamma \cdot \theta}$. 
The condition $(\gamma \cdot \theta)(\mathsf{T}) = \mathsf{T}$ simply means that $\gamma^{-1}\sT \gamma$ is $\theta$-stable.  But Lemma \ref{tenfour} then implies $\gamma^{-1}\sT \gamma$ is $\theta$-split or, equivalently, $\sT$ is $(\gamma\cdot \theta)$-split.    This implies
$$\sigma \left(Z_{\sG_y^0}\left( (\gamma^{-1}\mathsf{T} \gamma \cap \sG_y^{0,\theta})^\circ\right)\right) = \sigma (\sG_y^0 ) = (-1)^{n_0}.$$
So we have
$$\langle \zeta ,\xi \rangle_{T} = \sum_{\gamma \in \mathsf{T}(\f)\bs \Xi_{\mathsf{T},\lambda,\eta_\theta}/\sG_y^0(\f)^\theta} 1\quad  =\quad  \# (\mathsf{T}(\f)\bs \Xi_{\mathsf{T},\lambda,\eta_\theta}/\sG_y^0(\f)^\theta).$$

If $\gamma \in \Xi_{\mathsf{T},\lambda,\eta_\theta}$ then $\sT(\f)^{\gamma \cdot \theta} = \{ \pm 1\}$.  Therefore, $\Xi_{\mathsf{T},\lambda,\eta_\theta}$ is  the set of all 
$\gamma \in \sG_y^0(\f)$ such that $\mathsf{T}$ is $(\gamma\cdot \theta)$-split  and
$$\varphi (-1) = \lambda (-1) = \eta_\theta (-1) \ \varepsilon_{\gamma^{-1}\mathsf{T} \gamma } ( -1)=1.$$ Since we are assuming $\varphi (-1) =1$, we see that $\Xi_{\mathsf{T},\lambda,\eta_\theta}$ is simply the set of all 
$\gamma \in \sG_y^0(\f)$ such that $\mathsf{T}$ is $(\gamma\cdot \theta)$-split.

We now adopt the notations in
Proposition \ref{uniqueTinKzero}.  Consider the map from $\Xi_{\mathsf{T},\lambda,\eta_\theta}$ into the set of $\sT(\f)$-orbits of involutions in $\bar\vartheta$ 
given by mapping the double coset of $\gamma$ to the $\sT(\f)$-orbit of the involution $\gamma\cdot \bar\theta$.  According to 
Proposition \ref{uniqueTinKzero}, the image of this map must be the unique split $\sT(\f)$-orbit in $\bar\vartheta$ or, in other words, the orbit of $\bar\theta$.

So if $\gamma\in \Xi_{\mathsf{T},\lambda,\eta_\theta}$ then there exists $t\in \sT (\f)$ such that $\gamma\cdot \bar\theta = t\cdot \bar\theta$.  In other words, $t^{-1}\gamma$ lies 
in the orthogonal similitude group $\sG_y^0(\f)_{\bar\theta}$.
Now Lemma \ref{mTfinite} implies that $$\sG_y^0(\f)_{\bar\theta} = (\sT(\f) \cap \sG_y^0(\f)_{\bar\theta}) \sG_y^{0,\theta}.$$
It follows that $\gamma \in \mathsf{T}(\f)\bs \Xi_{\mathsf{T},\lambda,\eta_\theta}/\sG_y^0(\f)^\theta$ has a single element, namely, the double coset of the identity element.
Therefore, $\langle \zeta,\xi\rangle_T=1$ and our assertion follows.
\end{proof}

Note that in the previous proof, we essentially used a generalized version of Theorem 10.3 \cite{L}.  Our Lemma \ref{tenfour} provided the generalized version of Lusztig's Lemma 10.4 needed to generalize the  proof of Lusztig's theorem.

\subsection{Proof of Theorems \ref{mainone}, \ref{maintwo} and \ref{mainthree}}

We now evaluate the formula in
Proposition \ref{sumofmTs}
to obtain Theorems \ref{mainone}, \ref{maintwo} and \ref{mainthree}.  
We may as well assume $\varphi (-1) =1$.   Then Proposition \ref{sumofmTs} says $$\langle \Theta ,\xi\rangle_G = \sum_{\zeta\in \cO^T(\Theta)} m_T (\zeta).$$

Proposition \ref{indices} says that if $\zeta\in \cO^T(\Theta)$ then 
$m_T(\zeta)=1$, except when $y_{E/F}=2$ and $\zeta$ is the $T$-orbit of $\theta_J$.   

So the evaluation  of $\langle \Theta , \xi\rangle _G$ rests on the determination of $\cO^T(\Theta)$.  But the set $\cO^T(\Theta)$ is described by
Proposition
\ref{TorbGorb} which says that
we have the following cases:
\begin{itemize}
\item $y_{E/F}=1$ (equivalently, $n$ is odd):  There is one element of $\cO^T$, namely, the $T$-orbit of $\theta_J$,
\item $y_{E/F}=2$: There are two elements of  $\cO^T$:  the $T$-orbit of $\theta_J$ and the $T$-orbit of involutions $\theta_{Jx}$ where $x\in E^\times - ((E^\times)^2F^\times)$.  The two $T$-orbits lie in distinct $G$-orbits.
\item $y_{E/F}=4$: There are four elements of $\cO^T$: three of them  lie in $\Theta_J$ and the fourth lies in $\Theta_{\rm nqs}$.\end{itemize}

Suppose $\pi$ is a tame supercuspidal representation associated to a $G$-datum in $\xi$.  Then, as we are continuing to assume that $\varphi (-1)=1$, we see that  $\pi$ is $G^\theta$ with respect to some (hence all) $\theta\in \Theta$ precisely when $\cO^T(\Theta)$ is nonempty.  Since $\cO^T(\Theta_J)$ is always nonempty, Theorem \ref{mainone} follows.  Similarly, Theorem \ref{maintwo} reduces to the observation that if $\Theta\ne \Theta_J$ then $\cO^T(\Theta)$ is nonempty precisely in the two cases listed in the statement of Theorem \ref{maintwo}.

It remains to verify Theorem
 \ref{mainthree}.  Assume $\langle \Theta ,\xi\rangle_G$ is nonzero.
We have the following cases:
\begin{enumerate}
\item $\Theta = \Theta_J$ and $n$ odd:  
\begin{itemize}
\item There is a unique $\zeta$ in $\cO^T(\Theta)$ and  $m_T (\zeta)=1$.  Thus $\langle \Theta_J,\xi\rangle_G = 1$.
\end{itemize}
\item $\Theta = \Theta_J$ and $y_{E/F}=2$:  
\begin{itemize}
\item There is a unique $\zeta$ in $\cO^T(\Theta)$ and  $m_T (\zeta)=2$.  Thus $\langle \Theta_J,\xi\rangle_G = 2$.
\end{itemize}
\item $\Theta \ne \Theta_J$ and $y_{E/F}=2$:  
\begin{itemize}
\item There is a unique $\zeta$ in $\cO^T(\Theta)$ and  $m_T (\zeta)=1$.  Thus $\langle \Theta_J,\xi\rangle_G = 1$.
\end{itemize}
\item $\Theta = \Theta_J$ and $y_{E/F}=4$:  
\begin{itemize}
\item There are three $T$-orbits $\zeta$ in $\cO^T(\Theta)$ and we have $m_T (\zeta)=1$ in each case.  Thus $\langle \Theta_J,\xi\rangle_G = 3$.
\end{itemize}
\item $\Theta \ne \Theta_J$ and $y_{E/F}=4$:  
\begin{itemize}
\item There is a unique $\zeta$ in $\cO^T(\Theta)$ and  $m_T (\zeta)=1$.  Thus $\langle \Theta_J,\xi\rangle_G = 1$.
\end{itemize}
\end{enumerate}
This completes the proof of Theorem \ref{mainthree}.\hfill$\square$

\bibliographystyle{amsalpha}


%
%

\end{document}